\documentclass[11pt, twoside, leqno]{article}

\usepackage{amssymb}
\usepackage{amsmath}
\usepackage{amsthm}
\usepackage{color}
\usepackage{mathrsfs}
\usepackage{txfonts}

\usepackage{indentfirst}

\allowdisplaybreaks

\def\rr{{\mathbb R}}
\def\rn{{\mathbb{R}^n}}
\def\zz{{\mathbb Z}}

\def\nn{{\mathbb N}}

\def\cp{{\mathcal P}}
\def\cs{{\mathcal S}}

 \def\CO{{\mathcal O}}

\def\fz{\infty }
\def\az{\alpha}
\def\bz{\beta}
\def\dz{\delta}
\def\lz{\lambda}
\def\pa{\partial}
\def\f{\frac}

\def\boz{{\Omega}}
\def\sa{\sigma}

\def\lf{\left}
\def\r{\right}
\def\hs{\hspace{0.25cm}}
\def\ls{\lesssim}
\def\gs{\gtrsim}
\def\noz{\nonumber}
\def\wz{\widetilde}
\def\st{\subset}
\def\com{\complement}

\def\lg{\langle}
\def\rg{\rangle}

\def\loc{{\mathop\mathrm{\,loc\,}}}
\def\supp{\mathop\mathrm{\,supp\,}}
\def\esup{\mathop\mathrm{\,ess\,sup\,}}

\def\va{\vec{a}}
\def\vp{\vec{p}}
\def\HL{M_{{\rm HL}}}
\def\vh{{H_{\va}^{\vp}(\rn)}}
\def\vAh{{H_{A}^{p}(\rn)}}
\def\vah{{H_{\va}^{\vp,\,r,\,s}(\rn)}}
\def\vahfz{{H_{\va}^{\vp,\,\fz,\,s}(\rn)}}
\def\vfah{{H_{\va,\,{\rm fin}}^{\vp,\,r,\,s}(\rn)}}
\def\vfahfz{{H_{\va,\,{\rm fin}}^{\vp,\,\fz,\,s}(\rn)}}
\def\lv{{L^{\vp}(\rn)}}
\def\lvv{L^{{\vec{p}'}(\rn)}}
\def\de{\widetilde{A}}
\def\Qk{\widetilde{B}_k}
\def\Qkk{B_k^*}
\def\Bik{{B_i^k}}
\def\Qik{{Q_i^k}}

\newcommand{\sumn}{\sum_{i=1}^n}
\newtheorem{theorem}{Theorem}[section]
\newtheorem{lemma}[theorem]{Lemma}
\newtheorem{corollary}[theorem]{Corollary}
\newtheorem{proposition}[theorem]{Proposition}

\theoremstyle{definition}
\newtheorem{remark}[theorem]{Remark}
\newtheorem{definition}[theorem]{Definition}
\renewcommand{\appendix}{\par
   \setcounter{section}{0}%
   \setcounter{subsection}{0}%
   \setcounter{subsubsection}{0}%
   \gdef\thesection{\@Alph\c@section}%
   \gdef\thesubsection{\@Alph\c@section.\@arabic\c@subsection}%
   \gdef\theHsection{\@Alph\c@section.}%
   \gdef\theHsubsection{\@Alph\c@section.\@arabic\c@subsection}%
   \csname appendixmore\endcsname
 }

\numberwithin{equation}{section}

\begin{document}

\arraycolsep=1pt

\title{\bf\Large Atomic and Littlewood-Paley Characterizations of Anisotropic Mixed-Norm Hardy Spaces
and Their Applications
\footnotetext{\hspace{-0.35cm} 2010 {\it
Mathematics Subject Classification}. Primary 42B35;
Secondary 42B30, 42B25, 42B20, 30L99.
\endgraf {\it Key words and phrases.}
anisotropic Euclidean space, (mixed-norm) Hardy space, Calder\'on-Zygmund decomposition,
discrete Calder\'{o}n reproducing
formula, grand maximal function, atom,
Littlewood-Paley function, Calder\'{o}n-Zygmund operator.
\endgraf This project is supported by the National
Natural Science Foundation of China
(Grant Nos.~11761131002, 11571039, 11726621  and 11471042).}}
\author{Long Huang, Jun Liu, Dachun Yang
and Wen Yuan\footnote{Corresponding author / January 22, 2018.}}
\date{}
\maketitle

\vspace{-0.8cm}

\begin{center}
\begin{minipage}{13cm}
{\small {\bf Abstract}\quad
Let $\vec{a}:=(a_1,\ldots,a_n)\in[1,\infty)^n$,
$\vec{p}:=(p_1,\ldots,p_n)\in(0,\infty)^n$
and $H_{\vec{a}}^{\vec{p}}(\mathbb{R}^n)$
be the anisotropic mixed-norm Hardy space associated with $\vec{a}$
defined via the non-tangential grand maximal function. In this article,
via first establishing a
Calder\'{o}n-Zygmund decomposition and a discrete Calder\'{o}n reproducing
formula, the authors then characterize $H_{\vec{a}}^{\vec{p}}(\mathbb{R}^n)$,
respectively, by means of atoms,
the Lusin area function, the Littlewood-Paley $g$-function or
$g_{\lambda}^\ast$-function.
The obtained Littlewood-Paley $g$-function characterization of
$H_{\vec{a}}^{\vec{p}}(\mathbb{R}^n)$ coincidentally confirms a conjecture proposed
by Hart et al. [Trans. Amer. Math. Soc. (2017), DOI: 10.1090/tran/7312].
Applying the aforementioned Calder\'{o}n-Zygmund decomposition as well as
the atomic characterization of $H_{\vec{a}}^{\vec{p}}(\mathbb{R}^n)$,
the authors establish a finite atomic characterization of
$H_{\vec{a}}^{\vec{p}}(\mathbb{R}^n)$, which further induces
a criterion on the boundedness of sublinear operators
from $H_{\vec{a}}^{\vec{p}}(\mathbb{R}^n)$ into a quasi-Banach space.
Then, applying this criterion, the authors obtain the boundedness of anisotropic
Calder\'{o}n-Zygmund operators from $H_{\vec{a}}^{\vec{p}}(\mathbb{R}^n)$ to
itself [or to $L^{\vec{p}}(\mathbb{R}^n)$]. The obtained
atomic characterizations of $H_{\vec{a}}^{\vec{p}}(\mathbb{R}^n)$
and boundedness of anisotropic Calder\'{o}n-Zygmund
operators on these Hardy-type spaces positively answer
two questions mentioned by Cleanthous et al. in [J. Geom. Anal. 27 (2017), 2758-2787]. 
All these results are new even for the isotropic mixed-norm Hardy spaces on $\mathbb{R}^n$.
}
\end{minipage}
\end{center}

\vspace{0.215cm}

\section{Introduction\label{s1}}

The real-variable theory of Hardy spaces on the Euclidean space $\rn$ certainly plays an
important role in analysis, including harmonic analysis, partial differential equations
and geometrical analysis, and has been systematically studied and
developed so far; see, for example, \cite{fs72,lg14b,lu,s93,sw60}. It is well known that
Hardy spaces are good substitutes of Lebesgue spaces $L^p(\rn)$, with $p\in(0,1]$,
particularly, in the study on the boundedness of maximal functions
and Calder\'{o}n-Zygmund operators.
Notice that, as a generalization of the classical Lebesgue space $L^p(\rn)$,
the mixed-norm Lebesgue space $\lv$, in which the constant exponent
$p$ is replaced by an exponent vector $\vp\in [1,\fz]^n$, was studied by Benedek and
Panzone \cite{bp61} in 1961, which can be traced back to H\"{o}rmander \cite{h60}.
Later, in 1970, Lizorkin \cite{l70} further studied both the theory of
multipliers of Fourier integrals and estimates of convolutions
in the mixed-norm Lebesgue spaces.
Moreover, very recently, there appears a renewed increasing interest in
the theory of mixed-norm function spaces, including mixed-norm Lebesgue spaces,
mixed-norm Hardy spaces, mixed-norm Besov spaces and mixed-norm Triebel-Lizorkin spaces;
see, for example, \cite{cgn17,cgn17-2,gn16,htw17,jms13,jms14,jms15}.
For more developments of mixed-norm function spaces, we refer the reader to
\cite{cs,cgn17bs,f87,gjn17,js07,js08}.

On the other hand, due to the celebrated work \cite{c77,ct75,ct77} of Calder\'{o}n and
Torchinsky on parabolic Hardy spaces, there has been an increasing interest
in extending classical function spaces from Euclidean spaces to
some more general underlying spaces; see, for example,
\cite{mb03,ds16,fs82,hmy06,hmy08,s13,s16,s17,st87,t06,t15,yyh13}.
In particular, Bownik \cite{mb03} studied the anisotropic Hardy space $H_A^p(\rn)$ with
$A$ being a general expansive matrix on $\rn$ and $p\in(0,\fz)$,
which was a generalization of parabolic Hardy spaces introduced in \cite{c77}.
Later on, Bownik et al. \cite{blyz08} further extended the anisotropic Hardy space on $\rn$
to the weighted setting. For more progresses about the theory of anisotropic function spaces on $\rn$,
we refer the reader to \cite{blyz10,fhly17,lby14,lbyz10,lfy15,lwyy17,lyy16,lyy16LP,lyy17hl}
for anisotropic Hardy-type spaces and to \cite{mb05,mb07,bh06,lbyy12,lbyy14,lyy17}
for anisotropic Besov and Triebel-Lizorkin spaces.
Very recently, Cleanthous et al. \cite{cgn17} introduced the anisotropic mixed-norm Hardy space $\vh$
with $\va\in [1,\fz)^n$ and $\vp\in (0,\fz)^n$ via the non-tangential grand maximal function and
established its radial or non-tangential maximal function characterizations; moreover, they 
mentioned several natural questions to be studied, which include the atomic characterizations
of $H_{\vec{a}}^{\vec{p}}(\mathbb{R}^n)$ and the boundedness of anisotropic Calder\'{o}n-Zygmund
operators on these Hardy-type spaces.
For more progresses about this theory, we refer the reader to
\cite{bn08,cgn17,f00,fjs00,jms13,jms14,jms15,yam86a,yam86b}.
Notice that, when $\vp:=(p,\ldots,p)\in(0,\fz)^n$, the anisotropic mixed-norm Hardy space $\vh$
becomes the anisotropic Hardy space $H_{\va}^p(\rn)$. Here,
we should point out that, in this case, $H_{\va}^p(\rn)$
and the anisotropic Hardy space $H_A^p(\rn)$ (see \cite{mb03})
coincide with equivalent quasi-norms, where $A$ is as in \eqref{4e5} below
(see Proposition \ref{2r4'} below).
In addition, Hart et al. \cite{htw17} introduced the mixed-norm Hardy space
$H^{p,q}(\mathbb{R}^{n+1})$ with $p,\,q\in(0,\fz)$ via the Littlewood-Paley $g$-function and
showed that $H^{p,q}(\mathbb{R}^{n+1})$, when $p,\,q\in(1,\fz)$, coincides,
in the sense of equivalent quasi-norms,
with $H_{\va}^{\vp}(\mathbb{R}^{n+1})$ from \cite{cgn17} when
$$\va:=(\overbrace{1,\ldots, 1}^{n+1\ \mathrm{times}})\quad {\rm and}\quad
\vp:=(\overbrace{p,\ldots,p}^{n\ \mathrm{times}},q)\in(1,\fz)^{n+1},$$
which is defined via the non-tangential grand maximal function
(see \cite[Definition 3.3]{cgn17} or Definition 2.11 below); moreover,
Hart et al. in \cite[p.\,9]{htw17} stated that ``We do not know if such mixed Hardy spaces coincide with the
$H^{p,q}(\mathbb{R}^{n+1})$ above for other values of $p$ and $q$, but it is likely",
in which such mixed Hardy spaces mean the Hardy-type spaces $\vh$.

In addition, recall that the classical isotropic singular integral operator
was first introduced by Calder\'{o}n and Zygmund \cite{cz52}, in which they
established the boundedness of these operators on $L^p(\rn)$ for any $p\in(1,\fz)$.
Later, Fern\'{a}ndez \cite{f87} investigated the corresponding boundedness
of some classical isotropic singular integral operators on the mixed-norm Lebesgue space
$\lv$ with $\vp\in (1,\fz)^n$ (see also Stefanov and Torres \cite{st04}).
For more developments of the boundedness of the classical isotropic singular
integral operators, we refer the reader to Torres \cite{t91}.
On the other hand, in 1966, Besov et al. \cite{bil66} and, independently,
Fabes and Rivi\`{e}re \cite{f66} introduced a class of anisotropic singular
integral operators and obtained the $L^p(\rn)$ boundedness of these operators
for any $p\in(1,\fz)$. Moreover, the boundedness of these anisotropic singular
integral operators from \cite{f66} on generalized Morrey spaces was studied by
Guliyev and Mustafayev \cite{gm11} in 2011, which extends the corresponding
results obtained by Besov et al. \cite{bil66} as well as Fabes and Rivi\`{e}re
\cite{f66}. However, the boundedness of the anisotropic singular integral operators
on the mixed-norm Lebesgue space $\lv$ with $\vp\in (1,\fz)^n$ and
from the anisotropic mixed-norm Hardy space $\vh$ (even from the isotropic mixed-norm Hardy space)
to itself or to $\lv$ is still unknown so far, where $\va\in [1,\fz)^n$ and $\vp\in (0,1]^n$.

Let
$$\vec{a}:=(a_1,\ldots,a_n)\in[1,\infty)^n,\ \vec{p}:=(p_1,\ldots,p_n)\in(0,\infty)^n$$
and $H_{\vec{a}}^{\vec{p}}(\mathbb{R}^n)$
be the anisotropic mixed-norm Hardy space associated with $\vec{a}$
introduced by Cleanthous et al. in \cite{cgn17}, via the non-tangential grand maximal function.
In this article, we confirm the aforementioned conjecture proposed by Hart et al. \cite{htw17}
via completing the real-variable theory
of $\vh$ initially studied by Cleanthous et al. in \cite{cgn17}. To be precise, via first establishing a
Calder\'{o}n-Zygmund decomposition and a discrete Calder\'{o}n reproducing
formula, we then characterize $H_{\vec{a}}^{\vec{p}}(\mathbb{R}^n)$,
respectively, by means of atoms,
the Lusin area function, the Littlewood-Paley $g$-function or
$g_{\lambda}^\ast$-function.
The obtained Littlewood-Paley $g$-function characterization of
$H_{\vec{a}}^{\vec{p}}(\mathbb{R}^n)$ coincidentally confirms the aforementioned conjecture of
Hart et al. Applying the aforementioned Calder\'{o}n-Zygmund decomposition as well as
the atomic characterization of $H_{\vec{a}}^{\vec{p}}(\mathbb{R}^n)$,
we establish a finite atomic characterization of
$H_{\vec{a}}^{\vec{p}}(\mathbb{R}^n)$, which further induces
a criterion on the boundedness of sublinear operators
from $H_{\vec{a}}^{\vec{p}}(\mathbb{R}^n)$ into a quasi-Banach space.
Then, applying this criterion, we obtain the boundedness of anisotropic
convolutional $\delta$-type and non-convolutional $\bz$-order
Calder\'{o}n-Zygmund operators from $H_{\vec{a}}^{\vec{p}}(\mathbb{R}^n)$ to
itself [or to $\lv$] with
$\delta\in(0,1]$, $\beta\in(0,\fz)\setminus\mathbb{N}$, $\vec{p}\in (0,1]^n$
and $\widetilde{p}_-\in(\frac{\nu}{\nu+\delta},1]$
or $\widetilde{p}_-\in(\frac{\nu}{\nu+\beta},\frac{\nu}{\nu+\lfloor\beta\rfloor a_-}]$,
where $\nu:=a_1+\cdots+a_n$, $\widetilde{p}_-:=\min\{p_1,\ldots,p_n\}$,
$a_-:=\min\{a_1,\ldots,a_n\}$ and $\lfloor\beta\rfloor$ denotes the largest integer
not greater than $\beta$. We should point out that the obtained
atomic characterizations of $H_{\vec{a}}^{\vec{p}}(\mathbb{R}^n)$
and boundedness of anisotropic Calder\'{o}n-Zygmund 
operators on these Hardy-type spaces positively answer
two questions mentioned by Cleanthous et al. in \cite[p.\,2760]{cgn17}. All these results are
new even for the isotropic mixed-norm Hardy spaces on $\mathbb{R}^n$.

This article is organized as follows.

In Section \ref{s2}, we first present some notation and notions used in this article,
including the anisotropic homogeneous quasi-norm, the anisotropic bracket, the mixed-norm Lebesgue space
and their basic properties. Then we recall the definition of anisotropic mixed-norm Hardy
spaces $\vh$ via the non-tangential grand maximal function from \cite{cgn17}.

The aim of Section \ref{s3} is to establish the atomic characterizations of $\vh$.
Recall that, in the proof of the atomic decomposition for the classical isotropic Hardy space $H^p(\rn)$,
we need to use the Calder\'on-Zygmund decomposition to decompose any element of $H^p(\rn)$ into
a sum of atoms (see, for example, \cite{s93}). Thus, in this section, we first establish a Calder\'{o}n-Zygmund decomposition
in anisotropic $\rn$ (see Lemma \ref{3l4} below) by borrowing some ideas from
the proof of Stein \cite[p.\,101, Proposition]{s93}; the obtained Calder\'{o}n-Zygmund
decomposition actually extends Stein \cite[p.\,101, Proposition]{s93}
and Grafakos \cite[Theorem 5.3.1]{lg14} as well as Sawano et al. \cite[Lemma 2.23]{shyy17} to the present setting.
Then, applying this Calder\'{o}n-Zygmund decomposition, we show the density of the subset
$L^{\vp/\wz{p}_-}(\rn)\cap\vh$ in $\vh$ for any $\vp\in(0,\fz)^n$ with $\wz{p}_-$ as
in \eqref{3e1} below (see Lemma \ref{3l7} below). By this density and the anisotropic
Calder\'{o}n-Zygmund decomposition associated with non-tangential grand maximal
functions as well as an argument similar to that used in
the proof of Nakai and Sawano \cite[Theorem 4.5]{ns12}, we then prove that $\vh$ is continuously
embedded into $\vahfz$ and hence also into $\vah$ due to the fact that each
$(\vp,\infty,s)$-atom is also a $(\vp,r,s)$-atom for any $r\in(1,\fz)$,
where $\vah$ denotes the anisotropic mixed-norm atomic Hardy space (see Definition \ref{3d2} below).
Conversely, via borrowing some ideas from the proofs of
Sawano \cite[Theorem 1]{sawa13} and Zhuo at al.
\cite[Proposition 2.11]{zsy16}, we first show that some estimates related to $\lv$ norms
for some series of functions can be reduced into dealing with the $L^r(\rn)$ norms of
the corresponding functions with $r\in (\max\{1,p_+\},\fz]$ and $p_+$ as in \eqref{2e10}
below (see Lemma \ref{3l6} below), which plays a key role in the proof of the atomic
characterizations of $\vh$ (see Theorem \ref{3t1} below)
and is also of independent interest. Indeed, using this key
lemma, the anisotropic Fefferman-Stein vector-valued inequality of the Hardy-Littlewood
maximal operator $\HL$ on $\lv$ (see Lemma \ref{3l2} below) and an argument similar to that used in
the proof of Sawano et al. \cite[Theorem 3.6]{shyy17}, we prove that
$\vah\st\vh$ and the inclusion is continuous, which then completes the proof
of the atomic characterizations of $\vh$. We point out that the obtained 
atomic characterizations of $\vh$ gives a positive answer to a question
mentioned by Cleanthous et al. in \cite[p.\,2760]{cgn17}.

In Section \ref{s4}, as an application of the atomic characterization of
$\vh$ obtained in Theorem \ref{3t1}, we establish characterizations of
$\vh$ via Littlewood-Paley functions, including the Lusin
area function, the Littlewood-Paley $g$-function or $g^*_{\lambda}$-function.
Indeed, via  borrowing some ideas from the proof of
Bownik et al. \cite[Lemma 2.12]{blyz10}, we establish a discrete Calder\'{o}n reproducing
formula (see Lemma \ref{4l5} below) associated to the anisotropic homogeneous quasi-norm on
$\rn$ for distributions vanishing weakly at infinity, which was introduced by
Folland and Stein \cite{fs82} on homogeneous groups.
Applying this discrete Calder\'{o}n reproducing formula
and an argument similar to that used in the proof
of Theorem \ref{3t1}, we first establish the Lusin  area function
characterization of $\vh$ (see Theorem \ref{4t1} below).
Then, using this characterization and an approach initiated by Ullrich \cite{u12}
and further developed by Liang et al. \cite{lsuyy} and Liu et al. \cite{lyy17},
together with the anisotropic Fefferman-Stein
vector-valued inequality of the Hardy-Littlewood maximal operator
$\HL$ on $\lv$ (see Lemma \ref{3l2} below), we establish the Littlewood-Paley
$g$-function and $g_{\lambda}^\ast$-function characterizations of $\vh$
(see, respectively, Theorems \ref{4t2} and \ref{4t3} below).
Indeed, by the aforementioned approach from Ullrich \cite{u12}, via a key lemma
(see Lemma \ref{4l6} below) and an auxiliary function $g_{t,\ast}(f)$
(see \eqref{4e20} below), we show that the $\lv$ quasi-norm of the Lusin area
function can be controlled by that of the Littlewood-Paley $g$-function.
Moreover, the Littlewood-Paley $g$-function characterization of $\vh$ obtained in
Theorem \ref{4t2} below confirms the conjecture proposed by
Hart et al. in \cite[p.\,9]{htw17}.

Section \ref{s5} is devoted to establishing a finite atomic characterization of $\vh$.
In what follows, we use $C_c^\fz(\rn)$ to denote the set of all infinitely differentiable
functions with compact supports.
For any triplet $(\vp,r,s)$ as in Theorem \ref{3t1} below, we first show that $\vh\cap L^q(\rn)$,
with $q\in[1,\fz]$, and $\vh\cap C_c^\fz(\rn)$ are both dense in $\vh$ (see Lemma \ref{5l6} below),
and we then establish the finite atomic characterizations of $\vh$ (see Theorem \ref{5t1} below).
To be precise, via borrowing some ideas from the proofs of
\cite[Theorem 5.7]{lyy16} and \cite[Theorem 2.14]{lwyy17hlLP}, we prove that, for any given finite linear
combination of $(\vp,r,s)$-atoms with $r\in(\max\{p_+,1\},\fz)$ (or continuous $(\vp,\fz,s)$-atoms),
its quasi-norm in $\vh$ can be achieved via all its finite combinations of atoms of the same type.
This extends Meda et al.
\cite[Theorem 3.1 and Remark 3.3]{msv08} and Grafakos et al. \cite[Theorem 5.6]{gly08}
to the present setting of anisotropic mixed-norm Hardy spaces.

In Section \ref{s6}, applying the finite atomic
characterizations for $\vh$ obtained in Section \ref{s5},
we establish a criterion on the boundedness of
sublinear operators from $\vh$ into a quasi-Banach space
(see Theorem \ref{6t1} below), which is of independent interest;
moreover, using this criterion, we further show that,
if $T$ is a sublinear operator and maps all $(\vp,r,s)$-atoms
with $r\in(1,\fz)$ (or all continuous $(\vp,\fz,s)$-atoms) into uniformly bounded
elements of some $\gamma$-quasi-Banach space $\mathcal{B}_{\gamma}$ with $\gamma\in (0,1]$,
then $T$ has a unique bounded $\mathcal{B}_{\gamma}$-sublinear extension from $\vh$ into $\mathcal{B}_{\gamma}$
(see Corollary \ref{6c1} below).
This extends the corresponding results of Meda et al. \cite[Corollary 3.4]{msv08}
and Grafakos et al. \cite[Theorem 5.9]{gly08} as well as Ky \cite[Theorem 3.5]{ky14}
(see also \cite[Theorem 1.6.9]{ylk17})
to the present setting. Finally, via borrowing some ideas from the proofs
of Yan et al. \cite[Theorems 7.4 and 7.6]{yyyz16} and the criterion established
in Theorem \ref{6t1} and Corollary \ref{6c1} below, we also obtain the boundedness
of anisotropic convolutional $\delta$-type and anisotropic non-convolutional
$\bz$-order Calder\'{o}n-Zygmund operators
from $\vh$ to itself (see Theorems \ref{6t2} and  \ref{6t4} below) or to $\lv$
(see Theorems \ref{6t3} and \ref{6t5} below),
where $\delta\in(0,1]$, $\bz\in(0,\fz)\setminus\nn$, $\vp\in (0,1]^n$ and $\widetilde{p}_-\in(\frac{\nu}{\nu+\delta},1]$
or $\widetilde{p}_-\in(\frac{\nu}{\nu+\bz},\frac{\nu}{\nu+\lfloor\bz\rfloor a_-}]$
with $\widetilde{p}_-$ as in \eqref{3e1} and $a_-$ as in \eqref{2e9} below.
We point out that Theorem \ref{6t2} extends the corresponding results of
Fefferman and Stein \cite[Theorem 12]{fs72} and that
Theorems \ref{6t4} and \ref{6t5} extend the corresponding results of
Stefanov and Torres \cite[Theorem 1]{st04} as well as Yan et al.
\cite[Theorem 7.6]{yyyz16} to the present setting. We also point out that
the obtained boundedness of these anisotropic Calder\'{o}n-Zygmund operators
on $\vh$ positively answers a question
mentioned by Cleanthous et al. in \cite[p.\,2760]{cgn17}.

Finally, we make some conventions on notation.
We always let
$\mathbb{N}:=\{1,2,\ldots\}$,
$\mathbb{Z}_+:=\{0\}\cup\mathbb{N}$
and $\vec0_n$ be the \emph{origin} of $\rn$.
For any multi-index
$\bz:=(\bz_1,\ldots,\bz_n)\in(\mathbb{Z}_+)^n=:\mathbb{Z}_+^n$,
let
$|\bz|:=\bz_1+\cdots+\bz_n$ and
$\pa^{\bz}:=(\f{\pa}{\pa x_1})^{\bz_1} \cdots (\f{\pa}{\pa x_n})^{\bz_n}.$
We denote by $C$ a \emph{positive constant}
which is independent of the main parameters,
but may vary from line to line. We also use $C_{(\az,\bz,\ldots)}$ to denote a positive constant
depending on the indicated parameters $\az,\,\bz,\ldots$.
The notation $f\ls g$ means $f\le Cg$ and, if $f\ls g\ls f$,
then we write $f\sim g$. For any $q\in[1,\fz]$, we denote by $q'$
its \emph{conjugate index}, namely, $1/q+1/q'=1$.
Moreover, if $\vec{q}:=(q_1,\ldots,q_n)\in[1,\fz]^n$, we denote by
$\vec{q}':=(q_1',\ldots,q_n')$ its \emph{conjugate index}, namely,
for any $i\in\{1,\ldots,n\}$, $1/q_i+1/q_i'=1$. In addition,
for any set $E\subset\rn$, we denote by $E^\complement$ the
set $\rn\setminus E$, by $\chi_E$ its \emph{characteristic function},
by $|E|$ its \emph{n-dimensional Lebesgue measure}
and by $\sharp E$ its \emph{cardinality}.
For any $s\in\mathbb{R}$, we denote by $\lfloor s\rfloor$
the \emph{largest integer not greater than $s$}. In what follows, $C^{\fz}(\rn)$
denotes the set of all \emph{infinite differentiable functions} on $\rn$
and $C_c^{\fz}(\rn)$ the set of all $C^{\fz}(\rn)$ functions with compact supports.

\section{Preliminaries \label{s2}}

In this section, we present the definition of the anisotropic mixed-norm Hardy space
via the non-tangential grand maximal function from \cite{cgn17}. To this end,
we first recall the notion of anisotropic homogeneous quasi-norms
and then state some of their basic conclusions to be used in this article.

We begin with recalling the definition of anisotropic homogeneous quasi-norms
from \cite{bil66, f66} (see also \cite{sw78}) as follows.
For any $b:=(b_1,\ldots,b_n)$, $x:=(x_1,\ldots,x_n)\in \rn$ and $t\in[0,\fz)$,
let $t^b x:=(t^{b_1}x_1,\ldots,t^{b_n}x_n)$.

\begin{definition}\label{2d1}
Let $\va:=(a_1,\ldots,a_n)\in [1,\fz)^n$.
The \emph{anisotropic homogeneous quasi-norm} $|\cdot|_{\va}$,
associated with $\va$, is a non-negative measurable function on $\rn$
defined by setting $|\vec0_n|_{\va}:=0$ and, for any $x\in \rn\setminus\{\vec0_n\}$,
$|x|_{\va}:=t_0$, where $t_0$ is the unique positive number such that $|t_0^{-\va}x|=1$,
namely,
$$\f{x_1^2}{t_0^{2a_1}}+\cdots+\f{x_n^2}{t_0^{2a_n}}=1.$$
\end{definition}

We also need the following notion of the anisotropic bracket
and the homogeneous dimension from \cite{sw78}, which plays an important role
in the study on anisotropic function spaces.

\begin{definition}\label{2d2}
Let $\va:=(a_1,\ldots,a_n)\in [1,\fz)^n$. The \emph{anisotropic bracket}, associated with $\va$,
is defined by setting,  for any $x\in \rn$, $$\lg x\rg_{\va}:=|(1,x)|_{(1,\va)}.$$
Furthermore, the \emph{homogeneous dimension} $\nu$ is defined as
$$\nu:=|\vec{a}|:=a_1+\cdots +a_n.$$
\end{definition}

\begin{remark}\label{2r1}
\begin{enumerate}
\item[{\rm(i)}] It is easy to see that, for any $x\in \rn$, $|x|_{\va} < \lg x\rg_{\va}$.
\item[{\rm(ii)}] By $\lg \vec0_n\rg_{\va} = 1$ and the fact that,
for any $x\in \rn\setminus\{\vec0_n\}$, $\lg x\rg_{\va} > 1$, we find that,
for any $x\in \rn$, $\lg x\rg_{\va} \geq 1$.
\item[{\rm(iii)}] By \cite[(2.7)]{cgn17}, we know that $\lg \cdot \rg_{\va} $ belongs to $C^{\fz}(\rn)$ and,
for any $s\in \rr$ and multi-index $\bz \in \mathbb{Z}_+^n$, there exists a positive
constant $C_{(\va,s,\bz)}$, depending on $\va$, s and $\bz$, such that, for any $x\in \rn$,
\begin{align*}
\lf|\pa^{\bz}\lg x\rg_{\va}^s\r| \le C_{(\va,s,\bz)}\lg x\rg_{\va}^{s-\va\cdot\bz},
\end{align*}
here and hereafter, for any $\az:=(\az_1,\ldots,\az_n)$,
$\bz:=(\bz_1,\ldots,\bz_n) \in \rn$, $\az\cdot\bz:=\sum_{i=1}^{n}\az_i \bz_i$.
\end{enumerate}
\end{remark}

To compare the anisotropic homogeneous quasi-norm with the Euclidean norm,
we need the following Lemma \ref{2l1}, which is easy to be proved by using
Definition \ref{2d1}, the details being omitted.

\begin{lemma}\label{2l1}
Let $\va\in [1,\fz)^n$ and $x\in \rn$. Then
\begin{enumerate}
\item[{\rm(i)}] $|x|_{\va} > 1$ if and only if $|x| > 1$;
\item[{\rm(ii)}] $|x|_{\va} < 1$ if and only if $|x| < 1.$
\end{enumerate}
\end{lemma}

For any $\va:=(a_1,\ldots,a_n)\in [1,\fz)^n$, let
\begin{align}\label{2e9}
a_-:=\min\{a_1,\ldots,a_n\}\hspace{0.35cm}
{\rm and}\hspace{0.35cm} a_+:=\max\{a_1,\ldots,a_n\}.
\end{align}

Now let us recall some basic properties of $|\cdot|_{\va}$
and $\lg \cdot\rg_{\va}$;
see \cite{cgn17, js07, js08, sw78} for more details.

\begin{lemma}\label{2l2}
Let $\va:=(a_1,\ldots,a_n) \in [1,\fz)^n$, $t\in[0,\fz)$ and $a_-$ and $a_+$ be as
in \eqref{2e9}. Then, for any $x,\ y\in \rn$,

\begin{enumerate}
\item[{\rm(i)}] $\lf|t^{\va}x\r|_{\va}= t|x|_{\va}$;
\item[{\rm(ii)}] $|x+y|_{\va}\le |x|_{\va}+|y|_{\va}$;
\item[{\rm(iii)}] $\max\{|x_1|^{1/a_1},\ldots,|x_n|^{1/a_n}\}
\le |x|_{\va}\le \sumn |x_i|^{1/a_i}$;
\item[{\rm(iv)}] when $|x| \geq 1$, $|x|^{1/a_+} \le |x|_{\va} \le|x|^{1/a_-}$;
\item[{\rm(v)}] when $|x| < 1$, $|x|^{1/a_-} \le |x|_{\va} \le|x|^{1/a_+}$;
\item[{\rm(vi)}] $(\f{1}{2})^{a_-}(1+|x|_{\va})^{a_-} \le 1+|x| \le 2(1+|x|_{\va})^{a_+}$;
\item[{\rm(vii)}] $ \lg x\rg_{\va} \le 1+|x|_{\va}\le 2\lg x\rg_{\va}$;
\item[{\rm(viii)}] $\lg x+y\rg_{\va} \le 4\lg x\rg_{\va} \lg y\rg_{\va}$;
\item[{\rm(ix)}] for any measurable function $f$ on $\rn$,
$$\int_{\rn}f(x)\,dx=\int_0^{\fz}\int_{S^{n-1}}f(\rho^{\va}\xi)\rho^{\nu-1}\,d\sigma(\rho)\,d\rho,$$
where $S^{n-1}$ denotes the $n-1$ dimension unit sphere of $\rn$
and $\sigma(\rho)$ the spherical measure.
\end{enumerate}
\end{lemma}

\begin{remark}\label{2r2}
\begin{enumerate}
\item[{\rm(i)}] By Lemma \ref{2l2}(i), we easily know that the anisotropic
quasi-homogeneous norm $|\cdot|_{\va}$ is a norm if and only if
$\va=(\overbrace{1,\ldots,1}^{n\ \mathrm{times}})$ and, in this case,
the homogeneous quasi-norm $|\cdot|_{\va}$ becomes the Euclidean norm $|\cdot|$.
\item[{\rm(ii)}] It is easy to see that $\va\in [1,\fz)^n$ guarantees Lemma \ref{2l2}(ii).
\end{enumerate}
\end{remark}

For any $\va\in [1,\fz)^n$, $r\in (0,\fz)$ and $x\in \rn$,
we define the \emph{anisotropic ball} $B_{\va}(x,r)$, with center $x$ and radius $r$, by
setting $B_{\va}(x,r):=\{y\in \rn:\,|y-x|_{\va} < r\}$.
Then $B_{\va}(x,r)= x+r^{\va}B_{\va}(\vec0_n,1)$ and
$|B_{\va}(x,r)|=\nu_n r^{\nu}$, where $\nu_n:=|B_{\va}(\vec0_n,1)|$ (see \cite[(2.12)]{cgn17}).
Moreover, from Lemma \ref{2l1}(ii), we deduce that $B_0:=B_{\va}(\vec0_n,1)=B(\vec0_n,1)$,
where $B(\vec0_n,1)$ denotes the unit ball of $\rn$, namely, $B(\vec0_n,1):=\{y\in \rn:\,|y|<1\}$.
In what follows, we always let $\mathfrak{B}$ be the set of all anisotropic balls, namely,
\begin{align}\label{2e2}
\mathfrak{B}:=\lf\{B_{\va}(x,r):\ x\in\rn,\ r\in(0,\fz)\r\}.
\end{align}
For any $B\in\mathfrak{B}$ centered at $x\in\rn$ with radius $r\in (0,\fz)$ and $\delta\in(0,\fz)$,
let
\begin{align}\label{2e2'}
B^{(\delta)}:=B^{(\delta)}_{\va}(x,r):=B_{\va}(x,\delta r).
\end{align}
In addition, for any $x\in \rn$ and $r\in (0,\fz)$, the \emph{anisotropic cube}
$Q_{\va}(x,r)$ is defined by setting $Q_{\va}(x,r):=x + r^{\va}(-1,1)^n,$
whose Lebesgue measure $|Q_{\va}(x,r)|$ equals $2^n r^\nu$.
Denote by $\mathfrak{Q}$ the set of all anisotropic cubes, namely,
\begin{align}\label{2e3}
\mathfrak{Q}:=\lf\{Q_{\va}(x,r):\ x\in\rn,\ r\in(0,\fz)\r\}.
\end{align}

Now let us recall the definition of mixed-norm Lebesgue spaces from \cite{bp61}.
\begin{definition}\label{2d3}
Let $\vp:=(p_1,\ldots,p_n)\in (0,\fz]^n$. The \emph{mixed-norm Lebesgue space} $\lv$ is
defined to be the set of all measurable functions $f$ such that their quasi-norms
$$\|f\|_{\lv}:=\left\{\int_{\rr}\cdots\left[\int_{\rr}\left\{\int_{\rr}|f(x_1,\ldots,x_n)|^{p_1}
\,dx_1\right\}^{\f{p_2}{p_1}}\,dx_2\right]^{\f{p_3}{p_2}}\cdots\, dx_n\right\}^{\f{1}{p_n}}<\fz$$
with the usual modifications made when $p_i=\fz$, $i\in \{1,\ldots,n\}$.
\end{definition}

\begin{remark}\label{2r3}
\begin{enumerate}
  \item[{\rm(i)}] Obviously, when $\vp=(p,\ldots,p)\in (0,\fz]^n$, $\lv$ coincides with
  the classical Lebesgue space $L^p(\rn)$.
  \item[{\rm(ii)}] For any $\vp\in(0,\fz]^n$, $(\lv,\|\cdot\|_{\lv})$
  is a quasi-Banach space and, for any $\vp \in [1,\fz]^n$,
  $(\lv,\|\cdot\|_{\lv})$ becomes a Banach space; see \cite[p.\,304, Theorem 1]{bp61}.
  \item[{\rm(iii)}] Let $\vp\in (0,\fz]^n$. Then, for any $s\in (0,\fz)$ and $f\in\lv$,
  \begin{align}\label{2e8}\lf\||f|^s\r\|_{\lv}=\|f\|_{L^{s \vp}(\rn)}^s.\end{align}
  In addition, for any $\lz\in{\mathbb C}$, $\theta\in [0,\min(1,p_1,\ldots,p_n)]$
  and $f,\ g\in\lv$, $\|\lz f\|_{\lv}=|\lz|\|f\|_{\lv}$ and
  \begin{align*}
  \|f+g\|_{\lv}^{\theta}\le \|f\|_{\lv}^{\theta}
  +\|g\|_{\lv}^{\theta}
  \end{align*}
  (see \cite[p.\,188]{js07}).
  \item[{\rm(iv)}] Let $\vp\in [1,\fz]^n$. Then, for any $f\in\lv$ and $g\in L^{\vp'}(\rn)$,
  it is easy to see that
  \begin{align*}
  \int_{\rn}|f(x)g(x)|\,dx \le \|f\|_{\lv}\|g\|_{\lvv},
  \end{align*}
  where $\vp'$ denotes the conjugate index of $\vp$,
  namely, for any $i\in \{1,\ldots,n\}$, $1/p_i+1/p_i'=1$.
\end{enumerate}
\end{remark}

For any $\vp:=(p_1,\ldots,p_n)\in (0,\fz)^n$, we always let
\begin{align}\label{2e10}
p_-:=\min\{1,p_1,\ldots,p_n\}\hspace{0.35cm}
{\rm and}\hspace{0.35cm} p_+:=\max\{p_1,\ldots,p_n\}.
\end{align}

Recall that a \emph{Schwartz function} is a $C^\infty(\rn)$
function $\varphi$ satisfying,
for any $N\in\zz_+$ and multi-index $\az\in\zz_+^n$,
$$\|\varphi\|_{N,\alpha}:=
\sup_{x\in\rn}\lf\{(1+|x|)^N
|\partial^\alpha\varphi(x)|\r\}<\infty.$$
Denote by
$\cs(\rn)$ the set of all Schwartz functions, equipped
with the topology determined by
$\{\|\cdot\|_{N,\alpha}\}_{N\in\zz_+,\az\in\zz_+^n}$,
and $\cs'(\rn)$ the \emph{dual space} of $\cs(\rn)$, equipped
with the weak-$\ast$ topology.
For any $N\in\mathbb{Z}_+$, let
$$\cs_N(\rn):=\lf\{\varphi\in\cs(\rn):\
\|\varphi\|_{\cs_N(\rn)}:=
\sup_{x\in\rn}\lf[\lg x\rg_{\va}^N\sup_{|\az|\le N}
|\partial^\alpha\varphi(x)|\r]\le 1\r\}.$$
In what follows, for any $\varphi \in \cs(\rn)$ and $t\in (0,\fz)$,
let $\varphi_t(\cdot):=t^{-\nu}\varphi(t^{-\va}\cdot)$.

\begin{definition}\label{2d4}
Let $\varphi\in\cs(\rn)$ and $f\in\cs'(\rn)$. The
\emph{non-tangential maximal function} $M_\varphi(f)$,
with respect to $\varphi$, is defined by setting, for any $x\in\rn$,
$$
M_\varphi(f)(x):= \sup_{y\in B_{\va}(x,t),
t\in (0,\fz)}|f\ast\varphi_t(y)|.
$$
Moreover, for any given $N\in\mathbb{N}$, the
\emph{non-tangential grand maximal function} $M_N(f)$ of
$f\in\cs'(\rn)$ is defined by setting, for any $x\in\rn$,
\begin{equation*}
M_N(f)(x):=\sup_{\varphi\in\cs_N(\rn)}
M_\varphi(f)(x).
\end{equation*}
\end{definition}

\begin{remark}\label{2r5}
Obviously, for any given $N\in \nn$ and any $f\in \cs'(\rn)$, $\varphi\in\cs(\rn)$,
$t\in (0,\fz)$ and $x\in \rn$,
$$|f\ast\varphi_t(x)|\le \|\varphi\|_{\cs_N(\rn)}M_N(f)(x).$$
\end{remark}

We now recall the notion of anisotropic mixed-norm Hardy spaces as follows,
which is just \cite[Definition 3.3]{cgn17}.

\begin{definition}\label{2d5}
Let $\va:=(a_1,\ldots,a_n)\in[1,\fz)^n$, $\vp\in(0,\fz)^n$,
$N_{\vp}:=\lfloor(\nu\frac{a_+}{a_-}
(\f{1}{p_-}+1)+\nu+2a_+\rfloor+1$ and
\begin{align}\label{2e11}
N\in\mathbb{N}\cap \lf[N_{\vp},\fz\r),
\end{align}
where $a_-,\,a_+$ are as in \eqref{2e9} and $p_-$ is as in \eqref{2e10}.
The \emph{anisotropic mixed-norm Hardy space} $\vh$ is defined by setting
\begin{equation*}
\vh:=\lf\{f\in\cs'(\rn):\ M_N(f)\in\lv\r\}
\end{equation*}
and, for any $f\in\vh$, let
$\|f\|_{\vh}:=\| M_N(f)\|_{\lv}$.
\end{definition}

\begin{remark}\label{2r4}
The quasi-norm of $\vh$ in Definition \ref{2d5} depends on $N$, however,
by Theorem \ref{3t1} below, we know that the space $\vh$ is independent
of the choice of $N$ as long as $N$ is as in \eqref{2e11}.
In addition, if $\va:=(\overbrace{1,\ldots,1}^{n\ \rm times})$
and $\vp:=(\overbrace{p,\ldots,p}^{n\ \rm times})$, where $p\in(0,\fz)$,
then $\vh$ coincides with the classical isotropic Hardy space $H^p(\rn)$ of Fefferman and Stein \cite{fs72}.
\end{remark}

\section{Atomic characterizations of $\vh$\label{s3}}

In this section, we establish the atomic characterizations of $\vh$.
We begin with introducing the definition of anisotropic mixed-norm
$(\vp,r,s)$-atoms. In what follows, for any $r\in(0,\fz]$,
we use $L^r(\rn)$ to denote the space of all measurable functions $f$ such that
$$\|f\|_{L^r(\rn)}:=\lf\{\int_{\rn}|f(x)|^r\,dx\r\}^{1/r}<\fz$$
with the usual modification made when $r=\fz$.

\begin{definition}\label{3d1}
Let $\va\in[1,\fz)^n$, $\vp:=(p_1,\ldots,p_n)\in(0,\fz)^n$, $r\in (1,\fz]$ and
\begin{align}\label{3e1}
s\in\lf[\lf\lfloor\f{\nu}{a_-}\lf(\f{1}{\widetilde{p}_-}-1\r) \r\rfloor,\fz\r)\cap\zz_+,
\end{align}
where $a_-$ is as in \eqref{2e9} and $\widetilde{p}_-:=\min\{p_1,\ldots,p_n\}$.
An \emph{anisotropic mixed-norm $(\vp,r,s)$-atom} $a$ is
a measurable function on $\rn$ satisfying
\begin{enumerate}
\item[{\rm (i)}] $\supp a \st B$, where
$B\in\mathfrak{B}$ with $\mathfrak{B}$ as in \eqref{2e2};

\item[{\rm (ii)}] $\|a\|_{L^r(\rn)}\le \frac{|B|^{1/r}}{\|\chi_B\|_{\lv}}$;

\item[{\rm (iii)}] $\int_{\mathbb R^n}a(x)x^\az\,dx=0$ for any $\az\in\zz_+^n$
with $|\az|\le s$.
\end{enumerate}
\end{definition}

Throughout this article, we always call an anisotropic mixed-norm
$(\vp,r,s)$-atom simply by a $(\vp,r,s)$-atom.
Now, using $(\vp,r,s)$-atoms, we introduce the anisotropic
mixed-norm atomic Hardy space $\vah$ as follows.

\begin{definition}\label{3d2}
Let $\va\in[1,\fz)^n$, $\vp\in(0,\fz)^n$, $r\in (1,\fz]$
and $s$ be as in \eqref{3e1}. The \emph{anisotropic mixed-norm
atomic Hardy space} $\vah$ is defined to be the
set of all $f\in\cs'(\rn)$ satisfying that there exist
$\{\lz_i\}_{i\in\nn}\st\mathbb{C}$
and a sequence of $(\vp,r,s)$-atoms, $\{a_i\}_{i\in\nn}$,
supported, respectively, on
$\{B_i\}_{i\in\nn}\st\mathfrak{B}$ such that
\begin{align*}
f=\sum_{i\in\nn}\lz_ia_i
\quad\mathrm{in}\quad\cs'(\rn).
\end{align*}
Moreover, for any $f\in\vah$, let
\begin{align*}
\|f\|_{\vah}:=
{\inf}\lf\|\lf\{\sum_{i\in\nn}
\lf[\frac{|\lz_i|\chi_{B_i}}{\|\chi_{B_i}\|_{\lv}}\r]^
{p_-}\r\}^{1/{p_-}}\r\|_{\lv},
\end{align*}
where the infimum is taken over all decompositions of $f$ as above.
\end{definition}

Let $L_{\rm loc}^1(\rn)$ denote the collection of
all locally integrable functions on $\rn$.

\begin{definition}\label{3d3}
The \emph{Hardy-Littlewood maximal operator} $M_{{\rm HL}}(f)$ of $f\in L_{\rm loc}^1(\rn)$ is defined by setting,
for any $x\in\rn$,
\begin{align}\label{3e2}
M_{{\rm HL}}(f)(x):=\sup_{x\in Q\in\mathfrak{Q}}
\frac1{|Q|}\int_Q|f(y)|\,dy,
\end{align}
where $\mathfrak{Q}$ is as in \eqref{2e3}.
\end{definition}

\begin{remark}\label{3r1}
For any $f\in L_{\rm loc}^1(\rn)$ and $x\in \rn$, let
\begin{align*}
M(f)(x):=\sup_{I_n\in \mathbb{I}_{x_n}}\lf\{\f{1}{|I_n|}\int_{I_n}
\cdots\sup_{I_2\in \mathbb{I}_{x_2}}\lf[\f{1}{|I_2|}\int_{I_2}
\sup_{I_1\in \mathbb{I}_{x_1}}\lf\{\f{1}{|I_1|}\int_{I_1}
|f(y_1,y_2\ldots,y_n)|\,dy_1\r\}\,dy_2\r]\cdots \,dy_n\r\},
\end{align*}
where, for any $k\in \{1,\ldots,n\}$, ${\mathbb I}_{x_k}$ denotes the set of all intervals in $\rr_{x_k}$ containing
$x_k$. Then, it is easy to see that, for any $x\in \rn$, $$ M_{{\rm HL}}(f)(x)\le M(f)(x).$$
\end{remark}

To establish the atomic characterizations of $\vh$, we need several technical lemmas.
We begin with the following boundedness of the Hardy-Littlewood maximal operator
$M_{\rm HL}$ on $\lv$ with $\vp\in(1,\fz]^n$.

\begin{lemma}\label{3l1}
Let $\vp\in (1,\fz]^n$. Then there exists a positive
constant C, depending on $\vp$, such that, for any $f\in \lv$,
$$\|M_{{\rm HL}}(f)\|_{\lv}\le C\|f\|_{\lv},$$
where $M_{{\rm HL}}$ is as in \eqref{3e2}.
\end{lemma}

\begin{proof}
We first assume that $\vp:=(p_1,\ldots,p_n)\in(1,\fz)^n$. In this case, for any given
$m_1,~m_2\in \mathbb{Z}_+$ with $m_1+m_2=n$ and any $f\in\lv$, $s\in \rr^{m_1}$ and $t\in\rr^{m_2}$,
let $$f^*(s,t):=\sup_{r\in (0,\fz)}\f{1}{|B(s,r)|}
\int_{B(s,r)}|f(y,t)|\,dy.$$
In addition, for any given $\vec{p}_{m_2}:=(p_1,\ldots,p_{m_2})\in(1,\fz)^{m_2}$ and any $s\in \rr^{m_1}$, let
\begin{align*}
T_{L^{\vec{p}_{m_2}}(\rr^{m_2})}(f)(s)&:=\|f(s,\cdot)
\|_{L^{\vec{p}_{m_2}}(\rr^{m_2})}\\
&:=\left\{\int_{\rr}\cdots\left[\int_{\rr}\left\{\int_{\rr}
|f(s,t_1\ldots,t_{m_2})|^{p_1}\,dt_1\right\}^{\f{p_2}{p_1}}
\,dt_2\right]^{\f{p_3}{p_2}}\cdots \,dt_{m_2}\right\}^{\f{1}{p_{m_2}}}.
\end{align*}
Then it holds true that, for any $q\in (1,\fz)$,
\begin{align}\label{3e4}
\int_{\rr^{m_1}}\lf[T_{L^{\vec{p}_{m_2}}(\rr^{m_2})}(f^*)(s)\r]^q \,dx
\ls \int_{\rr^{m_1}}\lf[T_{L^{\vec{p}_{m_2}}(\rr^{m_2})}(f)(s)\r]^q \,dx
\end{align}
(see \cite[p.\,421]{b75}).
For any $k\in\{1,\ldots,n\}$ and $x:=(x_1,\cdots,x_n)\in \rn$, let
$$M_k(f)(x):=\sup_{I\in \mathbb{I}_{x_k}}\f{1}{|I|}\int_I
|f(x_1,\ldots,y_k,\ldots,x_n)|\,dy_k,$$
where $\mathbb{I}_{x_k}$ is as in Remark \ref{3r1}.
Then, for any $x\in\rn$, we have
\begin{align*}
M(f)(x)=M_n\lf(\cdots\lf(M_1(f)\r)\cdots\r)(x).
\end{align*}
By this, Remark \ref{3r1} and \eqref{3e4} with
$m_1:=1,~m_2:=n-1$, $s:=x_n$, $t:=(x_1,\ldots,x_{n-1})$
and $q:=p_n$, we conclude that
\begin{align}\label{3e5}
\|M_{{\rm HL}}(f)\|_{\lv}&\le\|M_n\lf(\cdots\lf(M_1(f)\r)\cdots\r)\|_{\lv}\\
&\ls\lf\{\int_{\rr}\lf[T_{L^{\vec{p}_{n-1}}(\rr^{n-1})}
\lf(\lf[M_{n-1}\lf(\cdots\lf(M_1(f)\r)\cdots\r)\r]^*\r)(x_n)\r]^{p_n}\, dx_n\r\}^{\f{1}{p_n}}\noz\\
&\ls \lf\{\int_{\rr}\lf[T_{L^{\vec{p}_{n-1}}(\rr^{n-1})}
(M_{n-1}\lf(\cdots\lf(M_1(f)\r)\cdots\r))(x_n)\r]^{p_n} \,dx_n\r\}^{\f{1}{p_n}}\noz\\
&\sim \|M_{n-1}\lf(\cdots\lf(M_1(f)\r)\cdots\r)\|_{\lv}\noz.
\end{align}
Repeating the above estimate $n-1$ times, we easily find that Lemma \ref{3l1} holds true
in the case when $\vp\in (1,\fz)^n$.

If $p_{i_0}=\fz$ for some $i_0\in \{1,\ldots,n\}$ and, for any $i\in \{1,\ldots,n\}$ with $i\neq i_0$,
$p_i\in (1,\fz)$, then, by an argument similar to that used in the estimation of \eqref{3e5} and the boundedness of $M_{\rm HL}$ on
$L^{\fz}(\rr)$ (see \cite[p.\,14]{mb03}), we easily conclude that Lemma \ref{3l1} also
holds true in this case. This finishes the proof of Lemma \ref{3l1}.
\end{proof}

By \cite[(2.24)]{gn16} and Remark \ref{3r1},
we immediately obtain the following Fefferman-Stein vector-valued inequality of
$M_{\rm HL}$ on $\lv$, the details being omitted.
\begin{lemma}\label{3l2}
Let $\vp\in (1,\fz)^n$ and $u\in(1,\fz]$.
Then there exists a positive
constant $C$ such that, for any
sequence $\{f_k\}_{k\in\nn}\st L^1_{\rm loc}(\rn)$,
$$\lf\|\lf\{\sum_{k\in\nn}
\lf[\HL(f_k)\r]^u\r\}^{1/u}\r\|_{\lv}
\le C\lf\|\lf(\sum_{k\in\nn}|f_k|^u\r)^{1/u}\r\|_{\lv}$$
with the usual modification made when $u=\fz$,
where $\HL$ denotes the Hardy-Littlewood maximal operator as in \eqref{3e2}.
\end{lemma}

\begin{remark}\label{3r2}
By Lemma \ref{3l2} and an argument similar to that used in the proof of \cite[Remark 2.5]{yyyz16},
we easily conclude that, for any given $\vp\in(0,\fz)^n$ and $r\in(0,p_-)$ with $p_-$
as in \eqref{2e10}, there exists a positive constant $C$ such that, for any $\bz\in[1,\fz)$
and sequence $\{B_i\}_{i\in\nn}\st\mathfrak{B}$,
$$\lf\|\sum_{i\in\nn}\chi_{B_i^{(\bz)}}\r\|_{\lv}\le C \bz^{\f{\nu}{r}}
\lf\|\sum_{i\in\nn}\chi_{B_i}\r\|_{\lv},$$
where $B_i^{(\bz)}$ is as in \eqref{2e2'}.
\end{remark}

\begin{definition}\label{3d4}
Let $\phi\in\cs(\rn)$ satisfying $\int_{\rn}\phi(x)\,dx\neq0$ and $f\in\cs'(\rn)$. The
\emph{radial maximal function} $M_\phi^0(f)$
of $f$, with respect to $\phi$, is defined by setting, for any $x\in\rn$,
\begin{equation*}
M_\phi^0(f)(x):= \sup_{t\in (0,\fz)}
|f\ast\phi_t(x)|.
\end{equation*}
\end{definition}

The following Lemma \ref{3l8} is from \cite[Theorem 3.4]{cgn17}.
\begin{lemma}\label{3l8}
Let $\va\in[1,\fz)^n$, $\vp\in(0,\fz)^n$
and $N$ be as in \eqref{2e11}.
Then, for any given $\phi\in\cs(\rn)$
with $\int_{\rn}\phi(x)\,dx\neq0$ and any $f\in\cs'(\rn)$,
the following statements are equivalent:
\begin{enumerate}
\item[{\rm(i)}] $f\in\vh;$
\item[{\rm(ii)}] $M_\phi^0(f)\in\lv.$
\end{enumerate}
Moreover, there exists a positive constant $C$ such that, for any $f\in\vh$,
\begin{align*}
\|f\|_{\vh}\le C\lf\|M_\phi^0(f)
\r\|_{\lv}\le C\|f\|_{\vh}.
\end{align*}
\end{lemma}

Observe that $(\rn,\,|\cdot|_{\va},\,dx)$ is an RD-space (see \cite{hmy08,zy11}). From this
and \cite[Lemma 4.6]{gly08} (see also \cite[lemma 4.5]{zsy16}), we deduce the following lemma,
the details being omitted.

\begin{lemma}\label{3l3}
Let $\Omega\st \rn$ be an open subset with $|\Omega|<\infty$ and, for any $x\in \rn$, let
$$d(x,\boz):=\inf\{|x-y|_{\vec{a}}:\ y\notin \boz\}.$$
Then there exists a sequence $\{x_k\}_{k\in\nn}\st \boz$ such that,
for any $A\in [1,\fz)$ and
$$r_k:=d(x_k,\boz)/(2A)\hspace{0.2cm}  with\hspace{0.2cm} k\in \nn,$$
it holds true that
\begin{enumerate}
\item[\rm(i)] $\boz=\bigcup_{k\in\nn} B_{\va}(x_k,r_k)$;

\item[\rm(ii)] $\{B_{\va}(x_k,r_k/4)\}_{k\in\nn}$
are pairwise disjoint;

\item[\rm(iii)] for any given $k\in \nn$,
$B_{\va}(x_k,Ar_k)\st \boz$;

\item[\rm(iv)] $Ar_k<d(x,\boz)<3Ar_k$ whenever
$k\in \nn$ and $x\in B_{\va}(x_k,Ar_k)$;

\item[\rm(v)] for any given $k\in \nn$,
there exists a $y_k\notin \boz$ such that
$|x_k-y_k|_{\vec{a}}<3Ar_k$;

\item[\rm(vi)] there exists a positive constant $R$, independent of $\Omega$,
such that, for any $k\in\nn$, $$\sharp \lf\{j\in\nn:\ \lf[B_{\va}(x_k,r_k)
\bigcap B_{\va}(x_j,Ar_j)\r]\neq\emptyset\r\}\le R.$$
\end{enumerate}
\end{lemma}

Let $\Phi$ be some fixed $C^\fz(\rn)$ function satisfying $\supp \Phi\st B(\vec{0}_n,1)$
and $\int_{\rn} \Phi(x)dx \neq 0$. For any $f\in\cs'(\rn)$ and $x\in \rn$, we always let
\begin{equation}\label{3e16}
M_0(f)(x):=M_\Phi^0(f)(x),
\end{equation}
where $M_\Phi^0(f)$ is as in Definition \ref{3d4} with $\phi$ replaced by $\Phi$.
In what follows, for any given $s\in \mathbb{Z}_+$,
the \emph{symbol $\cp_s(\rn)$} denotes the linear space of all polynomials
on $\rn$ with degree not greater than $s$.

The following Calder\'on-Zygmund decomposition extends the corresponding results of
Stein \cite[p.\,101, Proposition]{s93} and Grafakos \cite[Theorem 5.3.1]{lg14}
as well as Sawano et al. \cite[Lemma 2.23]{shyy17} to the present setting.

\begin{lemma}\label{3l4}
Let $\va \in [1,\fz)^n$, $\vp\in(0,\fz)^n$,
$s\in \mathbb{Z}_+$ and
$N$ be as in \eqref{2e11}.
For any $\sa\in (0,\fz)$ and $f\in \vh$,
let $$\CO := \{x\in \rn:\ M_N(f)(x)>\sa\},$$
where $M_N$ is as in Definition \ref{2d4}. Then
the following statements hold true:
\begin{enumerate}
\item[\rm(i)] There exists a sequence
$\{B_k^*\}_{k\in\nn}\subset \mathfrak{B}$
with $\mathfrak{B}$ as in \eqref{2e2},
which has finite intersection property,
such that $$\CO = \bigcup_{k\in \nn}\Qkk.$$

\item[\rm(ii)] There exist two distributions $g$ and $b$
such that $f=g+b$ in $\cs'(\rn)$.

\item[\rm(iii)] For the distribution $g$ as in (ii) and any $x\in \rn$,
\begin{equation}\label{3e6}
M_0(g)(x)\ls M_N(f)(x)\chi_{\CO^{\com}}(x)+\sum_{k\in\nn}
\f{\sa r_k^{\nu+(s+1)a_-}}{(r_k+|x-x_k|_{\va})^{\nu+(s+1)a_-}},
\end{equation}
where $a_-$ and $M_0$ are as in \eqref{2e9}, respectively, \eqref{3e16}, and the implicit positive
constant is independent of $f$ and $g$.
Moreover, for any $k\in \nn$, $x_k$ denotes
the center of $B_k^*$ and there exists a constant
$A^*\in (1,\fz)$, independent of $k$, such that $A^*-1$ is small enough
and $A^* r_k$ equals the radius of $B_k^*$.

\item[\rm(iv)] If $f\in L_{\loc}^1(\rn)$, then the distribution $g$ as in (ii)
belongs to $L^{\fz}(\rn)$
and $\|g\|_{L^{\fz}(\rn)} \ls \sa$ with the implicit positive
constant independent of $f$ and $g$.

\item[\rm(v)] If $s$ is as in \eqref{3e1} and $b$ as in (ii), then $b=\sum_{k\in\nn}b_k$ in $\cs'(\rn)$,
where, for each $k\in \nn,~b_k:=(f-c_k)\eta_k$, $\{\eta_k\}_{k\in \nn}$
is a partition of unity with respect to $\{B_k^*\}_{k\in \nn}$,
namely, for any $k\in\nn$, $\eta_k\in C_c^{\fz}(\rn)$,
$\supp \eta_k\st B_k^*$, $0\le\eta_k \le 1$
and $$\chi_{\CO}=\sum_{k\in \nn}\eta_k,$$ and
$c_k\in \cp_s(\rn)$ is a polynomial such that, for any $q\in\cp_s(\rn)$,
$$\langle f-c_k,q\eta_k\rangle=0.$$
Moreover, for any $k\in \nn$ and $x\in \rn$,
\begin{equation}\label{3e7}
M_0(b_k)(x)\ls M_N(f)(x)\chi_{B_k^*}(x)+\f{\sa r_k^{\nu+(s+1)a_-}}
{|x-x_k|_{\va}^{\nu+(s+1)a_-}}\chi_{({B_k^*})^{\com}}(x),
\end{equation}
where $a_-$ and $M_0$ are as in \eqref{2e9}, respectively, \eqref{3e16}, and the implicit
positive constant is independent of $f$ and $k$.
\end{enumerate}
\end{lemma}

To show Lemma \ref{3l4}, we need an auxiliary inequality as follows,
which is a slight modification of \cite[p.\,100, (21)]{s93},
the details being omitted.

\begin{lemma}\label{3l5}
Let $\va\in [1,\fz)^n$ and $N\in \nn$. Assume that $\varphi$ is a function supported on
$B\in\mathfrak{B}$ with $\mathfrak{B}$ as in \eqref{2e2} and satisfies that, for each multi-index
$\bz\in \mathbb{Z}_+^n $ with $|\bz|\le N$, $|\pa^{\bz}\varphi|\le C_{(N)}r^{-\nu-\va \cdot \bz}$,
where $r$ is the radius of $B$ and $C_{(N)}$ a positive constant independent of $B$,
but depending on $N$. Then there exists a positive constant $C_{(N)}$, depending on $N$, such that, for any
$f\in \cs'(\rn)$ and $x\in B$,
$$\lf|\langle f,\varphi \rangle\r|\le C_{(N)} M_N (f)(x),$$
where $M_N$ is as in Definition \ref{2d4}.
\end{lemma}

Now we prove Lemma \ref{3l4}.
\begin{proof}[Proof of Lemma \ref{3l4}]
We prove this lemma by five steps.

\emph{Step 1.} In this step, we show (i). To this end, applying Lemma \ref{3l3} to
$\CO$ with $\CO$ as in Lemma \ref{3l4}, we obtain a collection of cubes, $\{B_k\}_{k\in \nn} \subset\mathfrak{B}$ with $\mathfrak{B}$ as in \eqref{2e2},
which has finite intersection property, such that $\CO=\bigcup_{k\in \nn}B_k$. Next, fix two numbers
$\de$ and $A^*$ satisfying $1<\de<A^*<\fz$. For any $k\in\nn$, let $\Qk:=B_k^{(\de)}$
and $\Qkk:=B_k^{(A^*)}$, then $$B_k\subset \Qk \subset \Qkk,$$
where $B_k^{(\de)}$ and $B_k^{(A^*)}$ are as in \eqref{2e2'} with $\dz$ replaced, respectively, by
$\de$ and $A^*$.
Moreover, by choosing $A^*$ sufficiently close to $1$ such that $\CO=\bigcup_{k\in \nn}\Qkk$
and $\{\Qkk\}_{k\in \nn}$ also has finite intersection property which is guaranteed by Lemma \ref{3l3}(vi), we know that (i) holds true.

\emph{Step 2.} In this step, we prove (iv). Let $\zeta\in C^{\fz}(\rn)$ satisfy that
$\supp\zeta\subset B_{\va}(\vec0_n,\de)$, $0\le\zeta\le1$ and $\zeta\equiv1$ on $B(\vec0_n,1)$.
In addition, for any $k\in\nn$ and $x\in\rn$, let $\zeta_k(x):=\zeta(r_k^{-\va}(x-x_k))$,
where $x_k$ denotes the center of $B_k$ and $r_k$ its radius, and
\begin{align}\label{3e8}
\eta_k(x):=
\frac {\zeta_k(x)}{\Sigma_{k\in\nn}\zeta_k(x)}.
\end{align}
Then it is easy to see that $\{\eta_k\}_{k\in\mathbb{N}}$ forms a smooth partition of unity
of $\CO$, namely, for any $k\in \nn$, $\eta_k\in C_c^{\fz}(\rn)$, $\supp\eta_k\subset\Qk$, $0\le \eta_k\le 1$
and $$\chi_{\CO}=\sum_{k\in \nn}\eta_k.$$

Notice that, for any multi-index $\bz\in \mathbb{Z}_+^n $, $k\in \nn$ and $x\in \rn$,
\begin{equation}\label{3e18}
\lf|\pa^{\beta}\eta_k(x)\r|\ls r_k^{-\beta\cdot\va}
\end{equation}
and $\int_{\rn}\eta_k(x)\,dx\thicksim |\Qk|\thicksim r_k^{\nu}\in (0,\fz)$.
Due to this, for any $k\in\nn$ and $x\in \rn$, let
$$\widetilde{\eta}_k(x):=\f{\eta_k}{\int_{\rn}\eta_k(y)\,dy}.$$

On the other hand, for any $k\in \nn$, let $\widetilde{b}_k:=(f-\widetilde{c}_k)\eta_k$,
where $\widetilde{c}_k$ is a constant determined by the requirement
that $\int_{\rn}\widetilde{b}_k(x)\,dx=0$, namely, $\widetilde{c}_k=\langle f,\widetilde{\eta}_k\rangle$.
We then conclude that, for any $k\in \nn$,
\begin{align}\label{3e9}
\lf|\widetilde{c}_k\r|\ls \sigma.
\end{align}
Indeed, for any $k\in \nn$, by Lemma \ref{3l3}(iv), we find that there exists some
$\widetilde{x}\in B_k^{(3A)}\cap\CO^{\com}$ with $A\in[1,\fz)$ as in Lemma \ref{3l3}.
Then \eqref{3e9} follows from Lemma \ref{3l5} with $\varphi:=\widetilde{\eta}_k$
and $B:=B_k^{(3A)}$. Similarly, from the fact that $\Qkk\st\CO$
and Lemma \ref{3l5}, we deduce that, for any $N,\,k\in \nn$ and $x\in \Qkk$,
\begin{align}\label{3e10}
|\widetilde{c}_k|\ls M_N(f)(x).
\end{align}
Now, for any $x\in \rn$, let $$g(x):=f(x)\chi_{\CO^{\com}}(x)+\sum_{k\in \nn}\widetilde{c}_k\eta_k (x).$$
If $x\in \CO^{\com}$, then $g(x)=f(x)$ and $M_N(f)(x)\le \sigma$, which imply that $|g(x)|\ls \sigma$;
if $x\in \CO$, then, by \eqref{3e9} and the finite intersection property of $\{\widetilde{B}_{k}\}_{k\in\nn}$,
we know that $|g(x)|\ls \sigma$. Thus, $g\in L^{\fz}(\rn)$
and $\|g\|_{L^{\fz}(\rn)} \ls \sa$, which completes the proof of (iv).

\emph{Step 3.} In this step, we show (v).
For any $k\in \nn$, let $b_k:=(f-c_k)\eta_k$ with $\eta_k$ as in \eqref{3e8}, where $c_k\in\cp_s(\rn)$
such that, for any $q\in\cp_s(\rn)$,$$\langle f-c_k,q\eta_k\rangle=0.$$
The existence of $c_k$ can be verified by an argument similar to that used in \cite[p.\,104]{s93}.

Now we establish the following two estimates, namely,
for any $N,\,k\in \nn$ and $x\in \Qkk$,
\begin{equation}\label{3e11}
M_0(b_k)(x)\ls M_N(f)(x)
\end{equation}
and, for any $k\in \nn$ and $x\in {(B_k^*)}^{\com}$,
\begin{equation}\label{3e12}
M_0(b_k)(x)\ls \f{\sa r_k^{\nu+(s+1)a_-}}
{|x-x_k|_{\va}^{\nu+(s+1)a_-}}.
\end{equation}
To this end, we first claim that,
for any $k\in \nn$ and $x\in \rn$,
\begin{align}\label{3e13}
|c_k(x)\eta_k(x)|\ls \sigma
\end{align}
and, for any $N,\,k\in \nn$ and $x\in \Qkk$,
\begin{align}\label{3e14}
|c_k(x)\eta_k(x)|\ls M_N(f)(x).
\end{align}
Indeed, for any $k\in \nn$
and $q\in \cp_s(\rn)$, if we define
\begin{align*}
\|q\|_{\mathcal{H}}:=
\lf[\frac1{\int_\rn\eta_{k}(x)\,dx}\int_\rn
|q(x)|^2\eta_{k}(x)\,dx\r]^{1/2},
\end{align*}
then, by an argument similar to that used in the proof of \cite[p.\,104, (28)]{s93}, we find that,
for any multi-index $\bz\in\mathbb{Z}_+^n$, $k\in\nn$ and $q\in\cp_s(\rn)$,
\begin{align}\label{3e15}
\sup_{x\in \Qkk}\lf|\pa^{\bz} q(x)\r|\ls r_k^{-\bz\cdot\va}\|q\|_\mathcal{H},
\end{align}
where $\mathcal{H}$ denotes the Hilbert space $L^2(\Qkk,\widetilde{\eta}_kdx)$.
By \eqref{3e15} and an argument similar to that used in the estimations of \eqref{3e9} and \eqref{3e10},
we easily know that \eqref{3e13} and \eqref{3e14} hold true.

Let $k\in \nn$, $x\in \Qkk$, $t\in (0,\fz)$ and $\varphi(\cdot):=\eta_k(\cdot)\Phi_t(x-\cdot)$,
where $\Phi$ is as in \eqref{3e16}.
Then $(f\eta_k)*\Phi_t(x)=\langle f,\varphi\rangle$. When $t\in (0,r_k]$, from Lemma \ref{3l5} with $B:=B_{\va}(x,t)$
and the fact that, for any $y\in \rn$, $|\pa^\beta\varphi(y)|\ls t^{-\nu-\beta\cdot\va}$, it follows that
$\lf|(f\eta_k)*\Phi_t(x)\r|=|\langle f,\varphi\rangle|\ls_{(N)} M_N(f)(x)$ for any $N\in\nn$,
here and hereafter, the symbol $\ls_{(N)}$ means that the implicit positive constant may depend on $N$. When $t\in (r_k,\fz)$,
notice that, for any $y\in \rn$,
$|\pa^\beta\varphi(y)|\ls r_k^{-\nu-\beta\cdot\va}$. Then, by Lemma \ref{3l5} with $B:=B_{\va}(x,cr_k)$,
where $c$ is a positive constant large enough such that $\Qkk\st B_{\va}(x,cr_k)$,
we conclude that $\lf|(f\eta_k)*\Phi_t(x)\r|=|\langle f,\varphi\rangle|\ls_{(N)} M_N(f)(x)$ for any $N\in\nn$. Thus,
for any $N,\,k\in \nn$ and $x\in \Qkk$,
$$M_0(f\eta_k)(x)=\sup_{t\in (0,\fz)}
|f\eta_k\ast\Phi_t(x)|\ls_{(N)} M_N(f)(x).$$
On the other hand, by \eqref{3e16} and \eqref{3e14}, we find that, for any $N,\,k\in \nn$ and $x\in \Qkk$,
$M_0(c_k\eta_k)(x)\ls_{(N)} M_N(f)(x)$. Therefore, for any $N,\,k\in \nn$ and $x\in \Qkk$,
$$M_0(b_k)(x)\le M_0(f\eta_k)(x)+M_0(c_k \eta_k)(x)\ls_{(N)} M_N(f)(x),$$
which implies that \eqref{3e11} holds true.

Let $\Phi$ be as in \eqref{3e16}, $k\in\nn$ and $x\in(\Qkk)^{\com}$.
Then, by the fact that, for any $q\in \cp_s(\rn)$, $\langle b_k,q \rangle=0$, we know that, for any $k\in\nn$ and $t\in(0,\fz)$,
\begin{align}\label{3e3}
b_k*\Phi_t(x)&=b_k*\Phi_t(x)-\langle b_k,q_0 \rangle\\
&=\lf[f\eta_k*\Phi_t(x)-\langle f\eta_k,q_0 \rangle\r] -\lf[c_k\eta_k*\Phi_t(x)
-\langle c_k\eta_k,q_0 \rangle\r]=:\textrm{I}-\textrm{II},\noz
\end{align}
where $q_0$ is the Taylor expansion of $\Phi_t$ at the point $x-x_k$ of order $s$ and $x_k$
denotes the center of $B_k^*$.
For any $x\in(\Qkk)^{\com}$, let $\widetilde{\varphi}(\cdot):=\eta_k(\cdot)[\Phi_t(x-\cdot)-q_0(\cdot)]$.
Then, from the Taylor remainder theorem,
we deduce that, for any $k\in\nn$, $t\in(0,\fz)$ and $y\in \rn$,
\begin{align}\label{3e29}
\lf|\widetilde{\varphi}(y)\r|&=\lf|t^{-\nu}\eta_k(y)\lf[\Phi\lf(\dfrac{x-y}{t^{\va}}\r)-\sum_{|\az|\le s}
\dfrac{\pa^\az\Phi(\dfrac{x-x_k}{t^{\va}})}
{\az!}\lf(\dfrac{x_k-y}{t^{\va}}\r)^\az\r]\r|\\
&\ls \lf|t^{-\nu}\eta_k(y)
\sum_{|\az|=s+1}\pa^\az\Phi
\lf(\dfrac{\xi}{t^{\va}}\r)
\lf(\dfrac{x_k-y}{t^{\va}}\r)^\az\r|,\noz
\end{align}
where $\xi:=x-x_k+\theta(x_k-y)$ for some $\theta\in [0,1]$.

Recall that, for any $k\in\nn$, $\supp\eta_k\subset\Qk$ and $\Qk\subsetneqq\Qkk$.
Thus, for any $k\in \nn$ and $y\in \rn$, if $x\in(\Qkk)^{\com}$ and $\eta_k(y)\neq 0$,
then, it holds true that $|x-y|_{\va}\sim|x-x_k|_{\va}\gtrsim r_k$, which, combined with Lemma \ref{2l2}(ii), implies that
$|\xi|_{\va}\geq|x-x_k|_{\va}-|x_k-y|_{\va}\gtrsim |x-x_k|_{\va}$.
By this, the fact $|\xi/t^{\va}|_{\va}<1$ [which is deduced from $\widetilde{\varphi}(\cdot)\neq 0$
and $\supp \Phi\st B(\vec{0}_n,1)$], and Lemma \ref{2l2}(i), we conclude that $t\gs |x-x_k|_{\va}$.
Thus, we claim that, for any multi-index $\bz\in \mathbb{Z}_+^n $, $k\in \nn$ and $y\in \rn$,
\begin{equation}\label{3e17}
\lf|\pa^\bz \widetilde{\varphi}(y)\r| \ls \f{r_k^{\nu+(s+1)a_-}}
{|x-x_k|_{\va}^{\nu+(s+1)a_-}} r_k^{-\nu-\va\cdot\bz}.
\end{equation}
Indeed, by \eqref{3e29}, for any multi-index $\bz\in \mathbb{Z}_+^n $, $k\in \nn$ and $y\in \rn$,
\begin{align}\label{3e30}
\lf|\pa^\bz \widetilde{\varphi}(y)\r| \ls& \lf|t^{-\nu}\pa^\bz
\eta_k(y)\sum_{|\az|=s+1}\pa^\az\Phi
\lf(\dfrac{\xi}{t^{\va}}\r)
\lf(\dfrac{x_k-y}{t^{\va}}\r)^\az\r|
+\lf|t^{-\nu}\eta_k(y)
\sum_{|\az|=s+1}\pa^{\az+\bz}\Phi
\lf(\dfrac{\xi}{t^{\va}}\r)
\lf(\dfrac{x_k-y}{t^{\va}}\r)^\az\r|\\
&+\lf|t^{-\nu}\eta_k(y)
\sum_{|\az|=s+1}\pa^\az\Phi
\lf(\dfrac{\xi}{t^{\va}}\r)
\pa^\bz\lf(\dfrac{x_k-y}{t^{\va}}\r)^\az\r|
=:\textrm{I}_1+\textrm{I}_2+\textrm{I}_3.\noz
\end{align}
Then, by \eqref{3e18}, $\Phi\in \cs(\rn)$, the fact that $|\xi|_{\va}\gtrsim |x-x_k|_{\va}$,
and (i), (v) and (vi) of Lemma \ref{2l2}, we find that, for any $K\in \mathbb{Z}_+$,
\begin{align}\label{3e31}
\textrm{I}_1 &\ls t^{-\nu}r_k^{-\va\cdot\bz}
\dfrac1{(1+|\frac{\xi}{t^{\va}}|)^K}
\lf|\dfrac{x_k-y}{t^{\va}}\r|^{s+1}
\ls t^{-\nu}r_k^{-\va\cdot\bz}\lf(\frac{t}
{|\xi|_{\va}}\r)^{K a_- } \lf(\frac{r_k}{t}\r)^{(s+1)a_-}\\
&\ls t^{-\nu}r_k^{-\va\cdot\bz}\lf(\dfrac{t}
{|x-x_k|_{\va}}\r)^{ K a_-} \lf(\frac{r_k}{t}\r)^{(s+1)a_-}.\noz
\end{align}
Let $K:=s+1$. From \eqref{3e31} and the fact that $t\gs |x-x_k|_{\va}$, we further deduce that
\begin{align}\label{3e32}
\textrm{I}_1
\ls \f{r_k^{\nu+(s+1)a_-}}
{|x-x_k|_{\va}^{\nu+(s+1)a_-}} r_k^{-\nu-\va\cdot\bz}.
\end{align}
Similarly, we conclude that, for any $i\in\{2,3\}$,
\begin{align*}
\textrm{I}_i
\ls \f{r_k^{\nu+(s+1)a_-}}
{|x-x_k|_{\va}^{\nu+(s+1)a_-}} r_k^{-\nu-\va\cdot\bz}.
\end{align*}
This, together with \eqref{3e30} and \eqref{3e32}, implies that \eqref{3e17} holds true. For any $k\in \nn$,
from Lemma \ref{3l3}(iv), it is easy to see that $B_k^{(3A)}\cap\CO^{\com}\neq\emptyset$.
Thus, by Lemma \ref{3l5} with $B:=B_k^{(3A)}$ and \eqref{3e17}, we know that, for any $x\in(\Qkk)^{\com}$,
\begin{equation}\label{3e19}
\lf|\textrm{I}\r|=\lf|\langle f,\widetilde{\varphi}\rangle\r|\ls \f{\sa r_k^{\nu+(s+1)a_-}}
{|x-x_k|_{\va}^{\nu+(s+1)a_-}}.
\end{equation}
On the other hand, from \eqref{3e13}, \eqref{3e17} with $\bz=\vec0_n$ and the fact that $0\le \eta_k\le 1$, it follows that, for any $x\in(\Qkk)^{\com}$,
\begin{align*}
|\textrm{II}|\le\int_{\rn}\lf|c_k(y)\eta_k(y)\r|\lf|\lf[\Phi_t(x-y)-q_0(y)\r]\r|\,dy
\ls \f{\sa r_k^{\nu+(s+1)a_-}}{|x-x_k|_{\va}^{\nu+(s+1)a_-}}\int_{\Qk}r_k^{-\nu}\,dy\sim
\f{\sa r_k^{\nu+(s+1)a_-}}{|x-x_k|_{\va}^{\nu+(s+1)a_-}},
\end{align*}
which, combined with \eqref{3e19}, \eqref{3e3}, \eqref{3e16} and Definition \ref{3d4}, further implies that \eqref{3e12} holds true.
By this and \eqref{3e11}, we find that \eqref{3e7} holds true.

Next we show the series $\{\sum_{k=1}^r b_k\}_{r\in\nn}$ converges in $\cs'(\rn)$.
Indeed, by Lemma \ref{3l8}, \eqref{3e7} and \eqref{2e8}, we conclude that,
for any given $N$ as in \eqref{2e11} and any $k\in \nn$,
\begin{equation*}\begin{aligned}
\lf\|b_k\r\|_{\vh} &\sim \lf\|M_0(b_k)\r\|_{\lv}\ls \lf\|M_N(f)\chi_{B_k^*}\r\|_{\lv}+
\sa\lf\|\f{ r_k^{\nu+(s+1)a_-}}
{|\cdot-x_k|_{\va}^{\nu+(s+1)a_-}}
\chi_{(B_k^*)^{\com}}\r\|_{\lv}\\
&\ls \lf\|M_N(f)\chi_{B_k^*}\r\|_{\lv}+\sa\lf\|
M_{\rm HL}(\chi_{\Qkk})\r\|_{L^{\frac{\nu+(s+1)a_-}
{\nu}\vp}(\rn)}^{\frac{\nu+(s+1)a_-}{\nu}}.
\end{aligned}\end{equation*}
Notice that $\widetilde{p}_->\frac{\nu}{\nu+(s+1)a_-}$ with $\widetilde{p}_-$
as in \eqref{3e1}. Then, by Lemma \ref{3l1} and \eqref{2e8}, we find that,
for any given $N$ as in \eqref{2e11} and any $k\in \nn$,
\begin{equation*}
\lf\|b_k\r\|_{\vh} \ls \lf\|M_N(f)\chi_{B_k^*}\r\|_{\lv}+\sa\lf\|
\chi_{\Qkk}\r\|_{\lv} \ls \lf\|M_N(f)\chi_{B_k^*}\r\|_{\lv},
\end{equation*}
which, combined with \cite[p.\,304, Theorem 2]{bp61}, implies that there exists some $g\in L^{\vp'}(\rn)$
with $\|g\|_{L^{\vp'}(\rn)}\le 1$ such that, for any $m,\,p\in \nn$,
\begin{align*}
\lf\|\sum_{k=m}^{m+p}b_k\r\|_{\vh}
\ls& \sum_{k=m}^{m+p}\lf\|b_k\r\|_{\vh}
\ls\sum_{k=m}^{m+p}\lf\|M_N(f)\chi_{B_k^*}\r\|_{\lv}\\
\ls&\int_{\rn}\lf|\sum_{k=m}^{m+p}M_N(f)(x)\chi_{B_k^*}(x)g(x)\r|\,dx
\ls \lf\|\sum_{k=m}^{m+p}M_N(f)\chi_{B_k^*}\r\|_{\vh}.
\end{align*}
Thus,
\begin{equation}\label{3e20}
\lf\|\sum_{k\in\nn}b_k\r\|_{\vh} \ls
\lf\|\sum_{k\in\nn}M_N(f)\chi_{B_k^*}\r\|_{\lv}
\ls \lf\|M_N(f)\chi_{\CO}\r\|_{\lv}\ls \|f\|_{\vh}<\fz,
\end{equation}
which implies that the series $\{\sum_{k=1}^r b_k\}_{r\in\nn}$ converges in $\vh$
and hence converges in $\cs'(\rn)$.
This finishes the proof of (v).

\emph{Step 4.} In this step, we prove (ii). Let $b:=\sum_{k\in\nn}b_k$ in $\cs'(\rn)$,
which is well defined by Step 3. From this, it further follows that the distribution
\begin{align}
\label{3e33}g:=f-b:=f-\sum_{k\in\nn}b_k
\end{align}
is well defined. Thus, $f=g+b$ in $\cs'(\rn)$,
which completes the proof of (ii).

\emph{Step 5.} In this step, we show (iii).
If $x\in \CO^{\com}$, then, from (v) and the facts that $M_0(f)\ls M_N(f)$ and
$|x-x_k|_{\va}\gs r_k$, we deduce that, for any given $N$ as in \eqref{2e11},
\begin{align*}
M_0(g)(x)\le M_0(f)(x)+\sum_{k\in \nn}M_0(b_k)(x)\ls M_N(f)(x)+\sum_{k\in\nn}
\f{\sa r_k^{\nu+(s+1)a_-}}{(r_k+|x-x_k|_{\va})^{\nu+(s+1)a_-}},
\end{align*}
which implies that \eqref{3e6} holds true when $x\in \CO^{\com}$.

If $x\in \CO$, then there exists some $m\in \nn$ such that $x\in B_m$. Let
$$E_1:=\{k\in\nn:\ \Qkk\cap B_m^*\neq\emptyset\}\quad{\rm and}\quad
E_2:=\{k\in\nn:\ \Qkk\cap B_m^*=\emptyset\}.$$
By \eqref{3e33}, we rewrite $$g=\lf(f-\sum_{k\in E_1}b_k\r)-\sum_{k\in E_2}b_k.$$
If $k\in E_2$, then $x\in B_m\cap(\Qkk)^{\com}$
and $|x-x_k|_{\va}\gs r_k$. By this, \eqref{3e12} and the fact that $\{\sum_{k=1}^r b_k\}_{r\in\nn}$ converges in $\cs'(\rn)$,
we conclude that, for any $x\in B_m$,
\begin{align}\label{3e34}
M_0\lf(\sum_{k\in E_2}b_k\r)(x)\ls\sum_{k\in E_2}
\f{\sa r_k^{\nu+(s+1)a_-}}{(r_k+|x-x_k|_{\va})^{\nu+(s+1)a_-}}.
\end{align}

Notice that
\begin{align}\label{3e35}
f-\sum_{k\in E_1}b_k=
\lf(f-\sum_{k\in E_1}f \eta_k\r)-\sum_{k\in E_1}c_k\eta_k
=:L-\sum_{k\in E_1}c_k\eta_k.
\end{align}
For the second term, by the finite intersection property of $\{\Qkk\}_{k\in \nn}$, we know that there exists
a positive constant $R$, independent of $m$, such that $\sharp E_1\le R$. Then, by \eqref{3e13} and the fact that
$x\in B_m$, we have
$$M_0\lf(\sum_{k\in E_1}c_k\eta_k\r)(x)\le
\sum_{k\in E_1}M_0\lf(c_k\eta_k\r)(x)\ls\sa\sim
\f{\sa r_m^{\nu+(s+1)a_-}}{(r_m+|x-x_m|_{\va})^{\nu+(s+1)a_-}}.$$
Finally, we estimate $M_0(L)(x)$ by
considering $L*\Phi_t(x)$ with $\Phi$ as in \eqref{3e16} and $t\in (0,\fz)$.
If $t\in (0,c_0 {\rm dist}(B_m,\CO^\com)]$,
where $${\rm dist}(B_m,\CO^\com):=\inf\lf\{|x-y|_{\va}:\ x\in B_m,~y\in \CO^\com\r\}$$
and $c_0\in (0,\fz)$ is small enough such that, for any $x\in B_m$ and
$t\in (0,c_0 {\rm dist}(B_m,\CO^\com)]$, $B_{\va}(x,t)\subset \widetilde{B}_m$, then,
for any $x\in B_m$, from the fact that $1-\sum_{k\in E_1}\eta_k(x)=1-\chi_{\CO}(x)=0$
and $\supp\Phi_t\st B_{\va}(0,t)$, we deduce that
\begin{align}\label{3e36}
L*\Phi_t(x)=\lf[f\lf(1-\sum_{k\in E_1}\eta_k\r)\r]*\Phi_t(x)=0.
\end{align}
On the other hand, if $t\in (c_0 {\rm dist}(B_m,\CO^\com),\fz)$, then, for any $x\in B_m$,
let $$\psi(\cdot):=\lf[1-\sum_{k\in E_1}\eta_k(\cdot)\r]\Phi_t(x-\cdot).$$
Obviously, $L*\Phi_t(x)=\langle f,\psi \rangle$.
By this and an argument similar to that used in the estimation of \eqref{3e11},
we find that, for any $t\in (c_0 {\rm dist}(B_m,\CO^\com),\fz)$ and $x\in B_m$,
$$\lf|L*\Phi_t(x)\r|=\lf|\langle f,\psi \rangle\r|\ls\sa\sim
\f{\sa r_m^{\nu+(s+1)a_-}}{(r_m+|x-x_m|_{\va})^{\nu+(s+1)a_-}}.$$
This, together with \eqref{3e35}, \eqref{3e36} and \eqref{3e34}, implies that \eqref{3e6} holds true, which
completes the proof of (iii) and hence of Lemma \ref{3l4}.
\end{proof}

From Lemma \ref{3l4} and its proof, we deduce the  following result on the density,
which plays a key role in the proof of Theorem \ref{3t1} below.

\begin{lemma}\label{3l7}
Let $\va\in [1,\fz)^n$, $\vp\in(0,\fz)^n$ with $\widetilde{p}_-$ as
in \eqref{3e1} and $N$ be as in \eqref{2e11}.
Then $\vh\cap L^{\vp/\widetilde{p}_-}(\rn)$ is dense in $\vh$.
\end{lemma}
\begin{proof}
Let all the notation be the same as those used in the proof of Lemma \ref{3l4}.
For any $f\in\vh$, by (ii) and (v) of Lemma \ref{3l4}, we know that
$$f=g+b=g+\sum_{k\in \nn} b_k\quad\mathrm{in}\quad\cs'(\rn),$$
where $g,\,b$ and $\{b_k\}_{k\in\nn}$ are as in Lemma \ref{3l4}.
By \eqref{3e20}, we have
\begin{equation*}
\lf\|b\r\|_{\vh}=
\lf\|\sum_{k\in \nn}b_k\r\|_{\vh}
\ls \lf\|M_N(f)\chi_{\CO}\r\|_{\lv}\to 0
\quad{\rm as}\quad \sigma\to\fz,
\end{equation*}
where $\CO$ is as in Lemma \ref{3l4}.
Thus, for any $\varepsilon\in (0,\fz)$,
there exists some $\sa\in (0,\fz)$ such that
$$\|f-g\|_{\vh}=\|b\|_{\vh}<\varepsilon,$$
which, combined with the fact that $f\in\vh$, implies that $g\in\vh$. Therefore, to
complete the proof of Lemma \ref{3l7}, it suffices to show that $g\in L^{\vp/\widetilde{p}_-}(\rn)$.
To this end, by \cite[Theorem 6.1]{cgn17}, Definition \ref{2d5}, Lemma \ref{3l8} and \eqref{3e16},
we only need to prove that $M_0(g)\in L^{\vp/\widetilde{p}_-}(\rn)$. Indeed, by
\eqref{3e6}, \eqref{2e8}, Lemma \ref{3l2}, the finite intersection property of $\{B_k^*\}_{k\in\nn}$ and Lemma \ref{3l4}(i),
we conclude that
\begin{equation*}\begin{aligned}
\lf\|M_0(g)\r\|_{L^{\vp/{\widetilde{p}_-}}(\rn)}&\ls
\lf\|M_N(f)\chi_{\CO^{\com}}\r\|_{L^{\vp/{\widetilde{p}_-}}(\rn)}+
\sa\lf\|\sum_{k\in\nn}\f{ r_k^{\nu+(s+1)a_-}}
{(r_k+|\cdot-x_k|_{\va})^{\nu+(s+1)a_-}}\r\|_{L^{\vp/{\widetilde{p}_-}}(\rn)}\\
&\ls \sa^{1-{\widetilde{p}_-}}\lf\|M_N(f)\chi_{\CO^{\com}}\r\|_{\lv}^{\widetilde{p}_-}+
\sa\lf\|\sum_{k\in\nn}\lf[M_{\rm HL}(\chi_{\Qkk})
\r]^{\f{\nu+(s+1)a_-}{\nu}}\r\|_{L^{\vp/{\widetilde{p}_-}}(\rn)}\\
&\ls \sa^{1-{\widetilde{p}_-}}\lf\|M_N(f)\r\|_{\lv}^{\widetilde{p}_-}+
\sa\lf\|\lf\{\sum_{k\in\nn}\lf[M_{\rm HL}(\chi_{\Qkk})
\r]^{\f{\nu+(s+1)a_-}{\nu}}\r\}^{\f{\nu}{\nu+(s+1)a_-}}\r\|_{L^{{
\f{\nu+(s+1)a_-}{\nu}}\f{\vp}{{\widetilde{p}_-}}}(\rn)}^{\f{\nu+(s+1)a_-}{\nu}}\\
&\ls \sa^{1-{\widetilde{p}_-}}\lf\|M_N(f)\r\|_{\lv}^{\widetilde{p}_-}+
\sa\lf\|\sum_{k\in\nn}\chi_{\Qkk}\r\|_{L^{\vp/{\widetilde{p}_-}}(\rn)}\\
&\ls \sa^{1-{\widetilde{p}_-}}\lf\|M_N(f)\r\|_{\lv}^{\widetilde{p}_-}+
\lf\|M_N(f)\chi_{\CO}\r\|_{\lv}^{\widetilde{p}_-}<\fz.
\end{aligned}\end{equation*}
This implies that $M_0(g)\in L^{\vp/\widetilde{p}_-}(\rn)$ and hence
finishes the proof of Lemma \ref{3l7}.
\end{proof}

Via borrowing some ideas from the proofs of \cite[Proposition 2.11]{zsy16} and \cite[Theorem 1.1]{sawa13},
we obtain the following lemma, which is of independent interest.
\begin{lemma}\label{3l6}
Let $\va\in[1,\fz)^n$, $\vp\in(0,\fz)^n$, $\bz\in(0,\fz)$ and $r\in[1,\fz]\cap(p_+,\fz]$
with $p_+$ as in \eqref{2e10}. Assume that
$\{\lz_i\}_{i\in\nn}\st\mathbb{C}$, $\{B_i\}_{i\in\nn}\st\mathfrak{B}$
and $\{m_i\}_{i\in\nn}\st L^r(\rn)$ satisfy, for any $i\in\nn$,
$\supp m_i\st B_i^{(\bz)}$ with $B_i^{(\bz)}$ as in \eqref{2e2'},
\begin{align}\label{3e37}
\|m_i\|_{L^r(\rn)}
\le\frac{|B_i|^{1/r}}{\|\chi_{B_i}\|_{\lv}}
\end{align}
and
$$\lf\|\lf\{\sum_{i\in\nn}
\lf[\frac{|\lz_i|\chi_{B_i}}{\|\chi_{B_i}\|_{\lv}}\r]^
{{p_-}}\r\}^{1/{p_-}}\r\|_{\lv}<\fz.$$
Then
\begin{align}\label{3e38}
\lf\|\lf[\sum_{i\in\nn}\lf|\lz_im_i\r|^{{p_-}}\r]
^{1/{p_-}}\r\|_{\lv}
\le C\lf\|\lf\{\sum_{i\in\nn}
\lf[\frac{|\lz_i|\chi_{B_i}}{\|\chi_{B_i}\|_{\lv}}\r]^
{{p_-}}\r\}^{1/{p_-}}\r\|_{\lv},
\end{align}
where $p_-$ is as in \eqref{2e10} and $C$ a positive constant independent
of $\lz_i$, $B_i$ and $m_i$.
\end{lemma}
\begin{proof}
By \cite[p.\,304, Theorem 2]{bp61}, we find that there exists some $g\in L^{(\vp/{p_-})'}(\rn)$
with norm not greater than 1 such that
$$\lf\|\lf[\sum_{i\in\nn}\lf|\lz_im_i\r|^{{p_-}}\r]
^{1/{p_-}}\r\|_{\lv}^{p_-}=\lf\|\sum_{i\in\nn}\lf|\lz_im_i\r|^{{p_-}}
\r\|_{L^{\vp/{p_-}}(\rn)}\ls \int_{\rn}\sum_{i\in\nn}
\lf|\lz_im_i(x)\r|^{{p_-}}\lf|g(x)\r|\,dx.$$
Moreover, from the H\"{o}lder inequality and \eqref{3e37}, we deduce that,
for any $r\in[1,\fz]$,
\begin{equation*}\begin{aligned}
\int_{\rn}\sum_{i\in\nn}\lf|\lz_im_i(x)\r|^{{p_-}}\lf|g(x)\r|\,dx
&\le \sum_{i\in \nn}|\lz_i|^{p_-}\|m_i\|_{L^r (\rn)}^{p_-}
\lf\|g\r\|_{L^{(r/{p_-})'}(B_i^{(\bz)})}\\
&\le \sum_{i\in \nn}\frac{|\lz_i|^{p_-}|B_i|^{{p_-}/r}}
{\|\chi_{B_i}\|_{\lv}^{p_-}}\lf\|g\r\|_{L^{(r/{p_-})'}(B_i^{(\bz)})}\\
&\ls \sum_{i\in \nn}\frac{|\lz_i|^{p_-}|B_i^{(\bz)}|}
{\|\chi_{B_i}\|_{\lv}^{p_-}}\inf_{z\in B_i^{(\bz)}}
\lf[M_{\rm HL}\lf(|g|^{(r/{p_-})'}\r)(z)\r]^{1/{(r/{p_-})'}}\\
&\ls \int_{\rn} \sum_{i\in \nn}\frac{|\lz_i|^{p_-}
\chi_{B_i^{(\bz)}}(x)}{\|\chi_{B_i}\|_{\lv}^{p_-}}
\lf[M_{\rm HL}\lf(|g|^{(r/{p_-})'}\r)(x)\r]^{1/{(r/{p_-})'}}\,dx,
\end{aligned}\end{equation*}
which, together with the H\"{o}lder inequality [see Remark \ref{2r3}(iv)], Remark \ref{3r2}, Lemma \ref{3l1} and the fact that $r\in (p_+,\fz]$,
implies that
\begin{equation*}\begin{aligned}
\int_{\rn}\sum_{i\in\nn}\lf|\lz_im_i(x)\r|^{{p_-}}\lf|g(x)\r|\,dx
&\ls \lf\|\sum_{i\in \nn}\frac{|\lz_i|^{p_-}
\chi_{B_i}}{\|\chi_{B_i}\|_{\lv}^{p_-}}
\r\|_{L^{\vp/{p_-}}(\rn)}
\lf\|\lf[M_{\rm HL}\lf(|g|^{(r/{p_-})'}\r)\r]^{1/{(r/{p_-})'}}
\r\|_{L^{(\vp/{p_-})'}(\rn)}\\
&\ls \lf\|\lf\{\sum_{i\in\nn}
\lf[\frac{|\lz_i|\chi_{B_i}}{\|\chi_{B_i}
\|_{\lv}}\r]^{{p_-}}\r\}^{1/{p_-}}
\r\|_{L^{\vp}(\rn)}^{p_-}
\lf\|g\r\|_{L^{(\vp/{p_-})'}(\rn)}.
\end{aligned}\end{equation*}
From this and the fact that $\lf\|g\r\|_{L^{(\vp/{p_-})'}(\rn)}\le 1$, it follows that \eqref{3e38} holds true.
This finishes the proof of Lemma \ref{3l6}.
\end{proof}

Now we state the main result of this section as follows.

\begin{theorem}\label{3t1}
Let $\va\in [1,\fz)^n,\,\vp\in (0,\fz)^n,\,r\in(\max\{p_+,1\},\fz]$
with $p_+$ as in \eqref{2e10}, $N$ be as in \eqref{2e11} and $s$
as in \eqref{3e1}.
Then $\vh=\vah$ with equivalent quasi-norms.
\end{theorem}

\begin{remark}
By Proposition \ref{2r4'} below, we know that,
when $\va:=(a_1,\ldots, a_n)\in [1,\fz)^n$
and $\vp:=(\overbrace{p,\ldots,p}^{n\ \rm times})$,
where $p\in(0,1]$, Theorem \ref{3t1}
is just \cite[p.\,39, Theorem 6.5]{mb03}
with
\begin{align}\label{4e5}
A:=\left(
              \begin{array}{cccc}
                2^{a_1} & 0 & \cdots & 0\\
                0 & 2^{a_2} & \cdots & 0\\
                \vdots & \vdots& &\vdots \\
                0 & 0 & \cdots & 2^{a_n} \\
              \end{array}
            \right).
\end{align}
\end{remark}

Now we proof Theorem \ref{3t1}.
\begin{proof}[Proof of Theorem \ref{3t1}]
First, we show that
\begin{align}\label{3e21}
\vah\subset\vh.
\end{align}
To this end, for any $f\in\vah$, by Definition \ref{3d2}, we know that
there exist $\{\lz_i\}_{i\in\nn}\st\mathbb{C}$
and a sequence of $(\vp,r,s)$-atoms, $\{a_i\}_{i\in\nn}$,
supported, respectively, on
$\{B_i\}_{i\in\nn}\st\mathfrak{B}$ such that
\begin{align*}
f=\sum_{i\in\nn}\lz_ia_i
\quad\mathrm{in}\quad\cs'(\rn)
\end{align*}
and
\begin{align*}
\|f\|_{\vah}\sim
\lf\|\lf\{\sum_{i\in\nn}\lf[\frac{|\lz_i|
\chi_{B_i}}{\|\chi_{B_i}\|_{\lv}}\r]
^{p_-}\r\}^{p_-}\r\|_{\lv}.
\end{align*}
Let $M_0$ be as in \eqref{3e16}. Then, by Lemma \ref{3l8}, we have
$$\|f\|_{\vh}\sim\lf\|M_0(f)\r\|_{\lv}.$$
Thus, to prove \eqref{3e21}, it suffices
to show that
\begin{align}\label{3e22}
\lf\|M_0\lf(\sum_{i\in\nn}\lz_ia_i\r)\r\|_{\lv}\ls \lf\|\lf\{\sum_{i\in\nn}
\lf[\frac{|\lz_i|\chi_{B_i}}{\|\chi_{B_i}
\|_{\lv}}\r]^{p_-}\r\}^{p_-}\r\|_{\lv}.
\end{align}

For this purpose, first, for any $i\in\nn$ and $x\in B_i^{(2)}$,
where $B_i^{(2)}$ is as in \eqref{2e2'} with $\dz=2$, it is easy to see that
\begin{align}\label{3e23}
M_0(a_i)(x)\ls M_{\rm HL}(a_i)(x).
\end{align}

On the other hand, for any $i\in\nn$, by the vanishing
moment condition of $a_i$, we conclude that, for any $t\in(0,\fz)$ and $x\in (B_i^{(2)})^{\com}$,
\begin{align}\label{3e39}
\lf|a_i*\Phi_t(x)\r|&\le\int_{B_i}\lf|a_i(y)
\Phi_t(x-y)\r|\,dy\\
&=t^{-\nu}\int_{B_i}\lf|a_i(y)\r|
\lf|\Phi\lf(\dfrac{x-y}{t^{\va}}\r)-\sum_{|\az|\le s}
\dfrac{\pa^\az\Phi(\dfrac{x-x_i}{t^{\va}})}
{\az!}\lf(\dfrac{x_i-y}{t^{\va}}\r)^\az\r|\,dy\noz\\
&\ls t^{-\nu}\int_{B_i}\lf|a_i(y)\r|
\lf|\sum_{|\az|=s+1}\pa^\az\Phi
\lf(\dfrac{\xi}{t^{\va}}\r)
\lf(\dfrac{x_i-y}{t^{\va}}\r)^\az\r|\,dy,\noz
\end{align}
where $\Phi$ is as in \eqref{3e16}, for any $i\in \nn$, $x_i$ and $r_i$ denote the
center, respectively, the radius of $B_i$ and
$\xi:=x-x_i+\theta(x_i-y)$ for some $\theta\in [0,1]$.

For any $i\in\nn$, $x\in (B_i^{(2)})^{\com}$ and $y\in B_i$, we easily find that $|x-y|_{\va}\sim|x-x_i|_{\va}$ and $|\xi|_{\va}\geq
|x-x_i|_{\va}-|x_i-y|_{\va}\gtrsim |x-x_i|_{\va}$.
From this, \eqref{3e39}, the fact that $\Phi\in \cs(\rn)$, (i),
(v) and (vi) of Lemma \ref{2l2}, the H\"{o}lder inequality and Definition \ref{3d1}(ii),
we deduce that, for any $i\in\nn$, $t\in(0,\fz)$, $K\in\mathbb{Z}_+$ and $x\in (B_i^{(2)})^{\com}$,
\begin{align}\label{3e40}
\lf|a_i*\Phi_t(x)\r|&\ls t^{-\nu}\int_{B_i}\lf|a_i(y)\r|
\dfrac1{(1+|\xi/t^{\va}|)^K}
\lf|\dfrac{x_i-y}{t^{\va}}\r|^{s+1}\,dy\\
&\ls t^{-\nu}\lf(\frac{t}{|\xi|_{\va}}\r)^{K a_-}
\lf(\dfrac{r_i}{t}\r)^{(s+1)a_-}\int_{B_i}\lf|a_i(y)\r|\,dy\noz\\
&\ls t^{-\nu}\lf(\frac{t}{|x-x_i|_{\va}}\r)^{K a_-}
\lf(\dfrac{r_i}{t}\r)^{(s+1)a_-}\|a_i\|_{L^r(\rn)}\lf|B_i\r|^{1/{r'}}\noz\\
&\ls t^{-\nu}\lf(\frac{t}{|x-x_i|_{\va}}\r)^{K a_-}
\lf(\dfrac{r_i}{t}\r)^{(s+1)a_-}\frac{|B_i|}{\|\chi_{B_i}\|_{\lv}}.\noz
\end{align}
Without loss of generality, we may assume that, for any $i\in\nn$
and $x\in (B_i^{(2)})^{\com}$, $a_i*\Phi_t(x)\neq0$.
By this and the fact that $\supp\Phi\st B(\vec0_n,1)$,
we know that $t\gs|x-x_i|_{\va}$. From this and \eqref{3e40} with
$K:=s+1$, it follows that, for any $i\in\nn$,
$t\in(0,\fz)$ and $x\in (B_i^{(2)})^{\com}$,
\begin{align*}
\lf|a_i*\Phi_t(x)\r|\ls \frac1{\|\chi_{B_i}\|_{\lv}}
\lf(\f{r_i}{|x-x_i|_{\va}}\r)^{\nu+(s+1)a_-},
\end{align*}
which implies that, for any $i\in \nn$ and $x\in (B_i^{(2)})^{\com}$,
\begin{align*}
M_0(a_i)(x)\ls \frac1{\|\chi_{B_i}\|_{\lv}}
\lf(\f{r_i}{|x-x_i|_{\va}}\r)^{\nu+(s+1)a_-}\ls
\frac1{\|\chi_{B_i}\|_{\lv}}\lf[M_{\rm HL}\lf(\chi_{B_i}\r)(x)\r]^
{\frac{\nu+(s+1)a_-}{\nu}}.
\end{align*}
By this and \eqref{3e23}, we conclude that, for any $i\in \nn$ and $x\in \rn$,
\begin{align}\label{3e24}
M_0(a_i)(x)\ls M_{\rm HL}(a_i)(x)\chi_{B_i^{(2)}}(x)+
\frac1{\|\chi_{B_i}\|_{\lv}}\lf[M_{\rm HL}\lf(\chi_{B_i}\r)(x)\r]^
{\frac{\nu+(s+1)a_-}{\nu}}.
\end{align}
Notice that $r\in (\max\{p_+,1\},\fz]$. Then, by \eqref{3e23} and Lemma \ref{3l1},
we find that
$$\lf\|M_0(a_i)\chi_{B_i^{(2)}}\r\|_{L^r(\rn)}
\ls\lf\|M_{\rm HL}(a_i)\chi_{B_i^{(2)}}\r\|_{L^r(\rn)}
\ls\frac{|B_i|^{1/r}}{\|\chi_{B_i}\|_{\lv}},$$
which, combined with Lemma \ref{3l6}, further implies that
$$\lf\|\lf\{\sum_{i\in\nn}\lf[\lf|\lz_i\r|M_0(a_i)
\chi_{B_i^{(2)}}\r]^{{p_-}}\r\}^{1/{p_-}}\r\|_{\lv}
\ls \lf\|\lf\{\sum_{i\in\nn}
\lf[\frac{|\lz_i|\chi_{B_i}}{\|\chi_{B_i}\|_{\lv}}\r]^
{{p_-}}\r\}^{1/{p_-}}\r\|_{\lv}.$$
From this, Remark \ref{2r3}(ii),
\eqref{3e24}, \eqref{2e8} and Lemma \ref{3l2},
we deduce that
\begin{align*}
\lf\|M_0\lf(\sum_{i\in\nn}\lz_ia_i\r)\r\|_{\lv}&\ls
\lf\|\sum_{i\in\nn}|\lz_i| M_0(a_i)\r\|_{\lv}\\
&\ls \lf\|\sum_{i\in\nn}|\lz_i| M_{\rm HL}(a_i)\chi_{B_i^{(2)}}\r\|_{\lv}\noz\\
&\hs+
\lf\|\sum_{i\in\nn}\frac{|\lz_i|}{\|\chi_{B_i}\|_{\lv}}\lf[M_{\rm HL}\lf(\chi_{B_i}\r)\r]^
{\frac{\nu+(s+1)a_-}{\nu}}\r\|_{\lv}\\
&\ls \lf\|\lf\{\sum_{i\in\nn}\lf[\lf|\lz_i\r|M_{\rm HL}(a_i)
\chi_{B_i^{(2)}}\r]^{{p_-}}\r\}^{1/{p_-}}\r\|_{\lv}\noz\\
&\hs+
\lf\|\sum_{i\in\nn}\lf\{\frac{|\lz_i|}{\|\chi_{B_i}\|_{\lv}}\lf[M_{\rm HL}\lf(\chi_{B_i}\r)\r]^
{\frac{\nu+(s+1)a_-}{\nu}}\r\}^{\frac{\nu}{\nu+(s+1)a_-}}\r\|_{L^
{\frac{\nu+(s+1)a_-}{\nu}\vp}(\rn)}^
{\frac{\nu+(s+1)a_-}{\nu}}\\
&\ls\lf\|\lf\{\sum_{i\in\nn}
\lf[\frac{|\lz_i|\chi_{B_i}}{\|\chi_{B_i}
\|_{\lv}}\r]^{p_-}\r\}^{1/p_-}\r\|_{\lv},
\end{align*}
which implies that \eqref{3e22} holds true. This proves \eqref{3e21}.

We now prove that $\vh\st\vah$. To this end, it suffices
to show that
\begin{align}\label{3e25}
\vh\st H_{\va}^{\vp,\fz,s}(\rn),
\end{align}
due to the fact that each $(\vp,\infty,s)$-atom is also a $(\vp,r,s)$-atom
and hence $H_{\va}^{\vp,\fz,s}(\rn)\st\vah$.

Next we prove \eqref{3e25} by two steps.

\emph{Step 1.} In this step, we show that,
for any $f\in\vh\cap L^{\vp/{p_-}}(\rn)$,
\begin{align}\label{3e26}
\|f\|_{H_{\va}^{\vp,\fz,s}(\rn)}\ls\|f\|_{\vh}
\end{align}
holds true.

For any $j \in \mathbb{Z}$, $N$ as in \eqref{2e11}
and $f\in\vh\cap L^{\vp/{\widetilde{p}_-}}(\rn)$, let
$\CO_j:=\{x\in\rn:\ M_N(f)(x)>2^j\}$.
Then, by Lemma \ref{3l4} with $\sigma=2^j$,
we know that there exist two distributions $g_j$ and $b_j$ such that
$$f=g_j+b_j
\quad {\rm in}\quad \cs'(\rn),
$$
and $b_j=\sum_{k\in\nn}b_{j,k}$ in $\cs'(\rn)$, where, for any $j\in \zz$ and
$k\in\nn$, $b_{j,k}:=(f-c_{j,k})\eta_{j,k}$, supported on $B_{j,k}\in\mathfrak{B}$,
$c_{j,k}\in \cp_s(\rn)$ and $\eta_{j,k}$ is constructed
as in \eqref{3e8} with $B_k$ replaced by $B_{j,k}$.

Moreover, by an estimation similar to that of \eqref{3e20}, we conclude that
\begin{align}\label{3e41}
\lf\|f-g_j\r\|_{\vh}=\lf\|b_j\r\|_{\vh}\ls
\lf\|M_N(f)\chi_{\CO_j}\r\|_{\lv}\to 0
\end{align}
as $j\to \fz$.
In addition, from the fact that $f\in L^{\vp/{\widetilde{p}_-}}(\rn)$ and
the H\"{o}lder inequality [see Remark \ref{2r3}(iv)],
it follows that $f\in L_{\rm loc}^1(\rn)$. By this and
Lemma \ref{3l4}(iv), we find that $\|g_j\|_{L^{\fz}(\rn)}\ls 2^j$.
This, together with \eqref{3e41}, further implies that
\begin{align}\label{3e42}
f=\sum_{j\in \mathbb{Z}}\lf(g_{j+1}-g_{j}\r)
\quad\mathrm{in}\quad\cs'(\rn).
\end{align}

On the other hand, for any $j\in \zz,\,k\in\nn$
and $q\in\cp_s(\rn)$, let
\begin{align}\label{3e28}
\|q\|_{j,k}:=
\lf[\frac1{\int_\rn\eta_{j,k}(x)\,dx}\int_\rn
|q(x)|^2\eta_{j,k}(x)\,dx\r]^{1/2}
\end{align}
and, for any $i$, $k\in\nn$ and $j\in\zz$,
$c_{j+1,k,i}$ be the orthogonal projection of
$(f-c_{j+1,i})\eta_{j,k}$ on $\cp_{s}(\rn)$
with respect to the norm in \eqref{3e28}.
Then, by \eqref{3e42} and an argument similar to that used
in \cite[pp.\,108-109]{s93}, we know that
$$f=\sum_{j\in \mathbb{Z}}\lf(g_{j+1}-g_{j}\r)=
\sum_{j\in \mathbb{Z}}\sum_{k\in\nn}\lf[b_{j,k}-\sum_{i\in \nn}\lf(b_{j+1,i}\eta_{j,k}-c_{j+1,k,i}\eta_{j+1,i}\r)\r]
=:\sum_{j\in \mathbb{Z}}\sum_{k\in\nn}
A_{j,k}\quad\mathrm{in}\quad\cs'(\rn),$$
and, for any $j\in\zz$ and $k\in\nn$, $A_{j,k}$ is supported on $B_{j,k}$ and
satisfies $\|A_{j,k}\|_{L^{\fz}(\rn)}\le C_02^j$ with $C_0$ being some positive constant,
independent of $j$ and $k$, and, for any $q\in \cp_s(\rn)$,
$$\int_{\rn}A_{j,k}(x)q(x)\,dx=0. $$
For any $j\in\zz$ and $k\in\nn$, let
\begin{align}\label{3e27}
{\kappa_{j,k}}:=C_0 2^j\lf\|B_{j,k}\r\|_{\lv}\quad {\rm and}\quad
a_{j,k}:=\f{A_{j,k}}{\kappa_{j,k}}.
\end{align}
Then it is easy to see that each $a_{j,k}$ is
a $(\vp,\fz,s)$-atom, namely,
\begin{align*}
\supp a_{j,k}\subset B_{j,k},\quad
\lf\|a_{j,k}\r\|_{L^{\infty}(\rn)}\le\frac1{\|B_{j,k}\|_{\lv}}
\end{align*}
and, for any $q\in\cp_{s}(\rn)$,
\begin{align*}
\int_\rn a_{j,k}(x)q(x)\,dx=0.
\end{align*}
Moreover, we have
$$f=\sum_{j\in \mathbb{Z}}\sum_{k\in\nn}
\kappa_{j,k}a_{j,k}\quad\mathrm{in}\quad\cs'(\rn).$$

In addition, from Definition \ref{3d2}, \eqref{3e27}, the fact that $\bigcup_{k\in\nn}B_{j,k}=\CO_j$,
the finite intersection property of $\{B_{j,k}\}_{k\in\nn}$
and the definition of $\CO_j$, we further deduce that
\begin{align*}
\|f\|_{H_{\va}^{\vp,\fz,s}(\rn)}&\ls
\lf\|\lf\{\sum_{j\in \mathbb{Z}}\sum_{k\in\nn}
\lf[\frac{\kappa_{j,k}\chi_{B_{j,k}}}{\|\chi_{B_{j,k}}
\|_{\lv}}\r]^{p_-}\r\}^{1/p_-}\r\|_{\lv}\\
&\sim \lf\|\lf[\sum_{j\in \mathbb{Z}}\sum_{k\in\nn}
\lf(2^j\chi_{B_{j,k}}\r)^{p_-}\r]^{1/p_-}\r\|_{\lv}
\ls \lf\|\lf[\sum_{j\in \mathbb{Z}}
\lf(2^j\chi_{\CO_{j}}\r)^{p_-}\r]^{1/p_-}\r\|_{\lv}\\
&\sim \lf\|\lf[\sum_{j\in \mathbb{Z}}
\lf(2^j\chi_{\CO_{j}\backslash
{\CO_{j+1}}}\r)^{p_-}\r]^{1/p_-}\r\|_{\lv}
\sim \lf\|\lf\{\sum_{j\in \mathbb{Z}}
\lf[M_N(f)\chi_{\CO_{j}\backslash
{\CO_{j+1}}}\r]^{p_-}\r\}^{1/p_-}\r\|_{\lv}\\
&\ls \lf\|M_N(f)\r\|_{\lv}\sim \|f\|_{\vh}.
\end{align*}
This implies that \eqref{3e26} holds true.

\emph{Step 2.} In this step, we prove that \eqref{3e26} also holds true
for any $f\in\vh$.

To this end, let $f\in\vh$. Then, by Lemma \ref{3l7},
we find that there exists a sequence
$\{f_i\}_{i\in\nn}\st\vh\cap L^{\vp/{\widetilde{p}_-}}(\rn)$
such that $f=\sum_{i\in\nn}f_i$ in $\vh$ and, for any $i\in\nn$,
$$\lf\|f_i\r\|_{\vh}\le2^{2-i}\|f\|_{\vh}.$$
Notice that, for any $i\in\nn$, by the conclusion obtained in Step 1,
we conclude that $f_i$ has an atomic decomposition, namely,
$$f_i=\sum_{j\in\zz}\sum_{k\in\nn}\kappa_{j,k,i}a_{j,k,i}
\hspace{0.4cm} {\rm in}\hspace{0.3cm} \cs'(\rn),$$
where $\{\kappa_{j,k,i}\}_{j\in\zz,\,k\in\nn}$
and $\{a_{j,k,i}\}_{j\in\zz,\,k\in\nn}$
are constructed as in \eqref{3e27}. Thus,
$\{a_{j,k,i}\}_{j\in\zz,\,k\in\nn}$ are $(\vp,\fz,r)$-atoms and
hence we have
$$f=\sum_{i\in\nn}\sum_{j\in\zz}\sum_{k\in\nn}\kappa_{j,k,i}a_{j,k,i}
\hspace{0.4cm} {\rm in}\hspace{0.3cm} \cs'(\rn)$$
and
$$\|f\|_{H_{\va}^{\vp,\fz,r}(\rn)}
\le\lf[\sum_{i\in\nn}\lf\|f_i\r\|_{\vh}^{p_-}\r]^{1/{p_-}}
\ls\|f\|_{\vh},$$
which implies that \eqref{3e26} holds true for any $f\in\vh$ and
hence completes the proof of Theorem \ref{3t1}.
\end{proof}

\section{Littlewood-Paley function characterizations of $\vh$\label{s4}}

In this section, as an application of the atomic
characterizations of $\vh$ obtained in Theorem \ref{3t1},
we establish the Littlewood-Paley function
characterizations of $\vh$.

Let $\va\in [1,\fz)^n$. Assume that $\phi\in\cs(\rn)$ is a radial function such that,
for any multi-index $\alpha\in\zz_+^n$ with $|\alpha|\le s$,
where $s$ is as in \eqref{3e1},
\begin{align}\label{4e1}
\int_{\rn}\phi(x)x^\alpha\,dx=0
\end{align}
and, for any
$\xi\in\rn\backslash\{\vec{0}_n\}$,
\begin{align}\label{4e2}
\sum_{k\in\zz}\lf|\widehat{\phi}\lf(2^{k\va}\xi\r)\r|^2=1,
\end{align}
here and hereafter, $\widehat{\phi}$ denotes the \emph{Fourier transform} of $\phi$,
namely, for any $\xi\in\rn$,
\begin{align*}
\widehat \phi(\xi) := \int_{\rn} \phi(x) e^{-2\pi\imath x \cdot \xi} \, dx,
\end{align*}
where $\imath:=\sqrt{-1}$.
Then, for any $\lambda\in(0,\fz)$ and
$f\in\cs'(\rn)$, the \emph{anisotropic Lusin area function} $S(f)$,
the \emph{Littlewood-Paley} $g$-\emph{function} $g(f)$ and
the \emph{Littlewood-Paley} $g_\lambda^\ast$-\emph{function} $g_\lambda^\ast(f)$
are defined, respectively, by setting, for any $x\in\rn$,
\begin{align*}
S(f)(x):=\lf[\sum_{k\in\mathbb{Z}}2^{-k\nu}\int_{B_{\va}(x,2^k)}
\lf|f\ast\phi_{k}(y)\r|^2\,dy\r]^{1/2},
\end{align*}
\begin{align}\label{4e25}
g(f)(x):=\lf[\sum_{k\in\mathbb{Z}}
\lf|f\ast\phi_{k}(x)\r|^2\r]^{1/2}
\end{align}
and
\begin{align*}
g_\lambda^\ast(f)(x):=
\lf\{\sum_{k\in\mathbb{Z}}2^{-k\nu}\int_{\rn}
\lf[\frac{2^{k}}{2^{k}+|x-y|_{\va}}\r]^{\lambda\nu}
\lf|f\ast\phi_{k}(y)\r|^2\,dy\r\}^{1/2},
\end{align*}
where, for any $k\in \mathbb{Z}$, $\phi_{k}(\cdot)
:=2^{-k\nu}\phi(2^{-k\va}\cdot)$.

Recall that $f\in\cs'(\rn)$ is said to
\emph{vanish weakly at infinity} if, for any $\phi\in\cs(\rn)$,
$f\ast\phi_{k}\to0$ in $\cs'(\rn)$ as $k\to \fz$.
In what follows, we always let $\cs'_0(\rn)$ be the set of all $f\in\cs'(\rn)$
vanishing weakly at infinity.

Then the main results of this section are the following
succeeding three theorems.

\begin{theorem}\label{4t1}
Let $\va\in [1,\fz)^n$, $\vp\in(0,\fz)^n$ and
$N$ be as in \eqref{2e11}.
Then $f\in\vh$ if and only if
$f\in\cs'_0(\rn)$ and $S(f)\in\lv$. Moreover,
there exists a positive constant $C$ such that,
for any $f\in\vh$,
$$C^{-1}\|S(f)\|_{\lv}\le\|f\|_{\vh}\le C\|S(f)\|_{\lv}.$$
\end{theorem}

\begin{theorem}\label{4t2}
Let $\va,\,\vp$ and $N$ be as in Theorem \ref{4t1}.
Then $f\in\vh$ if and only if
$f\in\cs'_0(\rn)$ and $g(f)\in\lv$. Moreover,
there exists a positive constant $C$ such that,
for any $f\in\vh$,
$$C^{-1}\|g(f)\|_{\lv}\le\|f\|_{\vh}\le C\|g(f)\|_{\lv}.$$
\end{theorem}

\begin{theorem}\label{4t3}
Let $\va,\,\vp$ and $N$ be as in Theorem \ref{4t1} and
$\lambda\in(1+\frac{2}{{\min\{\widetilde{p}_-,2\}}}, \fz)$,
where $\widetilde{p}_-:=\min\{p_1,\ldots,p_n\}$.
Then $f\in\vh$ if and only if $f\in\cs'_0(\rn)$ and
$g_\lambda^{\ast}(f)\in\lv$. Moreover,
there exists a positive constant $C$ such that,
for any $f\in\vh$,
$$C^{-1}\lf\|g_\lz^\ast(f)\r\|_{\lv}\le\|f\|_{\vh}
\le C\lf\|g_\lz^\ast(f)\r\|_{\lv}.$$
\end{theorem}

\begin{remark}\label{4r2}
\begin{enumerate}
\item[{\rm (i)}]
We should point out that Theorem \ref{4t2} gives a positive answer to the conjecture
proposed by Hart et al. in \cite[p.\,9]{htw17}, namely, the mixed-norm Hardy space
$H^{p,q}(\mathbb{R}^{n+1})$, with $p,\,q\in(0,\fz)$, introduced by Hart et al. \cite{htw17} via the
Littlewood-Paley $g$-function coincides, in the sense of equivalent quasi-norms, with $H^{\vp}_{\va}(\mathbb{R}^{n+1})$,
where $\va:=(\overbrace{1,\ldots, 1}^{n+1\ \mathrm{times}})$ and
$\vp:=(\overbrace{p,\ldots,p}^{n\ \mathrm{times}},q)$.
\item[{\rm (ii)}]
We should also point out that
the range of $\lz$ in Theorem \ref{4t3} does not coincide with the best known one,
namely, $\lz\in(2/p,\fz)$, of the $g_\lz^\ast$-function characterization
of the classical Hardy space $H^p(\rn)$ and it is still unclear whether or not the
$g_\lz^\ast$-function, when
$\lz\in(\frac{2}{\min\{\widetilde{p}_-,2\}},1+\frac{2}{\min\{\widetilde{p}_-,2\}}]$,
can characterize $\vh$, because the method used in the proof of
Theorem \ref{4t3} does not work in this case.
\end{enumerate}
\end{remark}

The following proposition establishes the relation between $\vh$ and $H_A^p(\rn)$,
where $H_A^p(\rn)$ is the anisotropic Hardy space introduced by Bownik in
\cite[p.\,17, Definition 3.11]{mb03}.

\begin{proposition}\label{2r4'}
Let $\va:=(a_1,\ldots,a_n)\in[1,\fz)^n$ and
$\vp:=(\overbrace{p,\ldots,p}^{n\ \rm times})$, where $p\in(0,\fz)$.
Then $\vh$ and the anisotropic Hardy space $\vAh$
coincide with equivalent quasi-norms, where
$A$ is as in \eqref{4e5}.
\end{proposition}
\begin{proof}
Let $A$ be as in \eqref{4e5}. Then, by \cite[Remark 2.5(i)]{lwyy17}
and \cite[Theorem 6.2]{lwyy17} with $p(\cdot):=p\in(0,\fz)$, we conclude that
$f\in\vAh$ if and only if $f\in\cs'_0(\rn)$ and
$$\lf[\sum_{k\in\mathbb{Z}}
\lf||\det A|^{-k}f\ast\phi(A^{-k}\cdot)\r|^2\r]^{1/2}
=\lf[\sum_{k\in\mathbb{Z}}
\lf|2^{-k\nu}f\ast\phi(2^{-k\va}\cdot)\r|^2\r]^{1/2}
=g(f)\in L^p(\rn),$$
where $\phi$ is as in \eqref{4e25}.
This, combined with Theorem \ref{4t2} and the obvious fact that,
when $\vp:=(\overbrace{p,\ldots,p}^{n\ \rm times})$ with $p\in(0,\fz)$,
$L^{\vp}(\rn)=L^p(\rn)$, further implies that, in this case, $f\in\vAh$
if and only if $f\in\vh$. Thus,
when $\va:=(a_1,\ldots,a_2)\in[1,\fz)^n$ and
$\vp:=(\overbrace{p,\ldots,p}^{n\ \rm times})$, where $p\in(0,\fz)$,
$\vh=\vAh$ with equivalent quasi-norms, where
$A$ is as in \eqref{4e5}. This finishes the proof of
Proposition \ref{2r4'}.
\end{proof}

\begin{remark}
Recall that, via the Lusin-area function,
the Littlewood-Paley $g$-function or $g_\lambda^\ast$-function,
Li et al. in \cite[Theorems 2.8, 3.1 and 3.9]{lfy15} characterized the anisotropic Musielak-Orlicz Hardy
space $H_A^\varphi(\rn)$ with $\varphi:\ \rn\times[0,\fz)\to[0,\fz)$
being an anisotropic growth function (see \cite[Definition 2.3]{lfy15}).
As was mentioned in \cite[p.\,285]{lfy15}, if, for any given $p\in(0,1]$
and any $x\in\rn$ and $t\in(0,\fz)$,
\begin{align}\label{4e3}
\varphi(x,t):=t^p,
\end{align}
then $H_A^\varphi(\rn)=\vAh$.
From this and Proposition \ref{2r4'}, we deduce that,
when $\vp:=(\overbrace{p,\ldots,p}^{n\ \rm times})$,
where $p\in(0,1]$, Theorems \ref{4t1}, \ref{4t2} and \ref{4t3}
are just \cite[Theorems 2.8, 3.1 and 3.9]{lfy15},
respectively, with $A$ as in \eqref{4e5} and $\varphi$ as in \eqref{4e3}.
\end{remark}

To prove Theorem \ref{4t1},
we need several technical lemmas. First, it is easy
to see that the following conclusion holds
true, the details being omitted.

\begin{lemma}\label{4l7}
Let $\va:=(a_1,\ldots,a_n)\in [1,\fz)^n$, $\vp:=(p_1,\ldots,p_n)\in(0,\fz)^n$, $r\in (0,\fz)$,
$x\in \rn$ and $Q_{\va}(x,r)\in \mathfrak{Q}$
with $\mathfrak{Q}$ as in \eqref{2e3}.
Then $\|\chi_{Q_{\va}(x,r)}\|_{\lv}=\prod_{i=1}^{n}2^{1/p_i}r^{a_i/p_i}$.
\end{lemma}

Moreover, using Lemma \ref{4l7} and borrowing some ideas from the proof of
\cite[Lemma 6.5]{yyyz16}, we obtain the following conclusion.

\begin{lemma}\label{4l1}
Let $\va\in [1,\fz)^n$, $\vp\in(0,\fz)^n$ and
$N$ be as in \eqref{2e11}.
Then $\vh\subset\cs'_0(\rn)$.
\end{lemma}

\begin{proof}
Let $f\in \vh$. Then, by Remark \ref{2r5}, we know that,
for any $\phi\in \cs(\rn)$, $k\in \mathbb{Z}$,
$x\in \rn$ and $y\in Q_{\va}(x,2^k)$,
$|f\ast\phi_k(x)|\ls M_N(f)(y)$ with $N$ as in \eqref{2e11}.
Thus, there exists a positive constant $C$ such that
\begin{align}\label{4e4}
Q_{\va}(x,2^k)\subset \{y\in \rn:\, M_N(f)(y)\geq C|f\ast\phi_k(x)|\}.
\end{align}
On the other hand, from Lemma \ref{4l7}, it follows that, for any $k\in \mathbb{Z}_+$,
$\|\chi_{Q_{\va}(x,2^k)}\|_{\lv}\gs 2^{k\nu/p_+}$
with $p_+$ as in \eqref{2e10}. By this and \eqref{4e4},
we conclude that, for any $x\in \rn$,
\begin{align*}
\lf|f\ast\phi_k(x)\r|&=\lf|Q_{\va}(x,2^k)\r|^{-1/{p_+}}
\lf|Q_{\va}(x,2^k)\r|^{1/{p_+}}\lf|f\ast\phi_k(x)\r|\\
&\ls 2^{-k\nu/p_+}\lf\|\chi_{Q_{\va}(x,2^k)}\r\|_{\lv}|f\ast\phi_k(x)|
\ls  2^{-k\nu/p_+}\lf\|M_N(f)\r\|_{\lv}\to 0
\end{align*}
as $k\to \fz$. This implies $f\in \cs'_0(\rn)$ and hence finishes the proof of Lemma \ref{4l1}.
\end{proof}

The following lemma is a special case of \cite[Lemma 2.3]{blyz10},
which is a variant of \cite[Theorem 11]{mc90}. Indeed, let $(a_1,\ldots,a_n)\in[1,\fz)^n$. Then, applying
\cite[Lemma 2.3]{blyz10} with $A$ as in \eqref{4e5},
we immediately obtain the following conclusions, the details being omitted.
\begin{lemma}\label{4l2}
Let $\va\in [1,\fz)^n$. Then there exists a set
$$\mathcal{Q}:=\lf\{Q_\alpha^k\subset\rn:\ k\in\mathbb{Z},
\,\alpha\in E_k\r\}$$
of open subsets, where $E_k$ is some index set, such that
\begin{enumerate}
\item[{\rm (i)}] for each $k\in\zz$,
$\lf|\rn\setminus\bigcup_{\alpha}Q_\alpha^k\r|=0$
and, when $\alpha\neq\beta$,
$Q_\alpha^k\cap Q_\beta^k=\emptyset$;
\item[{\rm(ii)}] for any $\alpha,\,\beta,\,k,\,\ell$ with $\ell\geq k$,
either $Q_\alpha^k\cap Q_\beta^\ell=\emptyset$ or
$Q_\alpha^\ell\subset Q_\beta^k$;
\item[{\rm(iii)}] for each $(\ell,\beta)$ and each $k<\ell$,
there exists a unique $\alpha$ such that
$Q_\beta^\ell\subset Q_\alpha^k$;
\item[{\rm(iv)}] there exist some $w\in\zz\setminus\zz_+$
and $u\in\nn$ such that, for any $Q_\alpha^k$
with $k\in\mathbb{Z}$ and $\alpha\in E_k$,
there exists $x_{Q_\alpha^k}\in Q_\alpha^k$
such that, for any $x\in Q_\alpha^k$,
$$x_{Q_\alpha^k}+2^{(wk-u)\va}B_0
\subset Q_\alpha^k\subset x+2^{(wk+u)\va}B_0,$$
where $B_0$ denotes the unit ball of $\rn$.
\end{enumerate}
\end{lemma}

In what follows, we call
$\mathcal{Q}:=
\{Q_\alpha^k\}_{k\in\mathbb{Z},\,\alpha\in E_k}$
from Lemma \ref{4l2} \emph{dyadic cubes} and
$k$ the \emph{level}, denoted by $\ell(Q_\alpha^k)$,
of the dyadic cube $Q_\alpha^k$
for any $k\in\mathbb{Z}$ and $\alpha\in E_k$.

\begin{remark}\label{4r1}
In the definition of $(\vp,r,s)$-atoms (see Definition \ref{3d1}),
if we replace anisotropic balls $\mathfrak{B}$ by
dyadic cubes, then, from Lemma \ref{4l2}, we deduce that
the corresponding variable anisotropic atomic Hardy space
coincides with the original one (see Definition \ref{3d2})
in the sense of equivalent quasi-norms.
\end{remark}

Now we establish the following Calder\'{o}n reproducing formula.

\begin{lemma}\label{4l3}
Let $\va \in [1,\fz)^n$ and $\varphi\in \cs(\rn)$ satisfy that
$\supp \widehat{\varphi}$ is compact and bounded away from the origin
and, for any $\xi\in\rn\setminus\{\vec{0}_n\}$,
\begin{align}\label{4e6}
\sum_{k\in\mathbb{Z}}
\widehat{\varphi}\lf(2^{k\va}\xi\r)=1.
\end{align}
Then, for any $f\in L^2(\rn)$,
$f=\sum_{k\in\mathbb{Z}}f\ast\varphi_k$ in $L^2(\rn)$.
The same holds true in $\cs'(\rn)$ for any $f\in \cs'_0(\rn)$.
\end{lemma}

To show Lemma \ref{4l3}, we need the following Lemma \ref{4l4},
which is just a variant of \cite[Lemma 3.8]{mb03},
the details being omitted.

\begin{lemma}\label{4l4}
Let $\varphi\in \cs(\rn)$ and $\int_{\rn}\varphi(x)\,dx=1$.
Then, for any $f\in \cs(\rn)$, $f\ast\varphi_k\to f$ in $\cs(\rn)$
as $k\to -\fz$. The same holds true in $\cs'(\rn)$ for any $f\in \cs'(\rn)$.
\end{lemma}

Now we prove Lemma \ref{4l3}.
\begin{proof}[Proof of Lemma \ref{4l3}]
We show this lemma by two steps.

\emph{Step 1.}
Assume $f\in L^2(\rn)$. For any $\xi\in \rn$, let $F(\xi):=\sum_{k\in\mathbb{Z}}
|\widehat{\varphi}(2^{k\va}\xi)|$.
Obviously, for any $\xi\in \rn$, $F(\xi)=F(2^{\va}\xi)$, which implies that, to show
$F\in L^{\fz}(\rn)$, it suffices to consider the values of
$F$ on $2^{\va}B_0\setminus B_0$, where $B_0$ denotes the unit ball of $\rn$.
Since $\widehat{\varphi}\in \cs(\rn)$ and $\supp \widehat{\varphi}$ is bounded away from $\vec0_n$,
it follows that, for any
$\xi\in \rn\setminus B_0$, $|\widehat{\varphi}(\xi)|\ls\frac1{1+|\xi|}$ and,
for any $\xi\in 2^{\va}B_0$, $|\widehat{\varphi}(\xi)|\ls|\xi|$.
Then, by (i), (v) and (vi) of Lemma \ref{2l2}, we find that, for any
$\xi \in 2^{\va}B_0\setminus B_0$,
\begin{align*}
F(\xi)&=\sum_{k\geq 0}\lf|\widehat{\varphi}\lf(2^{k\va}\xi\r)\r|+
\sum_{k< 0}\lf|\widehat{\varphi}\lf(2^{k\va}\xi\r)\r|
\ls \sum_{k\geq 0}\frac1{1+|2^{k\va}\xi|}+
\sum_{k< 0}\lf|2^{k\va}\xi\r|\\
&\ls\sum_{k\geq 0}\frac1{(2^k|\xi|_{\va})^{a_-}}+
\sum_{k< 0}\lf(2^k|\xi|_{\va}\r)^{a_-}
\ls\sum_{k\geq 0}\frac1{2^{ka_-}}+\sum_{k< 0}2^{(k+1)a_-}\ls 1,
\end{align*}
which implies $F\in L^{\fz}(\rn)$. By this, the Lebesgue
dominated convergence theorem and \eqref{4e6}, we conclude that, for any
$f\in L^2(\rn)$ and $\xi\in \rn$,
$$\widehat{f}(\xi)=\sum_{k\in \zz}\widehat{\varphi}\lf(2^{k\va}\xi\r)
\widehat{f}(\xi)\quad\mathrm{in}\quad L^2(\rn)$$
and hence $f=\sum_{k\in\mathbb{Z}}f\ast\varphi_k$ in $L^2(\rn)$.

\emph{Step 2.} Assume $f\in \cs'_0(\rn)$.
Let $\phi:=\sum_{k=0}^{\fz}\varphi_k$. Since $\varphi\in \cs(\rn)$
and $\varphi_{k}(\cdot):=2^{-k\nu}\varphi(2^{-k\va}\cdot)$,
it then follows that $\phi$ is well defined pointwise on $\rn$.
We now claim that \begin{align}\label{4e24}
\phi\in \cs(\rn)\quad \mathrm{and} \quad\int_{\rn}\phi(x)\,dx=1.
\end{align}
Assume that this claim holds true for the moment.
Then, by this and Lemma \ref{4l4}, we know that $f\ast\phi_{-N}\to f$
in $\cs'(\rn)$ as $N\to \fz$. On the other hand,
for any $f\in \cs'_0(\rn)$, it is easy to see
that $f\ast\phi_{N}\to 0$ in $\cs'(\rn)$ as $N\to \fz$.
Therefore, for any $f\in \cs'_0(\rn)$, as $N\to\fz$,
$$f\ast\phi_{-N}-f\ast\phi_{N}\to f\quad\mathrm{in}
\quad \cs'(\rn).$$
Moreover, since, for any $j\in \zz$,
$\phi_j=\sum_{k=0}^{\fz}(\varphi_k)_j=\sum_{k=j}^{\fz}\varphi_k$,
it follows that $\sum_{k=-N}^{N}\varphi_k=\phi_{-N}-\phi_{N+1}$. Therefore,
$$\lim_{N\to \fz}\sum_{k=-N}^{N}f\ast\varphi_k=\lim_{N\to \fz}f\ast
\lf(\sum_{k=-N}^{N}\varphi_k\r)=\lim_{N\to \fz}\lf(f\ast\phi_{-N}-f\ast\phi_{N+1}\r)=f$$
in $\cs'(\rn)$, which implies that, for any $f\in \cs'_0(\rn)$,
$f=\sum_{k\in\zz}f\ast\varphi_k$ holds true in $\cs'(\rn)$.

Let us now prove the above claim \eqref{4e24}. To this end, for any $\xi\in \rn$, let $G(\xi):=\sum_{k=0}^{\fz}
\widehat{\varphi}(2^{k\va}\xi)$.
Then, to show \eqref{4e24}, it suffices to prove that
$G\in \cs(\rn)$, $\phi=\mathcal{F}^{-1}G$ and
$\int_{\rn}\phi(x)\,dx=1$, where $\mathcal{F}^{-1}$ denotes
the inverse Fourier transform, namely, for any $\xi\in\rn$,
$\mathcal{F}^{-1}G(\xi):=\widehat{G}(-\xi)$.

Indeed, since $\supp \widehat{\varphi}$ is compact, we may assume that
$\supp \widehat{\varphi}\st 2^{k_0\va}B_0$ for some $k_0\in \zz$.
Then, for any $k\in\zz_+$, it is easy to see that
$\supp \widehat{\varphi}(2^{k\va}\cdot)\st 2^{(k_0-k)\va}B_0\st 2^{k_0\va}B_0$,
which implies that $\supp G\st 2^{k_0\va}B_0$. To prove $G\in C^{\fz}(\rn)$,
for any multi-index $\az\in \zz_+^n$ and $\xi\in \rn$, let
$$F_{\az}(\xi):=\sum_{k\in \zz}\lf|\pa^{\az}\lf[\widehat{\varphi}\lf(2^{k\va}\xi\r)\r]\r|.$$
We first show $F_{\az}\in L^{\fz}(\rn)$. Notice that, for any $\xi\in \rn$,
$$F_{\az}(2^{\va}\xi)=\sum_{k\in \zz}\lf|\pa^{\az}\lf[\widehat{\varphi}\lf(2^{(k+1)\va}\xi\r)\r]\r|
=\sum_{k\in \zz}\lf|\pa^{\az}\lf[\widehat{\varphi}\lf(2^{k\va}\xi\r)\r]\r|=F_{\az}(\xi),$$
which implies that, to show $F_{\az}\in L^{\fz}(\rn)$, we only need to consider the value
of $F_{\az}$ on $2^{\va}B_0\setminus B_0$.
From the fact that $\widehat{\varphi}\in \cs(\rn)$, (i) and (vi) of Lemma \ref{2l2}, we deduce that,
for any $\xi\in 2^{\va}B_0\setminus B_0$,
$$\lf|\pa^{\az}\lf[\widehat{\varphi}\lf(2^{k\va}\xi\r)\r]\r|\ls \frac1{1+|2^{k\va}\xi|}
\ls \frac1{(2^k|\xi|_{\va})^{a_-}}\ls 2^{-k a_-}$$
when $k\in\nn$, and $|\pa^{\az}\widehat{\varphi}(2^{k\va}\xi)|\ls 2^{k|\az|a_-}$
when $k\in \zz\setminus\nn$. By this, we further conclude that, for any $\xi\in 2^{\va}B_0\setminus B_0$,
$$F_{\az}(\xi)\ls \sum_{k\in\nn}{2^{-ka_-}}+\sum_{k\in \zz\setminus\nn}2^{k|\az|a_-}\ls 1.$$
Thus, $F_{\az}\in L^{\fz}(\rn)$. This implies that, for any $\xi\in \rn$,
$\pa^{\az} G(\xi)=\sum_{k=0}^{\fz}\pa^{\az}[\widehat{\varphi}(2^{k\va}\xi)]$.
Therefore, $G\in C^{\fz}(\rn)$. From this and $\supp G\st 2^{k_0\va}B_0$,
it follows that $G\in \cs(\rn)$.

Moreover, by the facts that
$\supp\widehat{\varphi}(2^{k\va}\cdot)\st 2^{(k_0-k)\va}B_0$ and $\supp (\sum_{k=0}^{\fz}
|\widehat{\varphi}(2^{k\va}\cdot)|)\st 2^{k_0\va}B_0$,
the H\"{o}lder inequality and the Minkowski inequality, we find that
\begin{align*}
\int_{\rn} \sum_{k=0}^{\fz}\lf|\widehat{\varphi}\lf(2^{k\va}\xi\r)\r|\,d\xi
&\le \lf|2^{k_0\va}B_0\r|^{1/2}\lf\{\int_{\rn} \lf[\sum_{k=0}^{\fz}
\lf|\widehat{\varphi}\lf(2^{k\va}\xi\r)\r|\r]^2\,d\xi\r\}^{1/2}
\ls 2^{\nu k_0/2} \sum_{k=0}^{\fz}\lf[\int_{\rn}
\lf|\widehat{\varphi}\lf(2^{k\va}\xi\r)\r|^2\,d\xi\r]^{1/2}\\
&\ls 2^{\nu k_0/2} \sum_{k=0}^{\fz}\lf[\int_{\rn}
\chi_{2^{(k_0-k)\va}B_0}(\xi)\,d\xi\r]^{1/2}
\ls 2^{\nu k_0} \sum_{k=0}^{\fz}2^{-k\nu/2}\ls 1.
\end{align*}
Then, by the Fubini theorem, we obtain $\mathcal{F}^{-1}G=\sum_{k=0}^{\fz}
\mathcal{F}^{-1}[\widehat{\varphi}(2^{k\va}\cdot)]=\phi$ and hence $\phi\in \cs(\rn)$.

Let $e_1:=(1,0,\ldots,0)\in \rn$. Since $\widehat{\varphi}\in \cs(\rn)$, from \eqref{4e6},
we deduce that
$$\int_{\rn}\phi(x)\,dx=\widehat{\phi}(\vec0_n)=\lim_{j\to -\fz}\widehat{\phi}\lf(2^{j\va}e_1\r)=
\lim_{j\to -\fz}\sum_{k=0}^{\fz}\widehat{\varphi}\lf(2^{(j+k)\va} e_1\r)=\sum_{k\in\mathbb{Z}}
\widehat{\varphi}\lf(2^{k\va}e_1\r)=1,$$
which completes the proof of \eqref{4e24} and hence of Lemma \ref{4l3}.
\end{proof}

Using Lemma \ref{4l3}, we obtain the following Calder\'{o}n reproducing formula.
\begin{lemma}\label{4l5}
Let $\va\in [1,\fz)^n$ and $s\in\mathbb{Z_+}$.
Then there exist $\varphi,\,\psi\in\cs(\rn)$ satisfying
\begin{enumerate}
\item[{\rm(i)}] $\supp\varphi\subset B_0,
\,\int_{\rn}x^\gamma\varphi(x)\,dx=0$ for any
$\gamma\in\zz_+^n$ with $|\gamma|\le s,
\,\widehat{\varphi}(\xi)\geq C$
for any $\xi\in\{x\in\rn:\ m\le|x|_{\va}\le t\}$,
where $0<m<t<1$ and $C\in(0,\fz)$ are constants;
\item[{\rm(ii)}] $\supp \widehat{\psi}$
is compact and bounded away from the origin;
\item[{\rm(iii)}] for any $\xi\in\rn\setminus\{\vec{0}_n\}$,
$\sum_{k\in\mathbb{Z}}
\widehat{\psi}(2^{k\va}\xi)\widehat{\varphi}(2^{k\va}\xi)=1$.
\end{enumerate}

Moreover, for any $f\in L^2(\rn),\,f=
\sum_{k\in\mathbb{Z}}f\ast\psi_k\ast\varphi_k$ in $L^2(\rn)$.
The same holds true in $\cs'(\rn)$ for any $f\in \cs'_0(\rn)$.
\end{lemma}

We point out that the existences of such $\varphi$ and $\psi$
in Lemma \ref{4l5} can be verified by an argument similar to that used in the proof of
\cite[Theorem 5.8]{bh06}. Then the conclusions of Lemma \ref{4l5}
follow immediately from Lemma \ref{4l3} via replacing $\varphi$ by $\varphi\ast\psi$.

Now we prove Theorem \ref{4t1}.

\begin{proof}[Proof of Theorem \ref{4t1}]
We first show the sufficiency of this theorem. For this purpose,
let $f\in\cs'_0(\rn)$ and $S(f)\in\lv$. Then we need
to prove that $f\in\vh$ and
\begin{align}\label{4e7}
\|f\|_{\vh}\ls\|S(f)\|_{\lv}.
\end{align}
To this end, for any $k\in\mathbb{Z}$, let
$\Omega_k:=\{x\in\rn:\ S(f)(x)>2^k\}$ and
$$\mathcal{Q}_k:=\lf\{Q\in\mathcal{Q}:
\ |Q\cap\Omega_k|>\frac{|Q|}2\ \ {\rm and}\
\ |Q\cap\Omega_{k+1}|\le\frac{|Q|}2\r\}.$$
It is easy to see that, for any $Q\in\mathcal{Q}$,
there exists a unique $k\in\mathbb{Z}$
such that $Q\in\mathcal{Q}_k$.
For any given $k\in\zz$, denote by $\{Q_i^k\}_i$ the collection of all \emph{maximal dyadic cubes}
in $\mathcal{Q}_k$,
namely, there exists no $Q\in\mathcal{Q}_k$
such that $Q_i^k\subsetneqq Q$ for any $i$.

For any $Q\in\mathcal{Q}$, let
$$\widehat{Q}:=\lf\{(y,t)\in\rn\times\mathbb{R}:\
y\in Q\ \ {\rm and}\
\ t\sim w\ell(Q)+u\r\},$$
here and hereafter, $t\sim w\ell(Q)+u$ always means
\begin{align}\label{4e8}
w\ell(Q)+u+1\le t<w[\ell(Q)-1]+u+1,
\end{align}
where $w$ and $u$ are as in Lemma \ref{4l2}(iv) and $\ell(Q)$ denotes the level of $Q$.
Clearly, $\{\widehat{Q}\}_{Q\in\mathcal{Q}}$ are mutually disjoint and
\begin{align}\label{4e9}
\rn\times\mathbb{R}=\bigcup_{k\in\mathbb{Z}}\bigcup_i B_{k,\,i},
\end{align}
where, for any $k\in\mathbb{Z}$ and $i$,
$B_{k,\,i}:=\bigcup_{Q\subset Q_i^k,\,Q\in\mathcal{Q}_k}\widehat{Q}$.
Then, by Lemma \ref{4l2}(ii), we easily know that $\{B_{k,i}\}_{k\in\zz,\,i}$
are mutually disjoint.

Let $\psi$ and $\varphi$ be as in Lemma \ref{4l5}. Then
$\varphi$ has the vanishing moments up to order $s$ as in \eqref{3e1}. By
Lemma \ref{4l5}, the properties of the tempered distributions
(see \cite[Theorem 2.3.20]{lg14} or \cite[Theorem 3.13]{sw71}) and \eqref{4e9},
we find that, for any $f\in\cs'_0(\rn)$ with
$S(f)\in \lv$ and $x\in\rn$,
\begin{align*}
f(x)
&=\sum_{k\in\mathbb{Z}}f\ast\psi_k\ast\varphi_k(x)
=\int_{\rn\times\mathbb{R}}
f\ast\psi_t(y)\varphi_t(x-y)\,dy\,dm(t)\\
&=\sum_{k\in\mathbb{Z}}\sum_i\int_{B_{k,\,i}}
f\ast\psi_t(y)\varphi_t(x-y)\,dy\,dm(t)
=:\sum_{k\in\mathbb{Z}}\sum_i h_i^k(x)
\end{align*}
in $\cs'(\rn)$, where, for any $k\in\mathbb{Z},\,i$ and $x\in\rn$,
\begin{align}\label{4e10}
h_i^k(x)
:=&\int_{B_{k,\,i}}f\ast\psi_t(y)\varphi_t(x-y)\,dy\,dm(t)\\
=&\sum_{Q\subset Q_i^k,\,Q\in\mathcal{Q}_k}
\int_{\widehat{Q}}f\ast\psi_t(y)\varphi_t(x-y)\,dy\,dm(t)
=:\sum_{Q\subset Q_i^k,\,Q\in\mathcal{Q}_k}e_{Q}(x)\noz
\end{align}
with convergence in $\cs'(\rn)$, and $m(t)$ denotes
the \emph{counting measure} on $\mathbb{R}$, namely,
for any set $E\st \rr$, $m(E):=\sharp E$ if $E$ has
only finite elements, or else $m(E):=\fz$.

Using \cite[(3.23)]{lyy16LP}
with the dilation $A$ as in \eqref{4e5}, we conclude that,
for any $x\in\rn$,
\begin{align}\label{4e11}
\lf[S\lf(\sum_{Q\in\mathcal{R}}e_Q\r)(x)\r]^2\ls
\sum_{Q\in\mathcal{R}}\lf[M_{{\rm HL}}(c_Q\chi_Q)(x)\r]^2,
\end{align}
where $\mathcal{R}\st\mathcal{Q}$ is an arbitrary set of dyadic cubes, $e_Q$ is as in \eqref{4e10}
and, for any $Q\in\mathcal{R}$,
$$c_Q:=\lf[\int_{\widehat{Q}}|\psi_t\ast f(y)|^2
\,dy\frac{dm(t)}{2^{\nu t}}\r]^{1/2}.$$

Next we show that, for any $k\in\zz$ and $i$,
$h_i^k$ is a $(\vp,r,s)$-atom multiplied by a harmless constant. This is completed
by Steps 1 through 3 below.

\emph{Step 1.}
For any $x\in\supp h_i^k$, by \eqref{4e10}, $h_i^k(x)\neq0$
implies that there exists $Q\subset Q_i^k$ and
$Q\in\mathcal{Q}_k$ such that $e_{Q}(x)\neq0$.
Then there exists $(y,t)\in\widehat{Q}$ such that
$2^{-t\va}(x-y)\in B_0$, where $B_0$ denotes the unit ball of $\rn$.
By this, Lemma \ref{4l2}(iv),
\eqref{4e8} and Lemma \ref{2l2}(ii), we have
$$x\in y+2^{t\va}B_0\subset x_Q+2^{(w\ell(Q)+u)\va}B_0
+2^{(w[\ell(Q)-1]+u+1)\va}B_0\subset x_Q+2^{(w[\ell(Q)-1]+u+2)\va}B_0.$$
Thus,
$$\supp e_Q\subset x_Q+2^{(w[\ell(Q)-1]+u+2)\va}B_0.$$
From this, the fact that
$h_i^k=\sum_{Q\subset Q_i^k,\,Q\in\mathcal{Q}_k}e_{Q}$,
(ii) and (iv) of Lemma \ref{4l2} and Lemma \ref{2l2}(ii),
we further deduce that
\begin{align}\label{4e12}
\supp h_i^k
&\subset\bigcup_{Q\subset Q_i^k,\,Q\in\mathcal
{Q}_k}x_Q+2^{(w[\ell(Q)-1]+u+2)\va}B_0\\
&\subset x_{Q_i^k}+2^{w[\ell(Q_i^k)+u]\va}B_0+
2^{(w[\ell(Q_i^k)-1]+u+2)\va}B_0
\subset x_{Q_i^k}+2^{(w[\ell(Q_i^k)-1]+u+2)\va}B_0=:B_i^k.\noz
\end{align}

\emph{Step 2.}
For any $Q\in\mathcal{Q}_k$ and
$x\in Q$, by Lemma \ref{4l2}(iv), we find that
$$M_{{\rm HL}}\lf(\chi_{\Omega_k}\r)(x)\ge\frac1{2^{[w\ell(Q)+u]\nu}}
\int_{x_Q+2^{[w\ell(Q)+u]\va}B_0}\chi_{\Omega_k}(z)\,dz>2^{-2u\nu}
\frac{|\Omega_k\cap Q|}{|Q|}> 2^{-2u\nu-1},$$
which implies that
\begin{align}\label{4e13}
\bigcup_{Q\subset Q_i^k,\,Q\in\mathcal{Q}_k}Q
\subset\widehat{\Omega}_k:=\lf\{x\in\rn:\
M_{{\rm HL}}\lf(\chi_{\Omega_k}\r)(x)> 2^{-2u\nu-1}\r\}.
\end{align}
In addition, for any $Q\in\mathcal{Q}_k$ and $x\in Q$,
by Lemma \ref{4l2}(iv) and
$Q\subset\widehat{\Omega}_k$, we know that
$$M_{{\rm HL}}\lf(\chi_{Q\cap(\widehat{\Omega}_k
\setminus\Omega_{k+1})}\r)(x)
\geq\frac1{|Q|}\int_{Q}\chi_{\widehat{\Omega}_k
\setminus\Omega_{k+1}}(z)\,dz
\gs\frac{|Q|-|Q|/2}{|Q|}\gs\frac{\chi_Q(x)}2.$$
From this, \cite[Theorem 3.2]{blyz10} with the
dilation $A$ as in \eqref{4e5}, \eqref{4e11}, Lemma \ref{3l2}
and an argument similar to that used in
the proof of \cite[(3.26)]{lyy16LP},
it follows that, for any $r\in(1,\fz)$,
\begin{align}\label{4e14}
\lf\|\sum_{Q\subset Q_i^k,\,Q\in\mathcal{Q}_k}
e_{Q}\r\|_{L^r(\rn)}
\ls\lf\|\lf[\sum_{Q\subset Q_i^k,\,Q\in\mathcal{Q}_k}
\lf(c_Q\r)^2\chi_{Q\cap(\widehat{\Omega}_k
\setminus\Omega_{k+1})}\r]^{1/2}\r\|_{L^r(\rn)}.
\end{align}

On the other hand, for any $Q\in\mathcal{Q}_k,\,x\in Q$
and $(y,t)\in\widehat{Q}$, by Lemma
\ref{4l2}(iv), Lemma \ref{2l2}(ii) and \eqref{4e8},
we easily know that
$$x-y\in 2^{[w\ell(Q)+u]\va}B_0+2^{[w\ell(Q)+u]\va}B_0
\subset 2^{[w\ell(Q)+u+1]\va}B_0\subset 2^{t\va}B_0.$$
By this and the disjointness of
$\{\widehat{Q}\}_{Q\subset Q_i^k}$,
we find that
\begin{align}\label{4e15}
\sum_{Q\subset Q_i^k,\,Q\in\mathcal{Q}_k}
\lf(c_Q\r)^2\chi_{Q\cap(\widehat{\Omega}_k
\setminus\Omega_{k+1})}(x)
&=\sum_{Q\subset Q_i^k,\,Q\in\mathcal{Q}_k}
\int_{\widehat{Q}}
|\psi_t\ast f(y)|^2\,dy\frac{dm(t)}{2^{\nu t}}
\chi_{Q\cap(\widehat{\Omega}_k
\setminus\Omega_{k+1})}(x)\\
&\ls\lf[S(f)(x)\r]^2\chi_{Q_i^k\cap(\widehat
{\Omega}_k\setminus\Omega_{k+1})}(x).\noz
\end{align}
From the definition of $\widehat{\Omega}_k$
(see \eqref{4e13}), it follows that, for any $r\in(1,\fz)$,
\begin{align*}
\lf|\widehat{\Omega}_k\r|\le2^{(2u\nu+1)r}\int_{\rn}
\lf[M_{{\rm HL}}\lf(\chi_{\Omega_k}\r)(x)\r]^r\,dx
\ls|\Omega_k|,
\end{align*}
which, combined with \eqref{4e15}, implies that
\begin{align}\label{4e16}
&\lf\|\lf\{\sum_{Q
\subset Q_i^k,\,Q\in\mathcal{Q}_k}
\lf(c_Q\r)^2\chi_{Q\cap(\widehat{\Omega}_k
\setminus\Omega_{k+1})}\r\}^
{\frac12}\r\|^r_{L^r(\rn)}\\
&\hs\le\int_{\rn}
\lf[\chi_{Q_i^k\cap(\widehat{\Omega}_k
\setminus\Omega_{k+1})}(x)
\int_{\bigcup_{Q\subset Q_i^k,\,
Q\in\mathcal{Q}_k}\widehat{Q}}
|\psi_t\ast f(y)|^2\,dy\frac{dm(t)}
{2^{\nu t}}\r]^{r/2}\,dx\noz\\
&\hs\ls 2^{kr}\lf|\widehat{\Omega}_k\r|
\ls2^{kr}\lf|\Omega_k\r|<\fz.\noz
\end{align}
For any $N\in\mathbb{N}$, let
$\mathcal{Q}_{k,\,N}:=
\{Q\in\mathcal{Q}_k:\ |\ell(Q)|>N\}$.
Then, replacing
$\sum_{Q\subset Q_i^k,\,Q\in\mathcal{Q}_k}e_{Q}$
by $\sum_{Q\subset Q_i^k,\,Q\in\mathcal{Q}_{k,N}}e_Q$
in \eqref{4e14}, we obtain
\begin{align*}
\lf\|\sum_{Q\subset Q_i^k,\,Q\in
\mathcal{Q}_{k,\,N}}e_{Q}\r\|_{L^r(\rn)}^r
&\ls\lf\|\lf[\sum_{Q\subset Q_i^k,\,Q\in
\mathcal{Q}_{k,\,N}}\lf(c_Q\r)^2\chi_{Q\cap
(\widehat{\Omega}_k\setminus\Omega_{k+1})}
\r]^{1/2}\r\|_{L^r(\rn)}^r\noz\\
&\ls\int_{\rn}\chi_{Q_i^k\cap
(\widehat{\Omega}_k\setminus\Omega_{k+1})}(x)
\lf[\int_{\bigcup_{Q\subset Q_i^k,\,Q\in
\mathcal{Q}_{k,\,N}}\widehat{Q}}|\psi_t\ast f(y)|^2\,dy
\frac{dm(t)}{2^{\nu t}}\r]^{r/2}\,dx.\noz
\end{align*}

From this, \eqref{4e16} and the
Lebesgue dominated convergence theorem, we deduce that
$$\lf\|\sum_{Q\subset Q_i^k,\,Q\in
\mathcal{Q}_{k,\,N}}e_{Q}\r\|_{L^r(\rn)}\rightarrow0$$
as $N\rightarrow\fz$, and hence
$h_i^k=\sum_{Q\subset Q_i^k,\,Q\in\mathcal{Q}_k}e_{Q}$
in $L^r(\rn)$. This, together with \eqref{4e14},
\eqref{4e15}, the definition of $B_i^k$ (see \eqref{4e12})
and Lemma \ref{4l2}(iv), implies that
\begin{equation}\label{4e17}
\lf\|h_i^k\r\|_{L^r(\rn)}\ls\lf\{\int_{\rn}\lf[S(f)(x)\r]^r
\chi_{Q_i^k\cap(\widehat{\Omega}_k
\setminus\Omega_{k+1})}(x)\,dx\r\}^{1/r}\ls2^k\lf|Q_i^k\r|^{1/r}
\le C_1 2^k\lf|B_i^k\r|^{1/r},
\end{equation}
where $C_1$ is a positive constant
independent of $f$, $k$ and $i$.

\emph{Step 3.}
Recall that $\varphi$ has the vanishing moments up to
$s\ge\lfloor\nu/a_-(1/\widetilde{p}_--1)\rfloor$
and so does $e_Q$. For any
$k\in\mathbb{Z},\,i$, $\gamma\in\zz_+^n$ with
$|\gamma|\le s$ and $x\in\rn$,
let $g(x):=x^\gamma\chi_{B_i^k}(x)$.
Clearly, $g\in L^{r'}(\rn)$ with $r\in(1,\fz)$.
Thus, by \eqref{4e12} and the facts that
$(L^{r'}(\rn))^\ast=L^r(\rn)$ and
$$\supp e_Q\subset x_Q+
2^{(w[\ell(Q)-1]+u+2)\va}B_0,$$
we conclude that
\begin{equation*}
\int_{\rn}h_i^k(x)x^\gamma\,dx
=\langle h_i^k,g\rangle
=\sum_{Q\subset Q_i^k,\,Q\in
\mathcal{Q}_k}\langle e_{Q},g\rangle
=\sum_{Q\subset Q_i^k,\,Q\in\mathcal{Q}_k}
\int_{\rn}e_{Q}(x)x^\gamma\,dx=0,
\end{equation*}
namely, $h_i^k$ has the vanishing moments up to $s$,
which, combined with \eqref{4e12} and \eqref{4e17},
implies that $h_i^k$ is a harmless constant multiple of a $(\vp,r,s)$-atom
supported on $B_i^k$.

For any $k\in\mathbb{Z}$ and $i$,
let $\lambda_i^k:=C_1 2^k\|\chi_{B_i^k}\|_{\lv}$ and
$a_i^k:=(\lambda_i^k)^{-1}h_i^k$, where
$C_1$ is as in \eqref{4e17}.
Then
$$f=\sum_{k\in\mathbb{Z}}\sum_i h_i^k=
\sum_{k\in\mathbb{Z}}\sum_i \lambda_i^k a_i^k\qquad {\rm in}\quad \cs'(\rn).$$
Moreover, it is easy to see that, for any $k\in\zz$ and $i$, $a_i^k$ is a $(\vp,r,s)$-atom.

For any $k\in\mathbb{Z}$ and $i$, by the fact that
$|\Qik\cap\Omega_k|\ge\frac{|\Qik|}{2}$
and Lemma \ref{4l7}, we find that
$$\lf\|\chi_{\Qik}\r\|_{\lv}\ls
\lf\|\chi_{\Qik\cap\Omega_k}\r\|_{\lv}.$$
From this, Theorem \ref{3t1},
the mutual disjointness of $\{Q_i^k\}_{k\in\mathbb{Z},\,i}$ and
Lemma \ref{4l2}(iv), we further deduce that
\begin{align*}
\|f\|_{\vh}
&\ls\lf\|\lf\{\sum_{k\in\zz}\sum_{i}
\lf[\frac{\lz_i^k\chi_{\Bik}}{\|\chi_{\Bik}\|_{\lv}}\r]^
{p_-}\r\}^{1/p_-}\r\|_{\lv}\\
&\sim\lf\|\lf[\sum_{k\in\zz}\sum_{i}\lf(2^{k}
\chi_{\Bik}\r)^{{p_-}}\r]^{{1/p_-}}\r\|_{\lv}
\sim\lf\|\lf[\sum_{k\in\zz}\sum_{i}\lf(2^{k}
\chi_{\Qik}\r)^{p_-}\r]^{{1/p_-}}\r\|_{\lv}\\
&\sim\lf\|\sum_{k\in\zz}\sum_{i}
\lf(2^{k}\chi_{\Qik}\r)^{p_-}\r\|
_{L^{\vp/p_-}(\rn)}^{{1/p_-}}
\ls\lf\|\sum_{k\in\zz}\sum_{i}
\lf(2^{k}\chi_{\Qik\cap\Omega_k}\r)^{p_-}\r\|
_{L^{\vp/p_-}(\rn)}^{{1/p_-}}\\
&\ls\lf\|\lf[\sum_{k\in\zz}\lf(2^k\chi_{\Omega_k}\r)^
{p_-}\r]^{1/{p_-}}\r\|_{\lv}
\sim\lf\|\lf[\sum_{k\in\zz}
\lf(2^k\chi_{\Omega_k\setminus\Omega_{k+1}}\r)^
{p_-}\r]^{1/{p_-}}\r\|_{\lv}\\
&\ls\lf\|S(f)\lf[\sum_{k\in\zz}
\chi_{\Omega_k\setminus\Omega_{k+1}}
\r]^{1/{p_-}}\r\|_{\lv}
\sim\lf\|S(f)\r\|_{\lv},
\end{align*}
which implies that $f\in\vh$ and \eqref{4e7} holds true.
This finishes the proof of the sufficiency of Theorem \ref{4t1}.

Next we show the necessity of this theorem.
To this end, let $f\in\vh$. Then, by Lemma \ref{4l1}, we know that $f\in\cs'_0(\rn)$.
On the other hand, by Theorem \ref{3t1}, we conclude that
there exist $\{\lz_i\}_{i\in\nn}\st\mathbb{C}$
and a sequence of $(\vp,r,s)$-atoms, $\{a_i\}_{i\in\nn}$,
supported, respectively, on
$\{B_i\}_{i\in\nn}\st\mathfrak{B}$ such that
\begin{align*}
f=\sum_{i\in\nn}\lz_ia_i
\quad\mathrm{in}\quad\cs'(\rn)
\end{align*}
and
\begin{align*}
\|f\|_{\vh}\sim
\lf\|\lf\{\sum_{i\in\nn}
\lf[\frac{|\lz_i|\chi_{B_i}}{\|\chi_{B_i}\|_{\lv}}\r]^
{p_-}\r\}^{1/p_-}\r\|_{\lv}.
\end{align*}
Let $w$ and $u$ be as in Lemma \ref{4l2}(iv).
Then, by an argument similar to that used in the proof of
\cite[(5.10)]{lyy17hl}, we find that,
for any $i\in \nn$ and $x\in (B_i^{(2^{u-w+2})})^\com$
with $B_i^{(2^{u-w+2})}$ as in \eqref{2e2'},
$$S(a_i)(x)\ls \lf\|\chi_{B_i}\r\|_{\lv}^{-1}
\lf[\HL(\chi_{B_i})(x)\r]^{\frac{\nu+(s+1)a_-}{\nu}},$$
where $M_{\rm HL}$ denotes the Hardy-Littlewood maximal operator as in \eqref{3e2}.
From this, we further deduce that, for any $x\in \rn$,
\begin{align}\label{4e19}
S(f)(x)&\le\sum_{i\in\nn}|\lz_i|S(a_i)(x)\chi_{B_i^{(2^{u-w+2})}}(x)
+\sum_{i\in\nn}|\lz_i|S(a_i)(x)\chi_{(B_i^{(2^{u-w+2})})^\com}(x)\\
&\ls\lf\{\sum_{i\in\nn}\lf[|\lz_i|S(a_i)(x)\chi_{B_i^{(2^{u-w+2})}}(x)\r]
^{p_-}\r\}^{1/p_-}\noz\\
&\qquad+\sum_{i\in\nn}\frac{|\lz_i|}{\|\chi_{B_i}\|_{\lv}}
\lf[\HL(\chi_{B_i})(x)\r]^{\frac{\nu+(s+1)a_-}{\nu}}.\noz
\end{align}

By \cite[Theorem 3.2]{blyz10} with the dilation $A$ as in \eqref{4e5},
we know that, for any $r\in (1,\fz)$ and $i\in \nn$,
$$\lf\|S(a_i)\r\|_{L^r(\rn)}\ls\lf\|a_i\r\|_{L^r(\rn)}.$$
Then, by \eqref{4e19} and an argument similar to that
used in the proof of Theorem \ref{3t1}, we further conclude that
$$\|S(f)\|_{\lv}\ls\|f\|_{\vh},$$
which completes the proof of the necessity
and hence of Theorem \ref{4t1}.
\end{proof}

In what follows, for any $x\in\rn$, let
\begin{equation*}
\rho_{\va}(x):=\sum_{j\in\mathbb{Z}}
2^{\nu j}\chi_{2^{(j+1)\va}B_0\setminus 2^{j\va}B_0}(x)\hspace{0.25cm}
{\rm when}\ x\neq\vec0_n,\hspace{0.35cm} {\rm or\ else}
\hspace{0.25cm}\rho_{\va}(\vec0_n):=0.
\end{equation*}

Recall that, for any given $\va\in [1,\fz)^n$,
$\phi\in\cs(\rn)$, $t\in(0,\fz)$,
$k\in\zz$ and any $f\in\cs'(\rn)$,
the\emph{ anisotropic Peetre maximal function}
$(\phi_{k}^*f)_t$ is defined by setting,
for any $x\in\rn$,
\begin{align*}
\lf(\phi_{k}^*f\r)_t(x)
:=\esup_{y\in\rn}\frac{|(\phi_{-k}\ast f)(x+y)|}
{[1+2^{\nu k}\rho_{\va}(y)]^t}
\end{align*}
and the \emph{$g$-function associated with $(\phi_{k}^*f)_t$}
is defined by setting, for any $x\in\rn$,
\begin{align}\label{4e20}
g_{t,\ast}(f)(x)
:=\lf\{\sum_{k\in\zz}\lf[\lf(\phi_{k}^*f\r)_t(x)\r]^2\r\}^{1/2},
\end{align}
where, for any $k\in \mathbb{Z}$, $\phi_{k}(\cdot)
:=2^{-k\nu}\phi(2^{-k\va}\cdot)$.

The following estimate is just a variant of
\cite[(3.13)]{lyy17}, which originates from
\cite[(2.66)]{u12} and the argument
used in the proof of \cite[Theorem 2.8]{u12},
the details being omitted.

\begin{lemma}\label{4l6}
Let $\phi\in\cs(\rn)$ be a radial function satisfying
\eqref{4e1} and \eqref{4e2}.
Then, for any given $N_0\in\nn$ and $r\in(0,\fz)$, there exists a positive
constant $C_{(N_0,r)}$, which may depends on $N_0$ and $r$, such that,
for any $t\in(0,N_0)$, $\ell\in\zz$, $f\in\cs'(\rn)$ and $x\in\rn$,
it holds true that
\begin{align*}
\lf[\lf(\phi^*_\ell f\r)_t(x)\r]^r
\le C_{(N_0,r)}\sum_{k\in\zz_+}2^{-\nu kN_0r}2^{\nu(k+\ell)}
\int_\rn\frac{|(\phi_{-k-\ell}\ast f)(y)|^r}
{[1+2^{\nu\ell}\rho_{\va}(x-y)]^{tr}}\,dy.
\end{align*}
\end{lemma}

We now prove Theorem \ref{4t2}.

\begin{proof}[Proof of Theorem \ref{4t2}]
First, let $f\in\vh$. Then Lemma \ref{4l1} implies that
$f\in\cs'_0(\rn)$. In addition, repeating the proof
of the necessity of Theorem \ref{4t1} with some slight
modifications, we easily conclude that $g(f)\in \lv$
and $\lf\|g(f)\r\|_{\lv}\ls\|f\|_{\vh}$.
Thus, by Theorem \ref{4t1}, we know that, to prove Theorem \ref{4t2},
it suffices to show that, for any $f\in\cs'_0(\rn)$
with $g(f)\in\lv$,
\begin{align}\label{4e21}
\|S(f)\|_{\lv}\ls\|g(f)\|_{\lv}
\end{align}
holds true.
Indeed, from the fact that, for any $f\in\cs'_0(\rn)$, $t\in(0,\fz)$ and almost every $x\in\rn$,
$S(f)(x)\ls g_{t,\ast}(f)(x)$, it follows that, to show \eqref{4e21}, we only need to prove that
\begin{align}\label{4e22}
\lf\|g_{t,*}(f)\r\|_{\lv}\ls\lf\|g(f)\r\|_{\lv}
\end{align}
holds true for any $f\in\cs'_0(\rn)$ and some $t\in(\frac1{\min\{\widetilde{p}_-,2\}},\fz)$.

Now we show \eqref{4e22}. To this end, assume that $\phi\in\cs(\rn)$
is a radial function and satisfies \eqref{4e1} and \eqref{4e2}.
Obviously, $t\in(\frac1{\min\{\widetilde{p}_-,2\}},\fz)$ implies that there exists
$r\in\lf(0,{\min\{\widetilde{p}_-,2\}}\r)$ such that $t\in(\frac1{r},\fz)$.
Fix $N_0\in(\frac1r,\fz)$.
By this, Lemma \ref{4l6} and the Minkowski inequality,
we know that, for any $x\in\rn$,
\begin{align*}
g_{t,*}(f)(x)
&=\lf\{\sum_{k\in\zz}\lf[\lf(\phi_k^*f\r)_t(x)\r]^2\r\}^{1/2}\\
&\ls\lf[\sum_{k\in\zz}\lf\{\sum_{j\in\zz_+}2^{-\nu jN_0r}2^{\nu(j+k)}
\int_\rn\frac{|(\phi_{-j-k}\ast f)(y)|^r}
{[1+2^{\nu k}\rho_{\va}(x-y)]^{tr}}\,dy\r\}^{2/r}\r]^{1/2}\\
&\ls\lf\{\sum_{j\in\zz_+}2^{-j\nu(N_0r-1)}\lf[\sum_{k\in\zz}2^{\frac{2k\nu}r}
\lf\{\int_\rn\frac{|(\phi_{-j-k}\ast f)(y)|^r}
{[1+2^{\nu k}\rho_{\va}(x-y)]^{tr}}\,dy\r\}^{2/r}\r]^{r/2}\r\}^{1/r},
\end{align*}
which, combined with \eqref{2e8}, implies that
\begin{align*}
&\lf\|g_{t,*}(f)\r\|_{\lv}^{rp_-}\\
&\hs\ls\lf\|\sum_{j\in\zz_+}2^{-j\nu(N_0r-1)}
\lf[\sum_{k\in\zz}2^{\frac{2k\nu}r}
\lf\{\int_\rn\frac{|(\phi_{-j-k}\ast f)(y)|^r}
{[1+2^{\nu k}\rho_{\va}(\cdot-y)]^{tr}}\,dy\r\}^{2/r}\r]^{r/2}\r\|
_{L^{\frac{\vp}{r}}(\rn)}^{p_-}\\
&\hs\ls\sum_{j\in\zz_+}2^{-j\nu(N_0r-1)p_-}
\lf\|\lf[\sum_{k\in\zz}2^{\frac{2k\nu}r}
\lf\{\int_\rn\frac{|(\phi_{-j-k}\ast f)(y)|^r}
{[1+2^{\nu k}\rho_{\va}(\cdot-y)]^{tr}}\,dy\r\}^{2/r}\r]^{r/2}\r\|
_{L^{\frac{\vp}{r}}(\rn)}^{p_-}\\
&\hs\ls\sum_{j\in\zz_+}2^{-j\nu(N_0r-1)p_-}
\lf\|\lf\{\sum_{k\in\zz}2^{\frac{2k\nu}r}
\lf[\sum_{i\in\nn}2^{-\nu itr}\int_{\rho_{\va}(\cdot-y)\sim 2^{\nu(i-k)}}
\lf|(\phi_{-j-k}\ast f)(y)\r|^r\,dy\r]^{2/r}\r\}^{r/2}\r\|
_{L^{\frac{\vp}{r}}(\rn)}^{p_-},
\end{align*}
where $\rho_{\va}(\cdot-y)\sim 2^{\nu(i-k)}$ means that
$\{x\in\rn:\ \rho_{\va}(x-y)<2^{-\nu k}\}$ when $i=0$,
or $\{x\in\rn:\ 2^{\nu(i-k-1)}\le\rho_{\va}(x-y)<2^{\nu(i-k)}\}$ when $i\in\nn$.
Then, from the Minkowski inequality and Lemma \ref{3l2},
we further deduce that
\begin{align*}
&\lf\|g_{t,*}(f)\r\|_{\lv}^{rp_-}\\
&\hs\ls\sum_{j\in\zz_+}2^{-j\nu(N_0r-1)p_-}
\lf\|\sum_{i\in\nn}2^{-\nu itr}\lf\{\sum_{k\in\zz}2^{\frac{2k\nu}r}
\lf[\int_{\rho_{\va}(\cdot-y)\sim 2^{\nu(i-k)}}
\lf|(\phi_{-j-k}\ast f)(y)\r|^r\,dy\r]^{2/r}\r\}^{r/2}\r\|
_{L^{\frac{\vp}{r}}(\rn)}^{p_-}\\
&\hs\ls\sum_{j\in\zz_+}2^{-j\nu(N_0r-1)p_-}
\lf\|\sum_{i\in\nn}2^{\nu i(1-tr)}\lf\{\sum_{k\in\zz}
\lf[\HL\lf(\lf|\phi_{-j-k}\ast f\r|^r\r)\r]^{2/r}\r\}^{r/2}\r\|
_{L^{\frac{\vp}{r}}(\rn)}^{p_-}\\
&\hs\ls\sum_{j\in\zz_+}2^{-j\nu(N_0r-1)p_-}
\sum_{i\in\nn}2^{\nu ip_-(1-tr)}\lf\|\lf(\sum_{k\in\zz}
\lf|\phi_{-j-k}\ast f\r|^2\r)^{r/2}\r\|
_{L^{\frac{\vp}{r}}(\rn)}^{p_-}
\ls\|g(f)\|_{\lv}^{rp_-}.
\end{align*}
This proves \eqref{4e22} and hence
finishes the proof of Theorem \ref{4t2}.
\end{proof}

Applying Theorems \ref{4t1} and \ref{4t2}, we now
prove Theorem \ref{4t3}.
\begin{proof}[Proof of Theorem \ref{4t3}]
By Theorem \ref{4t1} and the fact that, for any $f\in\cs'(\rn)$ and $x\in\rn$,
$S(f)(x)\le g_\lz^{\ast}(f)(x)$,
we find that
the sufficiency of this theorem is obvious. Thus,
to prove this theorem, it suffices to show the necessity.

To this end, let $f\in\vh$ and $\varphi$ be as in the proof
of Theorem \ref{4t2}. Then, by Lemma \ref{4l1},
we find that $f\in\cs'_0(\rn)$. By the fact that
$\lambda\in(1+\frac{2}{{\min\{\widetilde{p}_-,2\}}}, \fz)$,
we conclude that there exists some
$t\in(\frac1{{\min\{\widetilde{p}_-,2\}}},\fz)$ such that $\lz\in(1+2t,\fz)$
and, for any $x\in\rn$,
\begin{align*}
g_\lambda^\ast(f)(x)&=
\lf\{\sum_{k\in\mathbb{Z}}2^{-k\nu}\int_{\rn}
\lf[\frac{2^{k}}{2^{k}+|x-y|_{\va}}\r]^{\lambda\nu}
\lf|f\ast\varphi_{k}(y)\r|^2\,dy\r\}^{1/2}\\
&\ls\lf\{\sum_{k\in\mathbb{Z}}2^{-k\nu}
\lf[\lf(\varphi_{-k}^*f(x)\r)_t\r]^2\int_{\rn}
\lf[1+\frac{\rho_{\va}(x-y)}{2^{\nu k}}\r]^{2t-\lz}\,dy\r\}^{1/2}\\
&\ls\lf\{\sum_{k\in\mathbb{Z}}
\lf[\lf(\varphi_{-k}^*f(x)\r)_t\r]^2\r\}^{1/2}
\sim g_{t,*}(f)(x).
\end{align*}
This, together with \eqref{4e22} and Theorem \ref{4t2},
further implies that $g_\lambda^\ast(f)\in\lv$ and
$$\lf\|g_\lambda^\ast(f)\r\|_{\lv}
\ls\|f\|_{\lv},$$
which completes the proof Theorem \ref{4t3}.
\end{proof}

\section{Finite atomic characterizations of $\vh$\label{s5}}

In this section, we establish the finite atomic characterizations of $\vh$.
We begin with introducing the following notion
of anisotropic mixed-norm finite atomic Hardy spaces $\vfah$.

\begin{definition}\label{5d1}
Let $\va\in[1,\fz)^n$, $\vp\in(0,\fz)^n$, $r\in (1,\fz]$
and $s$ be as in \eqref{3e1}. The \emph{anisotropic mixed-norm
finite atomic Hardy space} $\vfah$ is defined to be the set of all
$f\in\cs'(\rn)$ satisfying that there exist $I\in\nn$,
$\{\lz_i\}_{i\in[1,I]\cap\nn}\st\mathbb{C}$ and
a finite sequence of $(\vp,r,s)$-atoms,
$\{a_i\}_{i\in[1,I]\cap\nn}$, supported, respectively, on
$\{B_i\}_{i\in[1,I]\cap\nn}\st\mathfrak{B}$
such that
\begin{align*}
f=\sum_{i=1}^I\lambda_ia_i
\quad\mathrm{in}\quad\cs'(\rn).
\end{align*}
Moreover, for any $f\in\vfah$, let
\begin{align*}
\|f\|_{\vfah}:=
{\inf}\lf\|\lf\{\sum_{i=1}^{I}
\lf[\frac{|\lz_i|\chi_{B_i}}{\|\chi_{B_i}\|_{\lv}}\r]^
{p_-}\r\}^{1/p_-}\r\|_{\lv},
\end{align*}
where $p_-$ is as in \eqref{2e10} and the
infimum is taken over all decompositions of $f$ as above.
\end{definition}

By \cite[p.\,13, Theorem 3.6]{mb03} with $A$ as in \eqref{4e5},
we immediately obtain the following conclusion, the details being omitted.

\begin{lemma}\label{5l5} For any given $N\in\nn$, let $M_N$ be as in Definition
\ref{2d4}.
\begin{enumerate}
\item[\rm(i)]
Let $p\in(1,\fz]$. Then, for any given $N\in \nn$, there exists a positive constant $C_{(p,N)}$,
depending on $p$ and $N$, such that, for any $f\in L^p(\rn)$,
\begin{align*}
\lf\|M_N(f)\r\|_{L^p(\rn)}
\le C_{(p,N)}\|f\|_{L^p(\rn)}.
\end{align*}
\item[\rm(ii)] For any given $N\in \nn$, there exists a positive constant $C_{(N)}$,
depending on $N$, such that, for any $\lambda\in (0,\fz)$ and $f\in L^1(\rn)$,
\begin{align*}
\lf|\lf\{x\in\rn:\ M_N(f)(x)>\lz\r\}\r|\le C_{(N)}\frac{\|f\|_{L^1(\rn)}}\lambda.
\end{align*}
\end{enumerate}
\end{lemma}

Obviously, by Theorem \ref{3t1}, we easily know that,
for any $\va\in[1,\fz)^n$, $\vp\in(0,\fz)^n$,
$s\in\zz_+$ as in \eqref{3e1} and
$r\in(\max\{p_+,1\},\fz]$ with $p_+$ as in \eqref{2e10},
the set $H_{\va,{\rm fin}}^{\vp,r,s}(\rn)$ is dense in $\vh$ with
respect to the quasi-norm $\|\cdot\|_{\vh}$.
From this, we deduce the following density of $\vh$.

\begin{lemma}\label{5l6}
If $\va\in [1,\fz)^n$ and $\vp\in (0,\fz)^n$, then,
\begin{enumerate}
\item[{\rm (i)}] for any $q\in[1,\fz]$, $\vh\cap L^q(\rn)$
is dense in $\vh$;
\item[{\rm (ii)}] $\vh\cap C_c^\fz(\rn)$
is dense in $\vh$.
\end{enumerate}
\end{lemma}

\begin{proof}
We first prove (i).
For any $\vp\in(0,\fz)^n$, by the density
of the set $H_{\va,{\rm fin}}^{\vp,\fz,s}(\rn)$
in $\vh$ and $H_{\va,{\rm fin}}^{\vp,\fz,s}(\rn)\subset L^q(\rn)$ for any $q\in[1,\fz]$,
we easily know that
$\vh\cap L^q(\rn)$ is dense in $\vh$. This finishes the proof of (i).

Next we show (ii).
To this end, we first prove that, for any $\varphi\in\cs(\rn)$ with
$\int_{\rn}\varphi(x)\,dx\neq0$ and $f\in \vh$, as $k\to-\fz$,
\begin{align}\label{5e22}
f\ast\varphi_k\rightarrow f \hspace{0.5cm} {\rm in}\quad \vh.
\end{align}
To show this, we first assume that $f\in \vh\cap L^2(\rn)$.
In this case, to prove \eqref{5e22}, it suffices to show that, for almost every $x\in\rn$,
as $k\to-\fz$,
\begin{align}\label{5e23}
M_N\lf(f\ast\varphi_k-f\r)(x)\rightarrow0,
\end{align}
where $N:=N_{\vp}+2$. Indeed, by the fact that, for any $k\in \nn$,
$f\ast\varphi_k-f\in L^2(\rn)$ and Lemma \ref{5l5}(i), we find that, for any $k\in\mathbb{Z}$,
$M_N(f\ast\varphi_k-f)\in L^2(\rn)$. By this,
\cite[p.\,39, Lemma 6.6]{mb03} with $A$ as in \eqref{4e5}, \eqref{5e23}
and the Lebesgue dominated convergence theorem, we know that, for any $f\in \vh\cap L^2(\rn)$,
\eqref{5e22} holds true.

Now we show \eqref{5e23}.
Notice that, if $h$ is continuous and has compact support,
then $h$ is uniformly continuous on $\rn$. Therefore, for any $\delta\in(0,\fz)$,
there exists $\eta\in(0,\fz)$ such that, for any $y\in\rn$
satisfying $|y|_{\va}<\eta$ and $x\in\rn$,
$$|h(x-y)-h(x)|<\frac\delta{2\|\varphi\|_{L^1(\rn)}}.$$
Without loss of generality,
we may assume that $\int_{\rn}\varphi(x)\,dx=1$. Then, for any $k\in\zz$,
$\int_{\rn}\varphi_k(x)\,dx=1$. This further implies that, for any $k\in\zz$ and $x\in\rn$,
\begin{align}\label{5e24}
|h\ast\varphi_k(x)-h(x)|
&\le\int_{|y|_{\va}<\eta}|h(x-y)-h(x)||\varphi_k(y)|\,dy
+\int_{|y|_{\va}\ge\eta}\cdots\\
&<\frac\delta2+2\|h\|_{L^\fz(\rn)}
\int_{|y|_{\va}\ge 2^{-k}\eta}|\varphi(y)|\,dy.\noz
\end{align}
On the other hand, by the integrability of $\varphi$, we know that
there exists $\widetilde{k}\in\zz$ such that,
for any $k\in(-\fz,\widetilde{k}]\cap\zz$,
$$2\|h\|_{L^\fz(\rn)}
\int_{|y|_{\va}\ge 2^{-k}\eta}|\varphi(y)|\,dy<\frac\delta2,$$
which, combined with \eqref{5e24}, implies that
$\lim_{k\to-\fz}|h\ast\varphi_k(x)-h(x)|=0$ holds true uniformly for any $x\in\rn$.
Thus, $\|h\ast\varphi_k-h\|_{L^\fz(\rn)}\to0$ as $k\to-\fz$. From this and Lemma \ref{5l5}(i), we deduce that
\begin{align}\label{5e25}
\lf\|M_N\lf(h\ast\varphi_k-h\r)\r\|_{L^\fz(\rn)}\ls
\|h\ast\varphi_k-h\|_{L^\fz(\rn)}\to0\hspace{0.3cm} {\rm as}\ k\to-\fz.
\end{align}
For any given $\epsilon\in(0,\fz)$, there exists a continuous function $h$
with compact support such that
$$\|f-h\|^2_{L^2(\rn)}<\epsilon.$$
Then \eqref{5e25} and \cite[p.\,39, Lemma 6.6]{mb03} with $A$ as
in \eqref{4e5} imply that there exists a positive constant $C_2$ such that,
for any $x\in\rn$,
\begin{align*}
&\limsup_{k\to-\fz}M_N\lf(f\ast\varphi_k-f\r)(x)\\
&\hs\le\sup_{k\in\zz}M_N\lf((f-h)\ast\varphi_k\r)(x)
+\limsup_{k\to-\fz}M_N\lf(h\ast\varphi_k-h\r)(x)
+M_N(h-f)(x)\le C_2M_{N_{\vp}}(h-f)(x).
\end{align*}
By this and Lemma \ref{5l5}(ii),
we conclude that there exists a
positive constant $C_3$ such that, for any $\lz\in(0,\fz)$,
\begin{align*}
&\lf|\lf\{x\in\rn:\ \limsup_{k\to-\fz}
M_N\lf(f\ast\varphi_k-f\r)(x)>\lz\r\}\r|\\
&\hs\le\lf|\lf\{x\in\rn:\
M_{N_{\vp}}(h-f)(x)>\frac\lz{C_2}\r\}\r|
\le C_3\frac{\|f-h\|^2_{L^2(\rn)}}{\lz^2}\le C_3\frac{\epsilon}{\lz^2}.
\end{align*}
This implies that, for any
$f\in \vh\cap L^2(\rn)$, \eqref{5e23} holds true.

When $f\in \vh$, by an argument similar to that used in the proof of
\cite[Lemma 5.2(ii)]{lyy16}, we easily find that \eqref{5e22} also holds true.

Notice that, if $f\in H_{\va,{\rm fin}}^{\vp,r,s}(\rn)$
and $\varphi\in C_c^\fz(\rn)$
with $\int_{\rn}\varphi(x)\,dx\neq0$, then, for any $k\in\zz$,
$$f\ast\varphi_k\in C_c^\fz(\rn)\cap \vh$$
and, by \eqref{5e22},
\begin{align*}
f\ast\varphi_k\rightarrow f \quad\mathrm{in}\quad \vh
\quad\mathrm{as}\quad k\to-\fz.
\end{align*}
From this and the density of the set $H_{\va,{\rm fin}}^{\vp,r,s}(\rn)$ in
$\vh$, it follows that
$C_c^\fz(\rn)\cap \vh$ is dense in $\vh$. This finishes the proof
of (ii) and hence of Lemma \ref{5l6}.
\end{proof}

The following Lemmas \ref{5l4} and \ref{5l1} are from Theorem \ref{3t1}
and its proof, which are of independent interest and are needed
in the proof of Theorem \ref{5t1} below.

\begin{lemma}\label{5l4}
Let $\va\in[1,\fz)^n$, $\vp\in(0,\fz)^n$, $r\in(\max\{p_+,1\},\fz]$
with $p_+$ as in \eqref{2e10}, $s$ be as in \eqref{3e1}
and $N$ as in \eqref{2e11}. Then there exists a positive
constant $C$ such that, for any $(\vp,r,s)$-atom $a$,
$$\lf\|M_N(a)\r\|_{\lv}\le C.$$
\end{lemma}

\begin{lemma}\label{5l1}
Let $\va\in[1,\fz)^n$, $\vp\in(0,\fz)^n$, $r\in (1,\fz]$
and $s$ be as in \eqref{3e1}. Then,
for any $f\in \vh\cap L^{r}(\rn)$, there exist
$\{\lz_{j,k}\}_{j\in\zz,\,k\in\nn}\subset\mathbb{C}$,
$\{B_{j,k}\}_{j\in \zz,k\in\nn}\st\mathfrak{B}$ and
$(\vp,\fz,s)$-atoms $\{a_{j,k}\}_{j\in\zz,\,k\in\nn}$ such that
\begin{align*}
f=\sum_{j\in\zz}\sum_{k\in\nn}
\lz_{j,k}a_{j,k}\quad{\rm in}\quad \cs'(\rn),
\end{align*}
\begin{align}\label{5e1}
\supp a_{j,k}\subset B_{j,k}\hspace{0.2cm}
for\ any\ j\in\zz\ and\ k\in\nn,\hspace{0.2cm}
\CO_j=\bigcup_{k\in\mathbb{N}}B_{j,k}\hspace{0.2cm}
for\ any\ j\in\zz,
\end{align}
here $\CO_j:=\{x\in\rn:\ M_N(f)(x)>2^j\}$
with $N$ as in \eqref{2e11},
\begin{align}\label{5e2}
B_{j,k}^{(1/4)}\bigcap B_{j,m}^{(1/4)}
&=\emptyset\hspace{0.3cm} for\ any\ j\in\zz\ and\ k,\,m\in\nn\ with\ k\neq m,\\
&\qquad and\ B_{j,k}^{(1/4)}\ and B_{j,m}^{(1/4)}\ as\ in\ \eqref{2e2'}\ with\ \dz=1/4,\noz
\end{align}
and
\begin{align}\label{5e3}
\sharp\lf\{m\in\mathbb{N}:\
B_{j,k}\cap B_{j,m}\neq\emptyset\r\}\le
R\hspace{0.2cm} for\ any\hspace{0.2cm} k\in\mathbb{N}
\end{align}
with $R$ being a
positive constant independent of $j$ and $f$.
Moreover, there exists a positive constant $C$, independent of $f$, such that,
for any $j\in\zz$, $k\in\nn$ and almost every $x\in\rn$,
\begin{align}\label{5e4}
\lf|\lz_{j,k}a_{j,k}(x)\r|\le C2^j
\end{align}
and
\begin{align}\label{5e5}
\lf\|\lf\{\sum_{j\in\zz}\sum_{k\in\nn}
\lf[\frac{|\lz_{j,k}|\chi_{B_{j,k}}}
{\|\chi_{B_{j,k}}\|_{\lv}}\r]^
{p_-}\r\}^{1/p_-}\r\|_{\lv}
\le C\|f\|_{\vh},
\end{align}
where $p_-$ is as in \eqref{2e10}.
\end{lemma}

\begin{remark}\label{5r1}
For any $j\in\zz$, $k\in\nn$ and $s$ as in \eqref{3e1},
let $\eta_{j,k}$ be the same as in
the proof of Theorem \ref{3t1}. Then, for any
$f\in \vh\cap L^{r}(\rn)$ with $r\in(1,\fz]$,
by an argument similar to that used in the
proof of Theorem \ref{3t1}, together with Lemma \ref{3l4},
we know that there exists a unique polynomial $c_{j,k}\in\cp_s(\rn)$
such that, for any $q\in\cp_s(\rn)$,
\begin{align}\label{5e6}
\lf\langle f,q\eta_{j,k}\r\rangle=
\lf\langle c_{j,k},q\eta_{j,k}\r\rangle=
\int_\rn c_{j,k}(x)q(x)\eta_{j,k}(x)\,dx.
\end{align}
In addition,
for any  $i$, $k\in\nn$ and $j\in\zz$, let
$c_{j+1,k,i}$ be the orthogonal projection of
$(f-c_{j+1,i})\eta_{j,k}$ on $\cp_{s}(\rn)$
with respect to the norm defined by \eqref{3e28}, namely,
$c_{j+1,k,i}$ is the unique element of $\cp_s(\rn)$
such that, for any $q\in\cp_s(\rn)$,
\begin{align}\label{5e7}
\int_\rn \lf[f(x)-c_{j+1,i}(x)\r]\eta_{j,k}(x)q(x)
\eta_{j+1,i}(x)\,dx=\int_\rn c_{j+1,k,i}(x)q(x)
\eta_{j+1,i}(x)\,dx
\end{align}
and, for any $k\in\mathbb{N}$ and $j\in\mathbb{Z}$,
\begin{align}\label{5e8}
\lz_{j,k}a_{j,k}=(f-c_{j,k})\eta_{j,k}-
\sum_{i\in\mathbb{N}}\lf[(f-c_{j+1,i})\eta_{j,k}
-c_{j+1,k,i}\r]\eta_{j+1,i}.
\end{align}
\end{remark}

From \eqref{3e13}, \eqref{3e14} and their proofs,
we deduce the following Lemmas \ref{5l2}
and \ref{5l3} (see also the proofs of \cite[p.\,104, ($23'$) and p.\,108, (35)]{s93}), the details being omitted.

\begin{lemma}\label{5l2}
Let $\va\in[1,\fz)^n$ and $\vp\in(0,\fz)^n$. Then
there exists a positive constant $C$ such that,
for any $j\in\zz$, $k\in\nn$ and $f\in\vh$,
$$\sup_{y\in\rn}\lf|c_{j,k}(y)\eta_{j,k}(y)\r|\le C\sup_{y\in U_j^k}
M_N(f)(y)\le C2^j,$$
where $N\in \nn$, $M_N$ is as in Definition \ref{2d4} and, for any $j\in\zz$ and $k\in\nn$,
$\eta_{j,k}$ is as in the proof of Theorem \ref{3t1}, $c_{j,k}$ as in Remark \ref{5r1},
$\CO_j$ and $B_{j,k}$ as in the proof of Theorem \ref{3t1}, $B_{j,k}^{(2)}$ as in \eqref{2e2'}
with $\dz=2$,
and $U_j^k:=B_{j,k}^{(2)}\cap (\CO_j)^\com$.
\end{lemma}

\begin{lemma}\label{5l3}
Let $\va$, $\vp$, $f$ and $M_N$ be as in Lemma \ref{5l2}. Then there exists a positive constant $C$,
independent of $f$, such that, for any $j\in\zz$ and $i,\,k\in\nn$,
$$\sup_{y\in\rn}\lf|c_{j+1,k,i}(y)\eta_{j+1,i}(y)\r|\le C\sup_
{y\in\widetilde{U}_j^k}M_N(f)(y)\le C2^{j+1},$$
where, for any $j\in\zz$ and $i,\,k\in\nn$,
$\eta_{j+1,i}$ is as in the proof of Theorem \ref{3t1}, $c_{j+1,k,i}$ as in Remark \ref{5r1},
$\CO_{j+1}$ and $B_{j+1,k}$ as in the proof of Theorem \ref{3t1} with $j$ replaced by $j+1$, $B_{j+1,k}^{(2)}$ as in \eqref{2e2'} with $\dz=2$,
and $\widetilde{U}_j^k:=B_{j+1,k}^{(2)}\cap(\CO_{j+1})^\com$.
\end{lemma}

In what follows, denote by $C(\rn)$
the \emph{set of all continuous functions}.
Then we have the following finite atomic characterizations
of $\vh$, which extends \cite[Theorem 3.1 and Remark 3.3]{msv08} and
\cite[Theorem 5.6]{gly08} to the present
setting of anisotropic mixed-norm Hardy spaces.

\begin{theorem}\label{5t1}
Let $\va\in [1,\fz)^n,\,\vp\in(0,\fz)^n$,
$r\in(\max\{p_+,1\},\fz]$ with $p_+$ as in
\eqref{2e10} and $s$ be as in \eqref{3e1}.
\begin{enumerate}
\item[{\rm (i)}]
If $r\in(\max\{p_+,1\},\fz)$, then $\|\cdot\|_{\vfah}$
and $\|\cdot\|_{\vh}$ are equivalent quasi-norms on $\vfah$;
\item[{\rm (ii)}]
$\|\cdot\|_{\vfahfz}$
and $\|\cdot\|_{\vh}$ are equivalent quasi-norms on
$\vfahfz\cap C(\rn)$.
\end{enumerate}
\end{theorem}

\begin{remark}
Recall that
Bownik et al. in \cite[Theorem 6.2]{blyz08}
established the finite atomic characterizations of the weighted
anisotropic Hardy space $H_w^p(\rn;A)$ with $w$
being a Muckenhoupt weight (see \cite[Definition 2.5]{blyz08}).
As was mentioned in \cite[p.\,3077]{blyz08}, if $w:\equiv1$,
then $H_w^p(\rn;A)=\vAh$.
By this and Proposition \ref{2r4'}, we know that,
when $\vp:=(\overbrace{p,\ldots,p}^{n\ \rm times})$,
where $p\in(0,1]$, Theorem \ref{5t1} is just \cite[Theorem 6.2]{blyz08}
with the weight $w:\equiv1$ and $A$ as in \eqref{4e5}.
\end{remark}

Now we proof Theorem \ref{5t1}.
\begin{proof}[Proof of Theorem \ref{5t1}]
Let $\va\in [1,\fz)^n,\,\vp\in(0,\fz)^n$,
$r\in(\max\{p_+,1\},\fz]$ with $p_+$ as in
\eqref{2e10} and $s$ be as in \eqref{3e1}.
Then, by Theorem \ref{3t1}, we find that
$\vfah\subset \vh$ and, for any $f\in \vfah$,
$\|f\|_{\vh}\ls\|f\|_{\vfah}$.
Therefore, to prove Theorem \ref{5t1}, it suffices to show that,
for any $f\in \vfah$ when $r\in(\max\{p_+,1\},\fz)$
and, for any $f\in [\vfahfz\cap C(\rn)]$ when $r=\fz$,
$$\|f\|_{\vfah}\ls\|f\|_{\vh}.$$
We prove this by the following three steps.

\emph{Step 1.}
Let $r\in(\max\{p_+,1\},\fz]$. Without loss of generality,
we may assume that $f\in \vfah$ and $\|f\|_{\vh}=1$. Clearly, there
exists some $j_0\in\zz$ such that $\supp f\subset 2^{j_0\va}B_{0}$,
due to the fact that $f$ has compact support, where $B_{0}$ denotes the unit ball of $\rn$.
In the remainder of this section, we always let
$N:=N_{\vp}$ with $N_{\vp}$ as in Definition \ref{2d5} and, for any $j\in\zz$, let
$$\CO_j:=\lf\{x\in\rn:\ M_N(f)(x)>2^j\r\}.$$
Notice that
$f\in \vh\cap L^{\widetilde{r}}(\rn)$, where
$\widetilde{r}:=r$ when $r\in(\max\{p_+,1\},\fz)$ and
$\widetilde{r}:=2$ when $r=\fz$. Then, by Lemma \ref{5l1},
we conclude that there exist $\{\lz_{j,k}\}_{j\in\zz,\,k\in\nn}
\subset\mathbb{C}$ and a sequence of
$(\vp,\fz,s)$-atoms, $\{a_{j,k}\}_{j\in\zz,\,k\in\nn}$, such that
\begin{align}\label{5e9}
f=\sum_{j\in\zz}\sum_{k\in\nn}\lz_{j,k}a_{j,k}\quad{\rm in}\quad \cs'(\rn),
\end{align}
and \eqref{5e1} through \eqref{5e5} also hold true.

By this and an argument similar to that used in the proof of Step 2
of the proof of \cite[Theorem 5.7]{lyy16},
we know that there exists a positive constant $C_4$ such that, for any
$x\in (2^{ (j_0+4)\va}B_0)^\com$,
\begin{align}\label{5e10}
M_N(f)(x)\le C_4\lf\|\chi_{2^{j_0\va}B_{0}}\r\|_{\lv}^{-1}.
\end{align}
Let
\begin{align}\label{5e11}
\widetilde{j}:=
\sup\lf\{j\in\zz:\ 2^k<C_4\lf\|\chi_{2^{j_0\va}B_{0}}\r\|_{\lv}^{-1}\r\}
\end{align}
with $C_4$ as in \eqref{5e10}. Then, from \eqref{5e10}, we deduce that,
for any $j\in(\widetilde{j},\fz]\cap\zz$,
\begin{align}\label{5e12}
\CO_j\subset 2^{ (j_0+4)\va}B_0.
\end{align}
Using $\wz{j}$ as in \eqref{5e11}, we rewrite \eqref{5e9} as
\begin{align}\label{5e13}
f=\sum_{j=-\fz}^{\widetilde{j}}\sum_{k\in\nn}\lz_{j,k}a_{j,k}+
\sum_{j=\widetilde{j}+1}^\fz\sum_{k\in\nn}\lz_{j,k}a_{j,k}=:h+\ell
\quad{\rm in}\quad \cs'(\rn).
\end{align}
In the remainder of this step, we devote to proving that $h$ is a
$(\vp,\fz,s)$-atom multiplied by a harmless constant independent of $f$. For this purpose,
from \eqref{5e12}, it is easy to see that
$\supp \ell\subset\cup_{j=\widetilde{j}+1}^\fz\CO_j
\subset 2^{ (j_0+4)\va}B_0$. By this, the fact that
$\supp f\subset 2^{ (j_0+4)\va}B_0$ and \eqref{5e13}, we know
that $\supp h\subset 2^{ (j_0+4)\va}B_0$.

On the other hand, by the H\"{o}lder inequality, we find that,
for any $r\in(\max\{p_+, 1\}, \fz]$ and $r_1\in(\max\{p_+, 1\},r)$,
$$\int_{\rn}|f(x)|^{r_1}\,dx
\le\lf|2^{j_0\va}B_{0}\r|^{1-\frac{r_1}r}\|f\|_{L^r(\rn)}^{r_1}<\fz.$$
This, together with the facts that
$\supp f\subset 2^{j_0\va}B_{0}$ and that $f$
has vanishing moments up to order $s$, further implies that
$f$ is a harmless constant multiple of a $(1,r_1,s)$-atom.
By this and Lemma \ref{5l4},
we know that
$M_N(f)\in L^1(\rn)$. Therefore, by \eqref{5e3}, \eqref{5e1}, \eqref{5e12} and
\eqref{5e4}, we conclude that
$$\int_{\rn}\sum_{j=\widetilde{j}+1}^\fz\sum_{k\in\nn}
\lf|\lz_{j,k}a_{j,k}(x)x^\alpha\r|\,dx
\ls\sum_{j\in\zz}2^j|\CO_j|\ls\lf\|M_N(f)\r\|_{L^1(\rn)}<\fz.$$
From this and the vanishing moments of $a_{j,k}$, we deduce that $\ell$
has vanishing moments up to $s$ and hence so does $h$ by \eqref{5e13}.
Moreover, from \eqref{5e3}, \eqref{5e4} and \eqref{5e11},
it follows that, for any $x\in\rn$,
$$|h(x)|\ls\sum_{j=-\fz}^{\widetilde{j}}2^j
\ls\lf\|\chi_{2^{j_0\va}B_{0}}\r\|_{\lv}^{-1}.$$
Thus, there exists a positive constant $C_5$, independent of $f$, such that
$h/C_5$ is a $(\vp,\fz,s)$-atom and also a
$(\vp,r,s)$-atom for any $\vp\in (0,\fz)^n$, $r\in(\max\{p_+,1\},\fz]$ and $s$
as in \eqref{3e1}.

\emph{Step 2.}
This step is aimed to prove (i). To this end,
for any $J\in(\widetilde{j},\fz)\cap\zz$ and
$j\in[\widetilde{j}+1,J]\cap\zz$ with $\wz{j}$
as in \eqref{5e11}, let
\begin{align*}
I_{(J,j)}:=\lf\{k\in\nn:\ |k|+|j|\le J\r\}
\hspace{0.3cm} {\rm and}\hspace{0.3cm} \ell_{(J)}:=\sum_{j=\widetilde{j}+1}^{J}
\sum_{k\in I_{(J,j)}}\lz_{j,k}a_{j,k}.
\end{align*}
For any $r\in(\max\{p_+,1\},\fz)$,
we first show that $\ell\in L^r(\rn)$.
Indeed, for any $x\in\rn$,
since $\rn=\bigcup_{i\in\zz}(\CO_i\setminus\CO_{i+1})$,
it follows that there exists an $i_0\in\zz$ such that
$x\in(\CO_{i_0}\setminus\CO_{i_0+1})$. Notice that,
for any $j\in(i_0,\fz)\cap\zz$,
$\supp a_{j,k}\subset B_{j,k}\subset \CO_j\subset\CO_{i_0+1}$.
Then \eqref{5e3} and \eqref{5e4} imply that,
for any $x\in(\CO_{i_0}\setminus\CO_{i_0+1})$,
$$\lf|\ell(x)\r|\le\sum_{j=\widetilde{j}+1}^\fz\sum_{k\in\nn}|\lz_{j,k}a_{j,k}(x)|
\ls\sum_{j\le i_0}2^j\ls2^{i_0}\ls M_N(f)(x).$$
Since $f\in L^r(\rn)$, from Lemma \ref{5l5}(i), it follows that
$M_N(f)\in L^r(\rn)$. Therefore, by the Lebesgue
dominated convergence theorem, we find that $\ell_{(J)}$
converges to $\ell$ in $L^r(\rn)$ as $J\to\fz$.
This implies that, for any given
$\epsilon\in(0,1)$, there exists
a $J\in[\widetilde{j}+1,\fz)\cap\zz$ large enough,
depending on $\epsilon$, such that
$[\ell-\ell_{(J)}]/\epsilon$ is a $(\vp,r,s)$-atom and hence
$f=h+\ell_{(J)}+[\ell-\ell_{(J)}]$
is a finite linear combination of $(\vp,r,s)$-atoms.
By this, Step 1 and \eqref{5e5}, we conclude that
$$\|f\|_{\vfah}
\ls C_5+
\lf\|\lf\{\sum_{j=\wz{j}+1}^{J}\sum_{k\in I_{(J,j)}}
\lf[\frac{|\lz_{j,k}|\chi_{B_{j,k}}}
{\|\chi_{B_{j,k}}\|_{\lv}}\r]^
{p_-}\r\}^{1/p_-}\r\|_{\lv}
+\epsilon\ls1,$$
which completes the proof of (i).

\emph{Step 3.}
In this step, we prove (ii). For this purpose,
let $f\in \vfahfz\cap C(\rn)$.
Then, by \eqref{5e8}, we find that, for any $j\in\zz$ and $k\in\nn$,
$a_{j,k}$ is continuous. Moreover, by the fact that
there exists a positive constant $C_{(n,N)}$,
depending only on $n$ and $N$, such that, for any $x\in\rn$,
\begin{align}\label{5e16}
M_N(f)(x)\le C_{(n,N)}\|f\|_{L^\fz(\rn)},
\end{align}
we easily know that, for any $j\in\zz$ satisfying
$2^j\ge C_{(n,N)}\|f\|_{L^\fz(\rn)}$, the level set $\CO_j$ is empty.
Let
$$\widehat{j}:=\sup\lf\{j\in\zz:\ 2^j< C_{(n,N)}\|f\|_{L^\fz(\rn)}\r\}.$$
Then the index $j$ in the sum defining
$\ell$ runs only over $j\in\{\widetilde{j}+1,\ldots,\widehat{j}\}$.

Let $\epsilon\in(0,\fz)$. Then the fact that $f$ is uniformly continuous implies that there exists
a $\delta\in(0,\fz)$ such that
$|f(x)-f(y)|<\epsilon$ whenever $|x-y|_{\va}<\delta$.
Furthermore, for this $\epsilon$, let
$$\ell_1^\epsilon:=\sum_{j=\widetilde{j}+1}^{\widehat{j}}
\sum_{k\in E_1^{(j,\delta)}}\lz_{j,k}a_{j,k}\hspace{0.4cm} {\rm and}\hspace{0.4cm}
\ell_2^\epsilon:=\sum_{j=\widetilde{j}+1}^{\widehat{j}}
\sum_{k\in E_2^{(j,\delta)}}\lz_{j,k}a_{j,k},$$
where, for any $j\in\{\widetilde{j}+1,\ldots,\widehat{j}\}$,
$$E_1^{(j,\delta)}:=\lf\{k\in\nn:\ r_{j,k}\ge\delta\r\}\quad
\mathrm{and}\quad E_2^{(j,\delta)}:=\lf\{k\in\nn:\ r_{j,k}<\delta\r\}$$
with $x_{j,k}$ and
$r_{j,k}$ being the center and the radius of $B_{j,k}$, respectively.

Next we give a finite decomposition of $f$.
Clearly, by \eqref{5e2} and \eqref{5e12}, we know that,
for any fixed $j\in\{\widetilde{j}+1,\ldots,\widehat{j}\}$,
$E_1^{(j,\delta)}$ is a finite set and hence $\ell_1^\epsilon$
is a finite linear combination of continuous $(\vp,\fz,s)$-atoms.
Then, by \eqref{5e5}, we have
\begin{align}\label{5e14}
\lf\|\lf\{\sum_{j=\widetilde{j}+1}^{\widehat{j}}\sum_{k\in E_1^{(j,\delta)}}
\lf[\frac{|\lz_{j,k}|\chi_{B_{j,k}}}{\|\chi_{B_{j,k}}\|_{\lv}}\r]^
{p_-}\r\}^{1/p_-}\r\|_{\lv}
\ls\|f\|_{\vh}.
\end{align}

In addition, for any $j\in\{\widetilde{j}+1,\ldots,\widehat{j}\}$,
$k\in\nn$ satisfying $r_{j,k}<\delta$ and $x\in B_{j,k}$,
we have $|f(x)-f(x_{j,k})|<\epsilon$. By \eqref{5e6} and the fact that
$\supp\eta_{j,k}\subset B_{j,k}$,
we conclude that, for any $q\in\cp_s(\rn)$,
$$\frac 1{\int_\rn\eta_{j,k}(x)\,dx}\int_\rn
\lf[\widetilde{f}(x)-\widetilde{c}_{j,k}(x)\r]q(x)\eta_{j,k}(x)\,dx=0,$$
where, for any $x\in\rn$,
$$\widetilde{f}(x):=\lf[f(x)-f(x_{j,k})\r]
\chi_{B_{j,k}}(x)
\hspace{0.3cm} {\rm and}\hspace{0.3cm}\widetilde{c}_{j,k}(x):=
c_{j,k}(x)-f(x_{j,k}).$$
Since \eqref{5e16} and the fact that, for any $x\in\rn$, $|\widetilde{f}(x)|<\epsilon$
imply that, for any $x\in\rn$, $M_N(\widetilde{f})(x)\ls\epsilon$,
from Lemma \ref{5l2}, it follows that
\begin{align}\label{5e17}
\sup_{y\in\rn}\lf|\widetilde{c}_{j,k}(y)\eta_{j,k}(y)\r|
\ls\sup_{y\in\rn}M_N\lf(\widetilde{f}\r)(y)\ls\epsilon.
\end{align}
Similarly to Remark \ref{5r1},
for any $j\in\{\widetilde{j}+1,\ldots,\widehat{j}\}$,
$k\in E_2^{(j,\delta)}$ and $i\in\nn$, let
$\widetilde{c}_{j+1,k,i}$
be the orthogonal projection of
$(\widetilde{f}-\widetilde{c}_{j+1,i})\eta_{j,k}$ on
$\cp_{s}(\rn)$ with
respect to the norm in \eqref{3e28}.
Then, for any $q\in\cp_{s}(\rn)$,
\begin{align}\label{5e18}
\int_\rn \lf[\widetilde{f}(x)-\widetilde{c}_{j+1,i}(x)\r]\eta_{j,k}(x)q(x)
\eta_{j+1,i}(x)\,dx=\int_\rn \widetilde{c}_{j+1,k,i}(x)q(x)
\eta_{j+1,i}(x)\,dx.
\end{align}
By the fact that $\supp\eta_{j,k}\subset B_{j,k}$, we have
$[\widetilde{f}-\widetilde{c}_{j+1,i}]\eta_{j,k}
=[f-c_{j+1,i}]\eta_{j,k}$.
From this, \eqref{5e7} and \eqref{5e18}, we deduce that
$\widetilde{c}_{j+1,k,i}=c_{j+1,k,i}$. Then, by Lemma \ref{5l3},
we know that
\begin{align}\label{5e19}
\sup_{y\in\rn}\lf|\widetilde{c}_{j+1,k,i}(y)\eta_{j+1,i}(y)\r|
\ls\sup_{y\in\rn}M_N(\widetilde{f})(y)\ls\epsilon.
\end{align}
Moreover, by \eqref{5e8} and
$\sum_{i\in\nn}\eta_{j+1,i}=\chi_{\CO_{j+1}}$, we conclude that
\begin{align*}
\lz_{j,k}a_{j,k}&=(f-c_{j,k})\eta_{j,k}-
\sum_{i\in\mathbb{N}}\lf[(f-c_{j+1,i})
\eta_{j,k}-c_{j+1,k,i}\r]\eta_{j+1,i}\\
&=\eta_{j,k}\widetilde{f}\chi_{\CO_{j+1}^
\com}-\widetilde{c}_{j,k}\eta_{j,k}+\eta_{j,k}\sum
_{i\in\nn}\widetilde{c}_{j+1,i}\eta_{j+1,i}+\sum_
{i\in\mathbb{N}}\widetilde{c}_{j+1,k,i}\eta_{j+1,i},
\end{align*}
which, combined with \eqref{5e17}, \eqref{5e19}
and \eqref{5e3}, further implies that, for any
$j\in\{\widetilde{j}+1,\ldots,\widehat{j}\}$,
$k\in E_2^{(j,\delta)}$ and $x\in B_{j,k}$,
$|\lz_{j,k}a_{j,k}(x)|\ls\epsilon$.

Then, by \eqref{5e1} and \eqref{5e3}, we easily know that
there exists a positive constant $C_6$, independent of $f$, such that,
for any $x\in \rn$
\begin{align}\label{5e15}
\lf|\ell_2^\epsilon(x)\r|\le C_6\sum_{j=\widetilde{j}+1}^{\widehat{j}}
\epsilon=C_6\lf[\widehat{j}-\widetilde{j}\r]\epsilon.
\end{align}
Therefore, the arbitrariness of $\epsilon \in (0, \fz)$ implies that we split $\ell$ into a continuous
part and a part which is pointwisely uniformly arbitrarily small, namely,
$\ell=\ell_1^\epsilon+\ell_2^\epsilon$. Thus,
$\ell$ is continuous and, by Step 1, $h=f-\ell$ is a $C_5$ multiple of a continuous
$(\vp,\fz,s)$-atom.

Notice that $\ell$ and $\ell_1^\epsilon$ are both continuous and have vanishing
moments up to order $s$ and hence so does $\ell_2^\epsilon$. This, combined with
the fact that $\supp\ell_2^\epsilon\subset 2^{ (j_0+4)\va}B_0$ and
\eqref{5e15}, further implies that we can choose
$\epsilon$ small enough such that $\ell_2^\epsilon$ is an arbitrarily
small multiple of a continuous $(\vp,\fz,s)$-atom. Indeed,
$\ell_2^\epsilon=\lz^{(\epsilon)} a^{(\epsilon)}$, where
$$\lz^{(\epsilon)}:=C_6\lf[\widehat{j}-\widetilde{j}\r]
\epsilon\lf\|\chi_{2^{ (j_0+4)\va}B_0}\r\|_{\lv}^{-1}$$
and $a^{(\epsilon)}$ is a continuous $(\vp,\fz,s)$-atom. In this case,
$f=h+\ell_1^\epsilon+\ell_2^\epsilon$ is just a finite atomic
decomposition of $f$. Then, by \eqref{5e14} and the fact that $h/C_5$ is
a $(\vp,\fz,s)$-atom, we find that
$$\|f\|_{H_{\va,{\rm fin}}^{\vp,\fz,s}(\rn)}\lesssim
\|h\|_{H_{\va,{\rm fin}}^{\vp,\fz,s}(\rn)}
+\lf\|\ell_1^\epsilon\r\|_{H_{\va,{\rm fin}}^{\va,\fz,s}(\rn)}
+\lf\|\ell_2^\epsilon\r\|_{H_{\va,{\rm fin}}^{\vp,\fz,s}(\rn)}\ls1,$$
which completes the proof of (ii) and hence of Theorem \ref{5t1}.
\end{proof}

\section{Some applications\label{s6}}

As applications, in this section, we first establish a criterion on the
boundedness of sublinear operators from $\vh$ into a quasi-Banach
space. Applying this criterion, we further obtain
the boundedness of anisotropic convolutional $\delta$-type and
non-convolutional $\bz$-order Calder\'on-Zygmund operators
from $\vh$ to itself [or to $\lv$].

Recall that a complete vector space is called a \emph{quasi-Banach space} $\mathcal{B}$ if its quasi-norm $\|\cdot\|_{\mathcal{B}}$ satisfies
\begin{enumerate}
\item[{\rm (i)}] $\|f\|_{\mathcal{B}}=0\Longleftrightarrow f$ is the zero element of $\mathcal{B}$;
\item[{\rm (ii)}] there exists a positive constant $H\in[1,\fz)$ such that, for any
$f,\,g\in\mathcal{B}$,
$$\|f+g\|_{\mathcal{B}}\le H(\|f\|_{\mathcal{B}}+\|g\|_{\mathcal{B}}).$$
\end{enumerate}
Clearly, when $H=1$, a quasi-Banach space $\mathcal{B}$ is just a Banach space.
Moreover, for any given $\gamma\in(0,1]$, a
\emph{$\gamma$-quasi-Banach space} ${\mathcal{B}_{\gamma}}$ is a quasi-Banach space equipped
with a quasi-norm $\|\cdot\|_{\mathcal{B}_{\gamma}}$ satisfying that there exists a constant $C\in[1,\fz)$
such that, for any $t\in \nn$ and $\{f_i\}_{i=1}^{t}\st\mathcal{B}_{\gamma}$, $\|\sum_{i=1}^t f_i\|_{\mathcal{B}_{\gamma}}^{\gamma}\le
C \sum_{i=1}^t \|f_i\|_{\mathcal{B}_{\gamma}}^{\gamma}$ holds true
(see \cite{zy08, zy09, ky14, ylk17}).

Let $\mathcal{B}_{\gamma}$ be a $\gamma$-quasi-Banach space with
$\gamma\in(0,1]$ and $\mathcal{Y}$ a linear space. An operator
$T$ from $\mathcal{Y}$ to $\mathcal{B}_{\gamma}$ is said to be
$\mathcal{B}_{\gamma}$-\emph{sublinear} if
there exists a positive constant $C$ such that, for any $t\in \nn$,
$\{\mu_{i}\}_{i=1}^t\st \mathbb{C}$ and $\{f_{i}\}_{i=1}^t\st\mathcal{Y}$,
$$\lf\|T\lf(\sum_{i=1}^t \mu_i f_i\r)\r\|_{\mathcal{B}_{\gamma}}^{\gamma}\le
C\sum_{i=1}^t |\mu_i|^{\gamma}\lf\|T(f_i)\r\|_{\mathcal{B}_{\gamma}}^{\gamma}$$
and, for any $f,\,g\in \mathcal{Y}$,
$\|T(f)-T(g)\|_{\mathcal{B}_{\gamma}}\le C\|T(f-g)\|_{\mathcal{B}_{\gamma}}$
(see \cite{zy08, zy09, ky14, ylk17}).
Clearly, for any
$\gamma\in (0,1]$, the linearity of $T$ implies its $\mathcal{B}_{\gamma}$-sublinearity.

As an application of the finite atomic characterizations of $\vh$ obtained
in Section \ref{s5} (see Theorem \ref{5t1}),
we establish the following criterion on the boundedness of sublinear
operators from $\vh$ into a quasi-Banach
space $\mathcal{B}_{\gamma}$.

\begin{theorem}\label{6t1}
Assume that $\va\in [1,\fz)^n$, $\vp\in (0,\fz)^n$, $r\in(\max\{p_+,1\},\fz]$
with $p_+$ as in \eqref{2e10}, $\gamma\in (0,1]$, $s$ is
as in \eqref{3e1} and $\mathcal{B}_{\gamma}$ a $\gamma$-quasi-Banach space.
If either of the following two statements holds true:
\begin{enumerate}
\item[{\rm (i)}] $r\in(\max\{p_+,1\},\fz)$ and
$T:\ \vfah\to\mathcal{B}_{\gamma}$
is a $\mathcal{B}_{\gamma}$-sublinear operator satisfying that
there exists a positive constant $C_7$ such that,
for any $f\in \vfah$,
\begin{align}\label{6e1}
\|T(f)\|_{\mathcal{B}_{\gamma}}\le C_7\|f\|_{\vfah};
\end{align}
\item[{\rm (ii)}]
$T:\ \vfahfz\cap C(\rn)\to\mathcal{B}_{\gamma}$
is a $\mathcal{B}_{\gamma}$-sublinear operator satisfying that
there exists a positive constant $C_8$ such that,
for any $f\in \vfahfz\cap C(\rn)$,
$$\|T(f)\|_{\mathcal{B}_{\gamma}}\le C_8\|f\|_{\vfahfz},$$
\end{enumerate}
then $T$ uniquely extends to a bounded $\mathcal{B}_{\gamma}$-sublinear operator from $\vh$
into $\mathcal{B}_{\gamma}$. Moreover, there exists a positive constant $C_9$ such that,
for any $f\in \vh$,
$$\|T(f)\|_{\mathcal{B}_{\gamma}}\le C_9\|f\|_{\vh}.$$
\end{theorem}

The following conclusion is an immediate corollary of Theorem \ref{6t1}, which
extends the corresponding results of Meda et al. \cite[Corollary 3.4]{msv08}
and Grafakos et al. \cite[Theorem 5.9]{gly08} as well as Ky \cite[Theorem 3.5]{ky14}
(see also \cite[Theorem 1.6.9]{ylk17})
to the present setting, the details being omitted.

\begin{corollary}\label{6c1}
Let $\va$, $\vp$, $r$, $\gamma$, $s$ and $\mathcal{B}_{\gamma}$ be as
in Theorem \ref{6t1}. If either of the following two statements holds true:
\begin{enumerate}
\item[{\rm (i)}] $r\in(\max\{p_+,1\},\fz)$ and $T$ is a $\mathcal{B}_{\gamma}$-sublinear
operator from $\vfah$ to $\mathcal{B}_{\gamma}$ satisfying
$$\sup\lf\{\|T(a)\|_{\mathcal{B}_{\gamma}}:\
a\ is\ any\ (\vp,r,s){\text-}atom\r\}<\fz;$$
\item[{\rm(ii)}] $T$ is a $\mathcal{B}_{\gamma}$-sublinear
operator defined on all continuous $(\vp,\fz,s)$-atoms satisfying
$$\sup\lf\{\|T(a)\|_{\mathcal{B}_{\gamma}}:\
a\ is\ any\ continuous\ (\vp,\fz,s){\text-}atom\r\}<\fz,$$
\end{enumerate}
then $T$ has a unique bounded $\mathcal{B}_{\gamma}$-sublinear
extension $\widetilde{T}$ from $\vh$ to $\mathcal{B}_{\gamma}$.
\end{corollary}

We now prove Theorem \ref{6t1}.

\begin{proof}[Proof of Theorem \ref{6t1}]
To show (i), let $r\in(\max\{p_+,1\},\fz)$ and $f\in\vh$.
Then, by the density of
$\vfah$ in $\vh$, we know that there exists a Cauchy sequence
$\{f_k\}_{k\in\nn}\subset \vfah$ such that
$$\lim_{k\to\fz}\lf\|f_k-f\r\|_{\vh}=0.$$
By this, \eqref{6e1} and Theorem \ref{5t1}(i),
we conclude that, as $k$, $\ell\to\fz$,
\begin{align*}
\lf\|T(f_k)-T(f_{\ell})\r\|_{\mathcal{B}_{\gamma}}
\ls\lf\|T(f_k-f_{\ell})\r\|_{\mathcal{B}_{\gamma}}\ls
\lf\|f_k-f_{\ell}\r\|_{\vfah}\sim\lf\|f_k-f_{\ell}\r\|_{\vh}\to0,
\end{align*}
which implies that $\{T(f_k)\}_{k\in\nn}$ is a Cauchy sequence in $\mathcal{B}_{\gamma}$.
Therefore, by the completeness of $\mathcal{B}_{\gamma}$, we find that there exists
some $h\in\mathcal{B}_{\gamma}$ such that $h=\lim_{k\to\fz}T(f_k)$
in $\mathcal{B}_{\gamma}$.
Then let $T(f):=h$. From this, \eqref{6e1} and Theorem \ref{5t1}(i) again, we
further deduce that
\begin{align*}
\|T(f)\|_{\mathcal{B}_{\gamma}}^{\gamma}&\ls\limsup_{k\to\fz}\lf[\lf\|T(f)-T(f_k)\r\|_{\mathcal{B}_{\gamma}}^{\gamma}
+\lf\|T(f_k)\r\|_{\mathcal{B}_{\gamma}}^{\gamma}\r]\ls\limsup_{k\to\fz}\lf\|T(f_k)\r\|_{\mathcal{B}_{\gamma}}^{\gamma}\\
&\ls\limsup_{k\to\fz}\lf\|f_k\r\|_{\vfah}^{\gamma}
\sim\lim_{k\to\fz}\lf\|f_k\r\|_{\vh}^{\gamma}\sim\|f\|_{\vh}^{\gamma},
\end{align*}
which completes the proof of (i).

We now prove (ii).
First, by the proof of \cite[Theorem 6.13(ii)]{lyy16} with some slight modifications,
we easily know that
$H_{\va,{\rm fin}}^{\vp,\fz,s}(\rn)\cap C(\rn)$ is dense in $\vh$.
Then, from this and an argument similar to that used in the proof of (i),
we conclude that (ii) holds true.
This finishes the proof of (ii) and hence of Theorem \ref{6t1}.
\end{proof}

Let $\va\in [1,\fz)^n$. For any $\delta\in (0,1]$,
an \emph{anisotropic convolutional
$\delta$-type Calder\'{o}n-Zygmund operator} $T$ from \cite{bil66,f66}
is a linear operator, which is bounded on $L^2(\rn)$ with kernel
$k\in \cs'(\rn)$ coinciding with a locally integrable
function on $\rn\setminus\{\vec{0}_n\}$ and satisfying that
there exists a positive constant $C$ such that,
for any $x,\,y\in \rn$ with $|x|_{\va}>2|y|_{\va}$,
$$|k(x-y)-k(x)|\le C\frac{|y|_{\va}^{\delta}}{|x|_{\va}^{\nu+\delta}}$$
and, for any $f\in L^2(\rn)$, $T(f)(x):={\rm p.\,v.}\,k\ast f(x)$.

Via borrowing some ideas from the proof of Yan et al. \cite[Theorem 7.4]{yyyz16} and
the criterion established in Theorem \ref{6t1} and Corollary \ref{6c1},
we obtain the boundedness of anisotropic convolutional $\delta$-type Calder\'{o}n-Zygmund operators
from $\vh$ to itself (see Theorem \ref{6t2} below) or to $\lv$ (see Theorem \ref{6t3} below),
which extends the corresponding results of Fefferman and Stein \cite[Theorem 12]{fs72}
as well as Yan et al. \cite[Theorem 7.4]{yyyz16} to the present setting.

\begin{theorem}\label{6t2}
Let $\va\in [1,\fz)^n$, $\vp\in (0,1]^n$, $\delta\in(0,1]$
and $\widetilde{p}_-\in(\frac\nu{\nu+\delta},1]$ with $\widetilde{p}_-$ as in \eqref{3e1}.
Let $T$ be an anisotropic convolutional $\dz$-type Calder\'on-Zygmund operator.
Then there exists a positive constant $C$
such that, for any $f\in \vh$,
$$\|T(f)\|_{\vh}\le C\|f\|_{\vh}.$$
\end{theorem}

\begin{theorem}\label{6t3}
Let $\va\in [1,\fz)^n$, $\vp\in (0,1]^n$, $\delta\in(0,1]$
and $\widetilde{p}_-\in(\frac\nu{\nu+\delta},1]$ with $\widetilde{p}_-$ as in \eqref{3e1}.
Let $T$ be an anisotropic convolutional $\dz$-type Calder\'on-Zygmund operator,
then there exists a positive constant $C$
such that, for any $f\in \vh$,
$$\|T(f)\|_{\lv}\le C\|f\|_{\vh}.$$
\end{theorem}

\begin{remark}\label{6r3}
We point out that the boundedness of the anisotropic convolutional
$\dz$-type Calder\'on-Zygmund operator on the mixed-norm Lebesgue space
$\lv$ with $\vp\in(1,\fz)^n$ is still unknown so far.
\end{remark}

Now we prove Theorem \ref{6t2}.

\begin{proof}[Proof of Theorem \ref{6t2}]
Let $f\in H^{\vp,2,s}_{\va,\rm {fin}}(\rn)$ with $s$ as in \eqref{3e1}.
Then, without loss of generality,
we may assume that $\|f\|_{\vh}=1$. Notice that
$f\in \vh\cap L^2(\rn)$,
by an argument similar to that used in the proof of Theorem \ref{3t1},
we find that there exist a sequence of $(\vp,2,s)$-atoms,
$\{a_{k}\}_{k\in\mathbb{N}}$, supported, respectively, on
$\{B_k\}_{k\in\nn}:=\{B_{\va}(x_k,r_k)\}_{k\in\nn}\st \mathfrak{B}$ and
$\{\lz_{k}\}_{k\in\mathbb{N}}\subset\mathbb{C}$
such that
\begin{align}\label{6e2}
f=\sum_{k\in\mathbb{N}}\lambda_{k}a_{k}\quad{\rm in}\quad L^2(\rn)
\end{align}
and
\begin{align*}
\lf\|\lf\{\sum_{k\in\nn}
\lf[\frac{|\lz_{k}|\chi_{B_{k}}}
{\|\chi_{B_{k}}\|_{\lv}}\r]^
{p_-}\r\}^{1/p_-}\r\|_{\lv}
\lesssim\|f\|_{\vh}\ls 1
\end{align*}
with $p_-$ as in \eqref{2e10}. From the boundedness of $T$ on $L^2(\rn)$ and \eqref{6e2},
we deduce that, for any $f\in H^{\vp,2,s}_{\va,\rm {fin}}(\rn)$,
\begin{align}\label{6e4}
T(f)=\sum_{k\in\mathbb{N}}\lambda_{k}T(a_{k})\quad{\rm in}\quad L^2(\rn).
\end{align}
Thus, by Theorem \ref{6t1}(i) and Lemma \ref{3l8},
to prove Theorem \ref{6t2}, we only need to show that, for any $f\in H^{\vp,2,s}_{\va,\rm {fin}}(\rn)$,
\begin{align}\label{6e3}
\|T(f)\|_{\vh}\sim\|M_0(T(f))\|_{\lv}\ls 1,
\end{align}
where $M_0$ is as in \eqref{3e16}.

To this end,
from \eqref{6e4}, it is easy to see that
$$
\lf\|M_0(T(f))\r\|_{\lv}\ls \lf\|\sum_{k\in\mathbb{N}}|\lambda_{k}|M_0(T(a_{k}))
\chi_{B_k^{(2)}}\r\|_{\lv}+\lf\|\sum_{k\in\mathbb{N}}|\lambda_{k}|M_0(T(a_{k}))
\chi_{(B_k^{(2)})^\com}\r\|_{\lv}=:\textrm{I}+\textrm{II},
$$
where $B_k^{(2)}$ is as in \eqref{2e2'} with $\dz=2$.

For $\textrm{I}$,
by Lemma \ref{3l1} and the fact that $T$ is bounded on $L^2(\rn)$,
we conclude that, for any $k\in \nn$,
$$
\lf\|M_0\lf(T(a_k)\r)\chi_{B_k^{(2)}}\r\|_{L^2(\rn)}\ls \lf\|M_{\rm {HL}}
(T(a_k))\chi_{B_k^{(2)}}\r\|_{L^2(\rn)}\ls\lf\|T(a_k)\r\|_{L^2(\rn)}
\ls \lf\|a_k\r\|_{L^2(\rn)}\ls \frac{|B_k|^{1/2}}{\|\chi_{B_k}\|_{\lv}},
$$
where $M_{\rm {HL}}$ denotes the Hardy-Littlewood maximal operator
as in \eqref{3e2}.
This, combined with Lemma \ref{3l6}, implies that
\begin{align}\label{6e16}
\textrm{I}&\le \lf\|\lf\{\sum_{k\in\mathbb{N}}\lf[|\lambda_{k}|M_0(T(a_{k}))
\chi_{B_k^{(2)}}\r]^{p_-}\r\}^{1/p_-}\r\|_{\lv}
\ls\lf\|\lf\{\sum_{k\in\nn}
\lf[\frac{|\lz_{k}|\chi_{B_{k}}}
{\|\chi_{B_{k}}\|_{\lv}}\r]^
{p_-}\r\}^{1/p_-}\r\|_{\lv}
\ls 1.
\end{align}

Next, we deal with \textrm{II}. To this end, for any $t\in (0,\fz)$,
let $k^{(t)}:=k*\Phi_t$, where $k$ is the kernel of $T$ and $\Phi$ as in \eqref{3e16}.
Then, we claim that $k^{(t)}$ satisfies the same conditions as $k$.
Indeed, since, for any $t \in (0,\fz)$ and $f\in L^2(\rn)$, $k^{(t)}*f=k*\Phi_t*f$,
we have
\begin{align*}
\lf\|k^{(t)}*f\r\|_{L^2(\rn)}&=\|k*\Phi_t*f\|_{L^2(\rn)}
=\|k*(\Phi_t*f)\|_{L^2(\rn)}\\
&\ls \|\Phi_t*f\|_{L^2(\rn)}\ls\|f\|_{L^2(\rn)}.
\end{align*}
On the other hand, by an argument similar to that
used in the proof of \cite[p.\,117, Lemma]{s93},
we conclude that, for any $x,\,y\in \rn$ with
$|x|_{\va}>2|y|_{\va}$,
$$\lf|k^{(t)}(x-y)-k^{(t)}(x)\r|\le C\frac{|y|_{\va}^{\delta}}{|x|_{\va}^{\nu+\delta}},$$
where $C$ is a positive constant independent of $t,\,x$ and $y$.
Therefore, the above claim holds true.

Now, by the vanishing moment condition of $a_k$ and the H\"{o}lder
inequality, we know that, for any $x\in(B_k^{(2)})^{\com}$,
\begin{align*}
M_0(T(a_k))(x)&=\sup_{t\in (0,\fz)}\lf|\Phi_t*(k*a_k)(x)\r|
=\sup_{t\in (0,\fz)}\lf|k^{(t)}*a_k(x)\r|\\
&\le \sup_{t\in (0,\fz)}\int_{B_k}\lf|k^{(t)}(x-y)-k^{(t)}(x-x_k)\r||a_k(y)|\,dy\\
&\ls\int_{B_k}\frac{|y-x_k|_{\va}^{\delta}}{|x-x_k|_{\va}^{\nu+\delta}}|a_k(y)|\,dy
\ls \frac{r_k^{\delta}}{|x-x_k|_{\va}^{\nu+\delta}}\|a_k\|_{L^2(\rn)}|B_k|^{1/2}\\
&\ls \frac{r_k^{\nu+\delta}}{|x-x_k|_{\va}^{\nu+\delta}}\frac1{\|\chi_{B_k}\|_{\lv}}
\ls \lf[M_{\rm HL}\lf(\chi_{B_k}\r)(x)\r]^{\f{\nu+\delta}{\nu}}\frac1{\|\chi_{B_k}\|_{\lv}},
\end{align*}
which implies that
\begin{align}\label{6e8}
M_0(T(a_k))(x)\chi_{(B_k^{(2)})^{\com}}(x)\ls
\lf[M_{\rm HL}\lf(\chi_{B_k}\r)(x)\r]^{\f{\nu+\delta}{\nu}}\frac1{\|\chi_{B_k}\|_{\lv}}.
\end{align}
Therefore, by \eqref{2e8}, the fact that $\widetilde{p}_-\in(\frac\nu{\nu+\delta},1]$ and Lemma \ref{3l2},
we find that
\begin{align*}
\textrm{II}&\ls \lf\|\sum_{k\in\mathbb{N}}\frac{|\lambda_{k}|}{\|\chi_{B_k}\|_{\lv}}
\lf[M_{\rm HL}\lf(\chi_{B_k}\r)\r]^{\f{\nu+\delta}{\nu}}\r\|_{\lv}
\sim \lf\|\lf\{\sum_{k\in\mathbb{N}}\frac{|\lambda_{k}|}{\|\chi_{B_k}\|_{\lv}}
\lf[M_{\rm HL}\lf(\chi_{B_k}\r)\r]^{\f{\nu+\delta}{\nu}}\r\}^{\frac{\nu}{\nu+\delta}}
\r\|_{L^{\f{\nu+\delta}{\nu}\vp}(\rn)}^{\f{\nu+\delta}{\nu}}\\
&\ls \lf\|\lf\{\sum_{k\in\mathbb{N}}\frac{|\lambda_{k}|\chi_{B_k}}{\|\chi_{B_k}\|_{\lv}}
\r\}^{\frac{\nu}{\nu+\delta}}
\r\|_{L^{\f{\nu+\delta}{\nu}\vp}(\rn)}^{\f{\nu+\delta}{\nu}}\noz
\ls \lf\|\lf\{\sum_{k\in\nn}
\lf[\frac{|\lz_{k}|\chi_{B_{k}}}
{\|\chi_{B_{k}}\|_{\lv}}\r]^
{p_-}\r\}^{1/p_-}\r\|_{\lv}
\ls 1.
\end{align*}
Finally, combining the above estimates of \textrm{I} and \textrm{II},
we obtain \eqref{6e3}.
This finishes the proof of Theorem \ref{6t2}.
\end{proof}

Now we prove Theorem \ref{6t3}.

\begin{proof}[Proof of Theorem \ref{6t3}]
Let $\vp\in (0,1]^n$ and $s$ be as in \eqref{3e1}. By Corollary \ref{6c1}(i),
to prove this theorem, we know that it suffices to show that, for any $(\vp,2,s)$-atom $a$,
\begin{align}\label{6e5}
\lf\|T(a)\r\|_{\lv}\ls 1.
\end{align}

Now we show \eqref{6e5}. Let $\supp a\st B\in \mathfrak{B}$.
From the fact that $T$ is bounded on $L^2(\rn)$ and $a\in L^2(\rn)$,
we deduce that
$$\lf\|T(a)\chi_{B^{(2)}}\r\|_{L^2(\rn)}\ls \|a\|_{L^2(\rn)}
\ls \frac{|B|^{1/2}}{\|\chi_{B}\|_{\lv}},$$
where $B^{(2)}$ is as in \eqref{2e2'} with $\dz=2$, which, together with Lemma \ref{3l6}, implies that
\begin{align}\label{6e6}
\lf\|T(a)\chi_{B^{(2)}}\r\|_{\lv}\ls 1.
\end{align}

On the other hand, when $x\in (B^{(2)})^{\com}$,
by an argument similar to that used in the estimation
of \eqref{6e8}, we conclude that
$$\lf|T(a)(x)\r|\ls \lf[M_{\rm HL}(\chi_{B})(x)\r]
^{\f{\nu+\delta}{\nu}}\frac1{\|\chi_{B}\|_{\lv}}.$$
Therefore, by \eqref{2e8}, the fact that $\widetilde{p}_-\in(\frac\nu{\nu+\delta},1]$ and Lemma \ref{3l1},
we know that
\begin{align*}
\lf\|T(a)\chi_{(B^{(2)})^{\com}}\r\|_{\lv}&\ls \lf\|\lf[M_{\rm HL}(\chi_{B})\r]
^{\f{\nu+\delta}{\nu}}\frac1{\|\chi_{B}\|_{\lv}}\r\|_{\lv}\ls 1,
\end{align*}
which, combined with \eqref{6e6}, further implies \eqref{6e5} holds true and hence
completes the proof of Theorem \ref{6t3}.
\end{proof}

We now introduce a class of anisotropic $\beta$-order Calder\'{o}n-Zygmund operators
as follows.

\begin{definition}
Let $\va\in[1,\fz)^n$. For any given $\beta\in (0,\fz)\setminus\nn$, a linear operator
$T$ is called an \emph{anisotropic $\beta$-order Calder\'{o}n-Zygmund operator} if $T$ is bounded
on $L^2(\rn)$ and its kernel
$$\mathcal{K}:\ (\rn\times\rn)\setminus\{(x,x):\ x\in\rn\}\to \mathbb{C}$$
satisfies that there exists a positive constant $C$ such that, for any $\az\in\zz_+^n$ with
$|\alpha|\le \lfloor\bz\rfloor$ and $x,\,y,\,z\in \rn$,
\begin{align}\label{6e7}
|\pa^{\az}_x\mathcal{K}(x,y)-\pa^{\az}_x\mathcal{K}(x,z)|\le C\frac{|y-z|_{\va}^{\bz-\lfloor\bz\rfloor}}
{|x-y|_{\va}^{\nu+\bz}}
\quad{\rm when}\quad |x-y|_{\va}>2|y-z|_{\va}
\end{align}
and, for any $f\in L^2(\rn)$ with compact support and $x\notin \supp f$,
$$T(f)(x)=\int_{\supp f}\mathcal{K}(x,y)f(y)\,dy.$$
\end{definition}

For any $l\in\nn$, an operator $T$ is said to have the \emph{vanishing moment
condition up to order $l$} if, for any $a\in L^2(\rn)$ with compact support and
satisfying that, for any $\gamma\in\zz^+_n$ with $|\gamma|\le l$, $\int_{\rn}x^{\gamma}a(x)\,dx=0$,
it holds true that $\int_{\rn}x^{\gamma}T(a)(x)\,dx=0$.

Then we have the following boundedness of anisotropic $\beta$-order Calder\'{o}n-Zygmund operators
$T$ from $\vh$ to itself (see Theorem \ref{6t4} below) or to $\lv$ (see Theorem \ref{6t5} below),
which extends the corresponding results of Stefanov and Torres \cite[Theorem 1]{st04}
as well as Yan et al. \cite[Theorem 7.6]{yyyz16} to the present setting.

\begin{theorem}\label{6t4}
Let $\va\in[1,\fz)^n,\,\vp\in(0,1]^n$, $\bz\in(0,\fz)\setminus\nn$, $\widetilde{p}_-\in(\frac{\nu}{\nu+\bz},\frac{\nu}{\nu+\lfloor\bz\rfloor a_-}]$
with $\widetilde{p}_-$ as in $\eqref{3e1}$ and $a_-$ as in \eqref{2e9} and $T$ be an anisotropic $\beta$-order Calder\'{o}n-Zygmund operator having the
vanishing moment conditions up to order $\lfloor\bz\rfloor$. Then there exists a positive constant $C$
such that, for any $f\in \vh$,
$$\|T(f)\|_{\vh}\le C\|f\|_{\vh}.$$
\end{theorem}

\begin{theorem}\label{6t5}
Let $\va\in[1,\fz)^n,\,\vp\in(0,1]^n$, $\bz\in(0,\fz)\setminus\nn$, $\widetilde{p}_-\in(\frac{\nu}{\nu+\bz},\frac{\nu}{\nu+\lfloor\bz\rfloor a_-}]$
with $\widetilde{p}_-$ as in $\eqref{3e1}$ and $a_-$ as in \eqref{2e9} and $T$ be an anisotropic $\beta$-order Calder\'{o}n-Zygmund operator. Then there exists a positive constant $C$
such that, for any $f\in \vh$,
$$\|T(f)\|_{\lv}\le C\|f\|_{\vh}.$$
\end{theorem}

\begin{remark}\label{6r2}
\begin{enumerate}
\item[(i)]
When $\bz:=\delta\in(0,1)$, then $\az=(\overbrace{0,\ldots,0}^{n\ \mathrm{times}})$ and the operator $T$ in
Theorem \ref{6t4} (or Theorem \ref{6t5}) becomes an anisotropic non-convolutional
$\delta$-type Calder\'{o}n-Zygmund operator.
Thus, from Theorem \ref{6t4} (or Theorem \ref{6t5}), we deduce that, for any $\va\in[1,\fz)^n,\,\vp\in(0,1]^n,\,\delta\in(0,1]$ and
$\widetilde{p}_-\in(\frac{\nu}{\nu+\delta},1]$ with $\widetilde{p}_-$ as in \eqref{3e1},
the anisotropic non-convolutional $\delta$-type
Calder\'{o}n-Zygmund operator is bounded from $\vh$ to itself [or to $\lv$].
In addition, we point out that the boundedness of the anisotropic
$\bz$-order Calder\'{o}n-Zygmund operators on
the mixed-norm Lebesgue space $\lv$ with $\vp\in(1,\fz)^n$ is still unknown so far.

\item[(ii)] When $\va:=(\overbrace{1,\ldots,1}^{n\ \rm times})$ and $
\vp:=(\overbrace{p,\ldots,p}^{n\ \rm times})\in (0,\fz)^n$, $\vh$ and $\lv$
become the classical isotropic Hardy space $H^p(\rn)$ and Lebesgue space $L^p(\rn)$, respectively, and
$T$ becomes the classical $\dz$-type Calder\'on-Zygmund operator.
In this case, we know that, if $\delta\in(0,1]$ and
$p\in(\frac n{n+\delta},1]$, then Theorems \ref{6t2} and \ref{6t3}
and (i) of this remark imply the boundedness of
the classical $\dz$-type Calder\'on-Zygmund operator
from $H^p(\rn)$ to itself and
from $H^p(\rn)$ to $L^p(\rn)$ for any $\delta\in(0,1]$ and $p\in(\frac n{n+\delta},1]$,
which is a well-known result (see, for example, \cite{a86,lu,s93}).
\end{enumerate}
\end{remark}

Now we prove Theorem \ref{6t4}.
\begin{proof}[Proof of Theorem \ref{6t4}]
By an argument similar to that used in
the proof of Theorem \ref{6t2}, we know that, to show Theorem \ref{6t4}, we only need to prove that
\begin{align}\label{6e15}
\lf\|\sum_{k\in\mathbb{N}}|\lambda_{k}|M_0(T(a_{k}))\r\|_{\lv}\ls 1,
\end{align}
where $\{\lambda_k\}_{k\in\nn}$ and $\{a_k\}_{k\in\nn}$ are the same as in the
proof of Theorem \ref{6t2} and $M_0$ is as in \eqref{3e16}. For this purpose, first, it is easy to see that
\begin{align*}
\lf\|\sum_{k\in\mathbb{N}}|\lambda_{k}|M_0(T(a_{k}))\r\|_{\lv}
&\ls \lf\|\sum_{k\in\mathbb{N}}|\lambda_{k}|M_0(T(a_{k}))
\chi_{B_k^{(4)}}\r\|_{\lv}+\lf\|\sum_{k\in\mathbb{N}}|\lambda_{k}|M_0(T(a_{k}))
\chi_{(B_k^{(4)})^\com}\r\|_{\lv}\\
&=:\textrm{I}+\textrm{II},
\end{align*}
where, for any $k\in\nn$, $B_k:=B_{\va}(x_k,r_k)$
is the same as in the proof of Theorem \ref{6t2} and $B_k^{(4)}$ as in \eqref{2e2'} with $\dz=4$.

For \textrm{I}, by a proof similar to that of \eqref{6e16}, we conclude that $\textrm{I}\ls 1$.

Next, we deal with \textrm{II}. To this end, from the vanishing moment condition of $T$ and the fact
that $\lfloor\bz\rfloor\le \frac{\nu}{a_-}(\frac1{\widetilde{p}_-}-1)$, which implies $\lfloor\bz\rfloor\le s$,
it follows that, for any $k\in \nn,\,t\in (0,\fz)$ and $x\in (B_k^{(4)})^\com$,
\begin{align}\label{6e9}
\lf|\Phi_t*T(a_k)(x)\r|=&\frac1{t^{\nu}}\lf|\int_{\rn}\Phi\lf(\frac{x-y}{t^{\va}}\r)T(a_k)(y)\,dy\r|\\
\le &\frac1{t^{\nu}}\int_{\rn}\lf|\Phi\lf(\frac{x-y}{t^{\va}}\r)-\sum_{|\az|\le \lfloor\bz\rfloor}
\frac{\pa^{\az}\Phi(\frac{x-x_k}{t^{\va}})}{\az!}\lf(\frac{y-x_k}{t^{\va}}\r)^{\az}\r|\lf|T(a_k)(y)\r|\,dy\noz\\
=&\frac1{t^{\nu}}\lf(\int_{|y-x_k|_{\va}<2r_k}+\int_{2r_k\le|y-x_k|_{\va}<\frac{|x-x_k|_{\va}}{2}}+
\int_{|y-x_k|_{\va}\geq\frac{|x-x_k|_{\va}}{2}}\r)\noz\\
&\times \lf|\Phi\lf(\frac{x-y}{t^{\va}}\r)-\sum_{|\az|\le \lfloor\bz\rfloor}
\frac{\pa^{\az}\Phi(\frac{x-x_k}{t^{\va}})}{\az!}\lf(\frac{y-x_k}{t^{\va}}\r)^{\az}\r|\lf|T(a_k)(y)\r|\,dy\noz\\
=&:\textrm{II}_1+\textrm{II}_2+\textrm{II}_3,\noz
\end{align}
where $\Phi$ is as in \eqref{3e16}.

For $\textrm{II}_1$, by the Taylor remainder theorem and (vi), (iv) and (v)
of Lemma \ref{2l2}, we conclude that, for any $k\in\nn$, $N\in\nn$,
$t\in(0,\fz)$, $x\in (B_k^{(4)})^\com$ and $y\in\rn$ with $|y-x_k|_{\va}<2r_k$, there exists $\theta_1(y)\in B_k^{(2)}$ such that
\begin{align}\label{6e10}
\textrm{II}_1
\le&\frac1{t^{\nu}}\int_{|y-x_k|_{\va}<2r_k}\lf|\sum_{|\az|= \lfloor\bz\rfloor+1}
\pa^{\az}\Phi\lf(\frac{x-\theta_1(y)}{t^{\va}}\r)\r|
\lf|\frac{y-x_k}{t^{\va}}\r|^{\lfloor\bz\rfloor+1}\lf|T(a_k)(y)\r|\,dy\\
\ls& \frac1{t^{\nu}}\int_{|y-x_k|_{\va}<2r_k}\frac1{(1+|\frac{x-\theta_1(y)}{t^{\va}}|)^N}\lf|\frac{y-x_k}
{t^{\va}}\r|^{\lfloor\bz\rfloor+1}\lf|T(a_k)(y)\r|\,dy\noz\\
\ls& \frac1{t^{\nu}}\int_{|y-x_k|_{\va}<2r_k}\lf(\frac{t}{|x-x_k|_{\va}}\r)^{Na_-}\noz\\
&\times\max\lf\{\lf(\frac{|y-x_k|_{\va}}{t}\r)^{(\lfloor\bz\rfloor+1)a_-},\,
\lf(\frac{|y-x_k|_{\va}}{t}\r)^{(\lfloor\bz\rfloor+1)a_+}\r\}\lf|T(a_k)(y)\r|\,dy.\noz
\end{align}
When $t\le |x-x_k|_{\va}$, let
\begin{align*}
N:=\left\{
\begin{array}{cl}
\vspace{0.25cm}
&\lf\lfloor\dfrac{\nu+(\lfloor\bz\rfloor+1)a_-}{a_-}\r\rfloor+1
\hspace{0.5cm} {\rm when}\hspace{0.5cm} |y-x_k|_{\va}<t,\\
&\lf\lfloor\dfrac{\nu+(\lfloor\bz\rfloor+1)a_+}{a_-}\r\rfloor+1
\hspace{0.5cm}{\rm when}\hspace{0.5cm} |y-x_k|_{\va}\ge t
\end{array}\r.
\end{align*}
in \eqref{6e10}.
Then, by this, the H\"{o}lder inequality and the fact that $T$ is bounded on $L^2(\rn)$, we know that,
for any $k\in \nn,\,t\in (0,\fz)$ and $x\in (B_k^{(4)})^\com$,
\begin{align}\label{6e11}
\textrm{II}_1\ls& \int_{|y-x_k|_{\va}<2r_k}\max\lf\{\frac{|y-x_k|_{\va}^{(\lfloor\bz\rfloor+1)a_-}}{|x-x_k|_{\va}^{\nu+
(\lfloor\bz\rfloor+1)a_-}},\,\frac{|y-x_k|_{\va}^{(\lfloor\bz\rfloor+1)a_+}}
{|x-x_k|_{\va}^{\nu+(\lfloor\bz\rfloor+1)a_+}}\r\}\lf|T(a_k)(y)\r|\,dy\\
\ls&\max\lf\{\frac{r_k^{(\lfloor\bz\rfloor+1)a_-}}{|x-x_k|_{\va}^{\nu+
(\lfloor\bz\rfloor+1)a_-}},\, \frac{r_k^{(\lfloor\bz\rfloor+1)a_+}}{|x-x_k|_{\va}^{\nu+(\lfloor\bz\rfloor+1)a_+}}\r\}
\lf\|T(a_k)\r\|_{L^2(\rn)}\lf|B_k\r|^{1/2}\noz\\
\ls&\max\lf\{\lf(\frac{r_k}{|x-x_k|_{\va}}\r)^{\nu+
(\lfloor\bz\rfloor+1)a_-},\,\lf(\frac{r_k}{|x-x_k|_{\va}}\r)^{\nu+(\lfloor\bz\rfloor+1)a_+}\r\}
\frac1{\|\chi_{B_k}\|_{\lv}}.\noz
\end{align}
When $t>|x-x_k|_{\va}$, let $N:=\lfloor\frac{\nu+(\lfloor\bz\rfloor+1)a_-}{a_-}\rfloor$
in \eqref{6e10}. Then it is easy to see that \eqref{6e11} also holds true.

For $\textrm{II}_2$, by the Taylor remainder theorem, some arguments
similar to those used in the estimations of \eqref{6e10} and
\eqref{6e11}, the vanishing moment condition of $a_k$,
the fact that $\lfloor\bz\rfloor\le s$, \eqref{6e7}, the H\"{o}lder
inequality and Lemma \ref{2l2}(ix), we find
that, for any $z\in B_k$, there exists $\theta_2(z)\in B_k$ such that,
for any $t\in(0,\fz)$ and $x\in (B_k^{(4)})^\com$,
\begin{align}\label{6e12}
\textrm{II}_2
\ls& \int_{2r_k\le|y-x_k|_{\va}<\frac{|x-x_k|_{\va}}{2}}
\max\lf\{\frac{|y-x_k|_{\va}^{(\lfloor\bz\rfloor+1)a_-}}{|x-x_k|_{\va}^{\nu+
(\lfloor\bz\rfloor+1)a_-}},\,\frac{|y-x_k|_{\va}^{(\lfloor\bz\rfloor+1)a_+}}
{|x-x_k|_{\va}^{\nu+(\lfloor\bz\rfloor+1)a_+}}\r\}\lf|T(a_k)(y)\r|\,dy\\
\ls& \int_{2r_k\le|y-x_k|_{\va}<\frac{|x-x_k|_{\va}}{2}}
\max\lf\{\frac{|y-x_k|_{\va}^{(\lfloor\bz\rfloor+1)a_-}}{|x-x_k|_{\va}^{\nu+
(\lfloor\bz\rfloor+1)a_-}},\, \frac{|y-x_k|_{\va}^{(\lfloor\bz\rfloor+1)a_+}}{|x-x_k|_{\va}^{\nu+(\lfloor\bz\rfloor+1)a_+}}\r\}\noz\\
&\times\lf[\int_{B_k}|a_k(z)|\lf|\mathcal{K}(y,z)-\sum_{|\az|< \lfloor\bz\rfloor}
\frac{\pa^{\az}_y\mathcal{K}(y,x_k)}{\az!}(z-x_k)^{\az}\r|\,dz\r]\,dy\noz\\
\sim& \int_{2r_k\le|y-x_k|_{\va}<\frac{|x-x_k|_{\va}}{2}}
\max\lf\{\frac{|y-x_k|_{\va}^{(\lfloor\bz\rfloor+1)a_-}}{|x-x_k|_{\va}^{\nu+
(\lfloor\bz\rfloor+1)a_-}},\, \frac{|y-x_k|_{\va}^{(\lfloor\bz\rfloor+1)a_+}}{|x-x_k|_{\va}^{\nu+(\lfloor\bz\rfloor+1)a_+}}\r\}\noz\\
&\times\int_{B_k}|a_k(z)|\lf|\sum_{|\az|= \lfloor\bz\rfloor}
\frac{\pa^{\az}_y\mathcal{K}(y,x_k)-\pa^{\az}_y\mathcal{K}(y,\theta_2(z))}{\az!}(z-x_k)^{\az}\r|\,dz\,dy\noz\\
\ls& \int_{2r_k\le|y-x_k|_{\va}<\frac{|x-x_k|_{\va}}{2}}\max
\lf\{\frac{|y-x_k|_{\va}^{(\lfloor\bz\rfloor+1)a_-}}{|x-x_k|_{\va}^{\nu+
(\lfloor\bz\rfloor+1)a_-}},\, \frac{|y-x_k|_{\va}^{(\lfloor\bz\rfloor+1)a_+}}{|x-x_k|_{\va}^{\nu+(\lfloor\bz\rfloor+1)a_+}}\r\}\noz\\
&\times\int_{B_k}|a_k(z)|\frac{r_k^{\bz}}{|y-x_k|_{\va}^{\nu+\bz}}\,dz\,dy\noz\\
\ls& r_k^{\bz}\lf\|a_k\r\|_{L^2(\rn)}\lf|B_k\r|^{1/2}\noz\\ &\times\int_{2r_k\le|y-x_k|_{\va}<\frac{|x-x_k|_{\va}}{2}}
\max\lf\{\frac{|y-x_k|_{\va}^{-\nu-\bz+(\lfloor\bz\rfloor+1)a_-}}{|x-x_k|_{\va}^{\nu+
(\lfloor\bz\rfloor+1)a_-}},\,\frac{|y-x_k|_{\va}^{-\nu-\bz+(\lfloor\bz\rfloor+1)a_+}}
{|x-x_k|_{\va}^{\nu+(\lfloor\bz\rfloor+1)a_+}}\r\}\,dy\noz\\
\ls&\lf(\frac{r_k}{|x-x_k|_{\va}}\r)^{\nu+\bz}\frac1{\|\chi_{B_k}\|_{\lv}}.\noz
\end{align}

For $\textrm{II}_3$, from the Taylor remainder theorem, the vanishing moment condition of $a_k$,
the fact that $\lfloor\bz\rfloor\le s$, \eqref{6e7}, the H\"{o}lder inequality and Lemma \ref{2l2}(vi), we deduce
that, for any $z\in B_k$, there exists $\theta_3(z)\in B_k$ such that, for any $t\in(0,\fz)$ and $x\in (B_k^{(4)})^\com$,
\begin{align}\label{6e13}
\textrm{II}_3\le&\frac1{t^{\nu}}\int_{|y-x_k|_{\va}\geq\frac{|x-x_k|_{\va}}{2}}
\lf|\Phi\lf(\frac{x-y}{t^{\va}}\r)-\sum_{|\widetilde{\az}|\le \lfloor\bz\rfloor}
\frac{\pa^{\widetilde{\az}}\Phi(\frac{x-x_k}{t^{\va}})}{\widetilde{\az}!}\lf(\frac{y-x_k}{t^{\va}}\r)^{\widetilde{\az}}\r|\\
&\times\lf[\int_{B_k}|a_k(z)|\lf|\mathcal{K}(y,z)-\sum_{|\az|< \lfloor\bz\rfloor}
\frac{\pa^{\az}_y\mathcal{K}(y,x_k)}{\az!}(z-x_k)^{\az}\r|\,dz\r]\,dy\noz\\
\sim& \int_{|y-x_k|_{\va}\geq\frac{|x-x_k|_{\va}}{2}}\lf|\frac1{t^{\nu}}\lf[\Phi\lf(\frac{x-y}
{t^{\va}}\r)-\sum_{|\widetilde{\az}|\le \lfloor\bz\rfloor}
\frac{\pa^{\widetilde{\az}}\Phi(\frac{x-x_k}{t^{\va}})}{\widetilde{\az}!}
\lf(\frac{y-x_k}{t^{\va}}\r)^{\widetilde{\az}}\r]\r|\noz\\
&\times\int_{B_k}|a_k(z)|\lf|\sum_{|\az|= \lfloor\bz\rfloor}
\frac{\pa^{\az}_y\mathcal{K}(y,x_k)-\pa^{\az}_y\mathcal{K}(y,\theta_3(z))}{\az!}(z-x_k)^{\az}\r|\,dz\,dy\noz\\
\ls& \int_{|y-x_k|_{\va}\geq\frac{|x-x_k|_{\va}}{2}}\lf|\Phi_t(x-y)\r|\int_{B_k}
|a_k(z)|\frac{r_k^{\bz}}{|y-x_k|_{\va}^{\nu+\bz}}\,dz\,dy\noz\\
&+ \int_{|y-x_k|_{\widetilde{\va}}\geq\frac{|x-x_k|_{\va}}{2}}\lf|\frac1{t^{\nu}}\sum_{|\widetilde{\az}|\le \lfloor\bz\rfloor}
\frac{\pa^{\widetilde{\az}}\Phi(\frac{x-x_k}{t^{\va}})}{\widetilde{\az}!}
\lf(\frac{y-x_k}{t^{\va}}\r)^{\widetilde{\az}}\r|\int_{B_k}
|a_k(z)|\frac{r_k^{\bz}}{|y-x_k|_{\va}^{\nu+\bz}}\,dz\,dy\noz\\
\ls& \frac{r_k^{\bz}}{|x-x_k|_{\va}^{\nu+\bz}}\|a_k\|_{L^2(\rn)}|B_k|^{1/2}\int_{|y-x_k|_{\va}\geq\frac{|x-x_k|_{\va}}{2}}
\lf|\Phi_t(x-y)\r|\,dy\noz\\
&+r_k^{\bz}\|a_k\|_{L^2(\rn)}|B_k|^{1/2}\int_{|y-x_k|_{\va}\geq\frac{|x-x_k|_{\va}}{2}}
\frac1{t^{\nu}}\sum_{|\widetilde{\az}|\le \lfloor\bz\rfloor}\lf(\frac{t}{|x-x_k|_{\va}}\r)
^{Na_-}\lf|\frac{y-x_k}{t^{\va}}\r|^{|\widetilde{\az}|}\,dy\noz\\
=&:\textrm{II}_{3,1}+\textrm{II}_{3,2}.\noz
\end{align}

For $\textrm{II}_{3,1}$, from the size condition of $a_k$ and the fact that $\Phi$ is as in \eqref{3e16},
it follows that, for any $x\in (B_k^{(4)})^\com$,
\begin{align}\label{6e14'}
\textrm{II}_{3,1}\ls\lf(\frac{r_k}{|x-x_k|_{\va}}\r)^{\nu+\bz}\frac1{\|\chi_{B_k}\|_{\lv}}.
\end{align}
In addition, by an argument similar to that used in the estimation of \eqref{6e11} and Lemma \ref{2l2}(ix),
we find that, for any $x\in (B_k^{(4)})^\com$,
$$\textrm{II}_{3,2}\ls\lf(\frac{r_k}{|x-x_k|_{\va}}\r)^{\nu+\bz}\frac1{\|\chi_{B_k}\|_{\lv}},$$
which, combined with \eqref{6e13} and \eqref{6e14'}, further implies that, for any $x\in (B_k^{(4)})^\com$,
\begin{align}\label{6e14}
\textrm{II}_{3}\ls\lf(\frac{r_k}{|x-x_k|_{\va}}\r)^{\nu+\bz}\frac1{\|\chi_{B_k}\|_{\lv}}.
\end{align}

Combining \eqref{6e9}, \eqref{6e11}, \eqref{6e12} and \eqref{6e14}, we conclude that, for any $x\in (B_k^{(4)})^\com$,
\begin{align*}
M_0(T(a_k))(x)
&=\sup_{t\in(0,\fz)}\lf|\Phi_t*T(a_k)(x)\r|
\ls\lf(\frac{r_k}{|x-x_k|_{\va}}\r)^{\nu+\bz}\frac1{\|\chi_{B_k}\|_{\lv}}\\
&\ls \lf[M_{\rm HL}\lf(\chi_{B_k}\r)(x)\r]^{\frac{\nu+\bz}{\nu}}\frac1{\|\chi_{B_k}\|_{\lv}},
\end{align*}
which implies that
$$M_0(T(a_k))(x)\chi_{(B_k^{(4)})^\com}(x)\ls \lf[M_{\rm HL}\lf(\chi_{B_k}\r)(x)\r]^{\frac{\nu+\bz}{\nu}}\frac1{\|\chi_{B_k}\|_{\lv}}.$$
Then, by the fact that $\widetilde{p}_-<\frac{\nu+\bz}{\nu}$ and an argument similar to that used in the proof of Theorem \ref{6t2}, we know that
\eqref{6e15} holds true. This finishes the proof of Theorem \ref{6t4}.
\end{proof}

Now we prove Theorem \ref{6t5}.

\begin{proof}[Proof of Theorem \ref{6t5}]
Let $\vp\in(0,1]^n$ and $s$ be as in \eqref{3e1}. By an argument similar to that
used in the proof of Theorem \ref{6t3}, we know that, to show Theorem \ref{6t5},
it suffices to prove that, for any $(\vp,2,s)$-atom $a$ and $x\in\rn$,
\begin{align}\label{6e1'}
\lf|T(a)(x)\chi_{B^{(2)}}(x)\r|\ls \lf[M_{\rm HL}(\chi_{B})(x)\r]
^{\f{\nu+\bz}{\nu}}\frac1{\|\chi_{B}\|_{\lv}},
\end{align}
where $B$ and $B^{(2)}$ are as in the proof of Theorem \ref{6t3}.

Indeed, let $x_0$ and $r$ denote the center and the radius of $B$, respectively.
From the Taylor remainder theorem, the vanishing moment condition of $a$,
the fact that $\lfloor\bz\rfloor\le \frac{\nu}{a_-}(\frac1{\widetilde{p}_-}-1)$,
which implies $\lfloor\bz\rfloor\le s$, and the H\"{o}lder inequality,
we deduce that, for any $z\in B$,
there exists $\theta(z)\in B$ such that, for any $x\in B^{(2)}$,
\begin{align*}
\lf|T(a)(x)\r|&\leq\int_B |a(z)|\lf|\mathcal{K}(x,z)\r|\,dz\\
&=\int_{B}|a(z)|\lf|\mathcal{K}(x,z)-\sum_{|\az|< \lfloor\bz\rfloor}
\frac{\pa^{\az}_x\mathcal{K}(x,x_0)}{\az!}(z-x_0)^{\az}\r|\,dz\\
&\sim \int_{B}|a(z)|\lf|\sum_{|\az|= \lfloor\bz\rfloor}
\frac{\pa^{\az}_x\mathcal{K}(x,x_0)-\pa^{\az}_x\mathcal{K}(x,\theta(z))}{\az!}(z-x_0)^{\az}\r|\,dz\\
&\ls \int_{B}|a(z)|\frac{r^{\bz}}{|x-x_0|_{\va}^{\nu+\bz}}\,dz
\ls \frac{r^{\bz}}{|x-x_0|_{\va}^{\nu+\bz}}\|a\|_{L^2(\rn)}|B|^{1/2}\\
&\ls \lf(\frac{r}{|x-x_0|_{\va}}\r)^{\nu+\bz}\frac1{\|\chi_{B}\|_{\lv}}
\ls \lf[M_{\rm HL}(\chi_{B})(x)\r]^{\f{\nu+\bz}{\nu}}\frac1{\|\chi_{B}\|_{\lv}},
\end{align*}
which implies that \eqref{6e1'} holds true and hence completes the proof of Theorem \ref{6t5}.
\end{proof}

\bigskip

\noindent  Long Huang, Jun Liu, Dachun Yang  and Wen Yuan (Corresponding author)

\medskip

\noindent  Laboratory of Mathematics and Complex Systems
(Ministry of Education of China),
School of Mathematical Sciences, Beijing Normal University,
Beijing 100875, People's Republic of China

\smallskip

\noindent {\it E-mails}: \texttt{longhuang@mail.bnu.edu.cn} (L. Huang)

\hspace{0.99cm}\texttt{junliu@mail.bnu.edu.cn} (J. Liu)

\hspace{0.99cm}\texttt{dcyang@bnu.edu.cn} (D. Yang)

\hspace{0.99cm}\texttt{wenyuan@bnu.edu.cn} (W. Yuan)

\end{document}